\documentclass[a4paper, 12pt]{article}

\usepackage{amsmath} 
\usepackage{theorem}
\usepackage{amssymb}
\usepackage{amsfonts}
\usepackage{latexsym}
\usepackage{amscd}
\usepackage{diagramb}

\usepackage{graphicx}
\usepackage{enumerate}
\usepackage{cancel}
\usepackage{extarrows}
\usepackage{upgreek}

\usepackage{pxfonts}

\usepackage{stmaryrd}

\input xy
\xyoption{all}

\DeclareMathAlphabet{\mathpzc}{OT1}{pzc}{m}{it}

\usepackage{xspace,colortbl}
\usepackage{color}

\definecolor{verde}{rgb}{0.,0.7,0.}
\definecolor{indigo}{rgb}{.18, .34, .78}
\definecolor{indigo1}{rgb}{.18, .24, .78}
\definecolor{indigo2}{rgb}{.18, .14, .78}
\definecolor{indigo3}{rgb}{.18, 0., .78}
\definecolor{rojo}{rgb}{1,0,0}
\definecolor{negro}{rgb}{0,0,0}
\definecolor{lila}{rgb}{.46, .16, .78}
\definecolor{lila1}{rgb}{.46, .16, .86}
\definecolor{lila2}{rgb}{.56, .16, .86}
	\definecolor{lila3}{rgb}{.63, .16, .78}
\definecolor{lila4}{rgb}{.7, .16, .78}
\definecolor{lila5}{rgb}{.78, .26, .78}
\definecolor{lila6}{rgb}{.6, 0., .78}

\theoremstyle{plain}
\theoremheaderfont{\normalfont\bfseries}

\newtheorem{thm}{Theorem}[section]

\newtheorem{lma}[thm]{Lemma}
\newtheorem{cor}[thm]{Corollary}
\newtheorem{defn}[thm]{Definition}

\theorembodyfont{\rmfamily}

\newtheorem{rem}[thm]{Remark}
\newtheorem{prop}[thm]{Proposition}

\newcommand{\x}{\mbox{-}}

\def\r{{\mathfrak r}}

\newcommand{\s}{\hspace{0,06cm}}

\newcommand{\Ob}{{\mathcal Ob}}

\newcommand{\DblPs}{{\rm DblPs}}
\newcommand{\Fun}{{\rm Fun}}

\newcommand{\Aa}{{\mathbb A}}
\newcommand{\Bb}{{\mathbb B}}
\newcommand{\Cc}{{\mathbb C}}
\newcommand{\Dd}{{\mathbb D}}

\newcommand{\comp}{\circ}
\newcommand{\iso}{\cong}
\newcommand{\ot}{\otimes}

\newcommand{\C}{{\mathcal C}}
\newcommand{\Tau}{{\mathcal T}}
\newcommand{\M}{{\mathcal M}}
\newcommand{\D}{{\mathcal D}}

\newcommand{\A}{{\mathcal A}}
\newcommand{\B}{{\mathcal B}}

\newcommand{\Ll}{{\mathcal L}}

\newcommand{\crta}{\overline}

\newcommand{\Id}{\operatorname {Id}}
\newcommand{\id}{\operatorname {id}}

\newcommand{\Epsilon}{\varepsilon}

\def\Dd{{\mathbb D}}

\def\u#1{\underline{#1}}

\newcommand{\im}{{\rm Im}\,}

\newcommand{\cref}[1]{C.~\ref{c:#1}}

\newcommand{\lelabel}[1]{\label{le:#1}}
\newcommand{\leref}[1]{Lemma~\ref{le:#1}}
\newcommand{\eqlabel}[1]{\label{eq:#1}}
\newcommand{\equref}[1]{(\ref{eq:#1})}

\newcommand{\delabel}[1]{\label{de:#1}}
\newcommand{\deref}[1]{Definition~\ref{de:#1}}
\newcommand{\prlabel}[1]{\label{pr:#1}}
\newcommand{\prref}[1]{Proposition~\ref{pr:#1}}
\newcommand{\colabel}[1]{\label{co:#1}}
\newcommand{\coref}[1]{Corollary~\ref{co:#1}}

\newcommand{\rmlabel}[1]{\label{rm:#1}}
\newcommand{\rmref}[1]{Remark~\ref{rm:#1}}
\newcommand{\selabel}[1]{\label{se:#1}}
\newcommand{\seref}[1]{Section~\ref{se:#1}}
\newcommand{\sslabel}[1]{\label{ss:#1}}
\newcommand{\ssref}[1]{Subsection~\ref{ss:#1}}

\newcommand\tripplearrow{%
        \mathrel{\vcenter{\mathsurround0pt
                \ialign{##\crcr
                        \noalign{\nointerlineskip}$\longrightarrow$\crcr
                        \noalign{\nointerlineskip}$\longleftarrow$\crcr
                        \noalign{\nointerlineskip}$\longrightarrow$\crcr
                }%
        }}%
}

\newcommand\triarrows{%
        \mathrel{\vcenter{\mathsurround0pt
                \ialign{##\crcr
                        \noalign{\nointerlineskip}$\longrightarrow$\crcr
                        \noalign{\nointerlineskip}$\longrightarrow$\crcr
                        \noalign{\nointerlineskip}$\longrightarrow$\crcr
                }%
        }}%
}

\newcommand*{\threefrac}[3]{%
  \begin{array}{@{\,}c@{\,}}%
    #1\\
    \hline
    #2\\
    \hline
    #3%
  \end{array}%
}

\newcommand*{\fourfrac}[4]{%
 \begin{array}{@{\,}c@{\,}c@{\,}}%
    #1\\
    \hline
    #2\\
    \hline
    #3\\
    \hline
		#4%
  \end{array}%
	}

\newcommand*{\fivefrac}[5]{%
 \begin{array}{@{\,}c@{\,}c@{\,}c@{\,}}%
    #1\\
    \hline
    #2\\
    \hline
    #3\\
    \hline
		#4\\
    \hline
		#5%
  \end{array}%
	}

\newdir{:=}{{}}
\newcommand{\fit}[3]{\ar@{:=}@/{#3}/[#1] |{\Downarrow #2} }

\begin{document}

\title{Alternative notion to intercategories: part I. \\
A tricategory of double categories
}
\author{Bojana Femi\'c \vspace{6pt} \\
{\small Mathematical Institite of  \vspace{-2pt}}\\
{\small Serbian Academy of Sciences and Arts } \vspace{-2pt}\\
{\small  Knez Mihailova 35,} \vspace{-2pt}\\
{\small  11 000 Belgrade, Serbia}}

\date{}
\maketitle


\begin{center}
{\bf Dedicated to Gabi B\"ohm}
\end{center}

\smallskip

\begin{abstract}
This is a first of a series of two papers. Our motive is to tackle the question raised in B\"ohm's 
``The Gray Monoidal Product of Double Categories'' from Applied Categorical Structures: which would be 
an alternative notion to intercategories of Grandis and Par\'e, so that monoids in B\"ohm's monoidal category $Dbl$ of 
strict double categories and double pseudo functors be an example of it? Before addressing this question we 
observe that although bicategories embed into pseudo double categories, this embedding is not monoidal, with the usual notion of a 
monoidal pseudo double category. We then prove that monoidal bicategories embed into the mentioned monoids of B\"ohm. 
%
In order to fit B\"ohm's monoid into an intercategory-type object, we start by upgrading the category $Dbl$ to a 2-category 
and end up rather with a tricategory $\DblPs$. We propose an alternative definition of intercategories as internal 
categories in this tricategory, enabling $Dbl$ to be an example of this gadget. The formal definition of a category internal to 
the (type of a) tricategory (of) $\DblPs$, as well as another important example of these in the literature, we leave for a subsequent paper. 
\end{abstract}


{\em AMS Classification codes}: 18D40, 18N10, 18N20.

\section{Introduction}

It is well-known that 2-categories embed in strict double categories and that bicategories embed in pseudo double categories. 
However, it is not clear which of the definitions of a monoidal pseudo double category existent in the literature 
would be suitable to have the analogous result, that monoidal bicategories embed to monoidal pseudo double categories. 
This question we resolve in \ssref{mon embed}. Namely, seeing a monoidal bicategory as a one object tricategory, 
we consider the equivalent one object Gray 3-category (by the coherence of tricategories of \cite{GPS}), which is nothing but a monoid in the 
monoidal category $Gray$, {\em i.e.} a Gray monoid (see \cite{DS}, \cite[Lemma 4]{BN}). 
We then prove that $Gray$ embeds as a monoidal category in the monoidal category $Dbl$ from \cite{Gabi} of strict double categories and 
double functors which are given by isomorphisms in both directions. The latter functors are later introduced in \cite{Shul1} under the name 
{\em double pseudo functors}. 
For the above-mentioned embedding we give an explicit description of the monoidal structure of $Dbl$
(in \cite{Gabi} only an explicit description of the structure of a monoid is given). Analogously as in \cite{Gray}, 
we introduce a {\em cubical double functor} along the way. We also show why the other notions of monoidal double categories 
(monoids in the category of strict double categories and strict double functors from \cite{BMM},  
and pseudomonoids in the 2-category $PsDbl$ of pseudo double categories, pseudo double functors and vertical transformations from 
\cite{Shul}) do not obey the embedding in question. 


Then we turn to the following question of B\"ohm: if a monoid in her monoidal category $Dbl$ could fit some framework similar to 
intercategories of Grandis and Par\'e. Observe that neither of the two notions would be more general than the other. Namely, 
intercategories are categories internal in the 2-category $LxDbl$ of pseudo double categories, lax double functors and horizontal 
transformations, whereas in the structure of B\"ohm's monoid in $Dbl$ the relevant objects are {\em strict}  double categories and 
the functors are {\em double pseudo functors} in the sense of \cite{Shul1}. By Strictification Theorem of \cite[Section 7.5]{GP:Limits} 
every pseudodouble category is equivalent to a strict double category, thus on the level of objects in the ambient category basically 
nothing is changed. Though, going from lax double functors to double pseudo functors, one tightens in one direction and weakens in the other. 

In the search for a desired framework, we define 2-cells among 
double pseudo functors and we get that instead of a 2-category, we indeed have a tricategory structure, including modifications 
as 3-cells. This led us to propose an alternative notion for intercategories, as categories internal in this tricategory of strict double categories, which we denote by $\DblPs$. B\"ohm's monoids are example of this new gadget, as well as most of the examples of intercategories treated in \cite{GP:Fram}, except from duoidal categories, which rely on lax functors, rather than pseudo ones. 
Internal categories in this kind of tricategories we plan to develop in a subsequent paper. 

Contrarily to $LxDbl$, the 1-cells of the 2-category $PsDbl$ 
are particular cases of double pseudo functors. Having in mind the above Strictification Theorem and adding only the trivial 
3-cells to $PsDbl$, we also prove that thus obtained tricategory $PsDbl^*_3$ embeds in our tricategory $\DblPs$. 
As a byproduct to this proof we obtain a more general result: supposing that there is a connection (\cite{BS}) on 1v-components of 
strong vertical transformations (\cite[Section 7.4]{GP:Limits}), there is a bijection between 
strong vertical transformations and those strong horizontal transformations whose 1h-cell components are 1h-companions of some 1v-cells.
This we prove in \coref{bij}.

\bigskip 

The composition of the paper is the following. In Section 2 we recall the definition of a tricategory and give a detailed description of 
it; in \seref{mon} we 
give a description of the monoidal structure of $Dbl$ of B\"ohm, we define cubical double functors and prove that $Gray$ embeds 
into $Dbl$. Section 4 is dedicated to the construction of our tricategory $\DblPs$ (it occupies more than a half of the paper). 
In the last Section we prove that the tricategory $PsDbl^*_3$ embeds into $\DblPs$ and prove the bijection between vertical and horizontal 
strong transformation, supposing the mentioned connection. 

\bigskip

\section{Tricategories} \selabel{tricat}

Because of the complexity of its structure for a broader audience we unpack step-by-step the definition of a tricategory introduced in 
\cite{GPS}. The readers familiar with tricategories can skip this whole section. It has been proved in {\em. loc.cit.} 
that a tricategory is equivalent to a Gray-category, 
which is a semistrict version of construction between a 3-category and a tricategory. Gray-categories in turn were introduced in \cite{Gray} 
as categories enriched over the monoidal category {\bf Gray} of 2-categories and 2-functors with a particular tensor product 
constructed as an adjoint functor to the inner hom-functor $2\x\Fun(\A,-)$ for 2-categories, where $\A$ is a 2-category. This construction 
and the necessity of it are nicely explained for example in \cite[Section 5.1]{GP:Fram}. We summarize now the definition of a tricategory from \cite{GPS} 
with slight changes (in the direction of $r$ in (TD6) and accordingly in $\mu$ and $\rho$ in (TD8) ). 

A tricategory $\Tau$ consists of the following data (TDi) and axioms (TAj): 

(TD1) a set $\Ob\Tau$ of objects of $\Tau$; 

(TD2) a bicategory $\Tau(p,q)$ for objects $p,q$ of $\Tau$; 

(TD3) a pseudofunctor $\ot: \Tau(q,r)\times\Tau(p,q)\to\Tau(p,r)$, for $p,q,r\in\Ob\Tau$, called composition; 

(TD4) a pseudofunctor $I_p: 1\to\Tau(p,p)$, for $p\in\Ob\Tau$, where $1$ is the unit bicategory; 

(TD5) a pseudo natural equivalence $a: \ot\comp(\ot\times 1) \Rightarrow \ot\comp(1\times\ot)$, where the respective pseudofunctors 
act between bicategories $\Tau(r,s)\times\Tau(q,r)\times\Tau(p,q)\to\Tau(p,s)$, for $p,q,r,s\in\Ob\Tau$; 

(TD6) pseudo natural equivalences $l: \ot\comp(I_q\times\Id_{\Tau(p,q)})\to\Id_{\Tau(p,q)}$ and $r: \ot\comp(\Id_{\Tau(p,q)}\times I_p) \to \Id_{\Tau(p,q)}$ 
for objects $p,q$ in $\Tau$; 

(TD7) an invertible modification $\pi$ up to which the pentagon for $a$ commutes; 

(TD8) invertible modifications $\mu, \lambda$ and $\rho$ 
relating $a$ with $l$ and $r$, then $a$ with $l$ and $a$ with $r$, respectively;

(TA1) non abelian 4-cocycle condition for $\pi$; 

(TA2) left normalization for the 4-cocycle $\pi$, and 

(TA3) right normalization for the 4-cocycle $\pi$. 

A 3-category consists of the data (TD1) - (TD6), where $a,l$ and $r$ from (TD5) and (TD6) are identities, and 
with the following changes in the items (TD2) - (TD4): $\Tau(p,q)$'s are 2-categories 
and $\ot$ and $I_p$'s are 2-functors ($a, l$ and $r$ are identities). 

\medskip

We unpack now the conditions (TD2) -- (TD8). Observe that in (TD5-TD8) we give a description more detailed than in 
\cite[Section 4.4]{Gur-book}.


\u{(TD2):} That $\Ll:=\Tau(p,q)$ is a bicategory for every $p,q\in\Ob\Tau$ comprises the following items: 
\begin{enumerate} [1)]
\item 1-cells $A,B...$ acting from $p$ to $q$ (which we will write as horizontal simple arrows, straight or arched $p\stackrel{A}{\to}q$, 
$\xymatrix@C=8pt{ p \ar@/^1pc/[rr]^{B}\ar@{}[rr]| {}  &  &  q} $); 
\item 2-cells $F,G...$ (we will denote them by a double arrow in the vertical direction: 
$
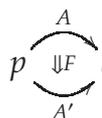
\xymatrix@C=8pt{ p \ar@/^1pc/[rr]^{A}\ar@{}[rr]| {\Downarrow F} \ar@/_1pc/[rr]_{ A'} &  & q }$, or in the form of a rectangular diagram 
whose vertical arrows are identities), 
3-cells $\alpha, \beta...$ (we will think them in the direction perpendicular to the plane of the paper, in the transversal direction, 
and will denote them by a triple arrow ...);

a strictly associative {\em transversal composition} of 3-cells denoted by $\cdot$, 
and an identity 3-cell $\Id_F$ (which is the strict unit for $\cdot$);
\item {\em vertical composition} of 2- and 3-cells denoted by $\odot$ 
such that $(a)$ $\Id_G\odot\Id_F=\Id_{G\odot F}$ and $(b)$ $(\beta'\cdot\beta)\odot(\alpha'\cdot\alpha)=(\beta'\odot\alpha')\cdot
(\beta\odot\alpha)$; 
\item for each 1-cell $A$ of $\Tau$ the identity 2-cell $\id_A$; 
\item associativity isomorphism 3-cell $\upalpha: (H\odot G)\odot F \Rrightarrow H\odot(G\odot F)$ for the vertical composition of 2-cells,
natural in them; 
\item left and right unity isomorphism 3-cells $\lambda_F: \id_B\odot F\Rrightarrow F$ and $\rho_F: F\odot\id_A\Rrightarrow F$, natural in any 2-cell 
$A\stackrel{F}{\Rightarrow} B$; 
\item the pentagon constraint for $\upalpha$ and triangle constraint for $\upalpha\x\lambda\x\rho$ commute;. 
\end{enumerate}
* (by the coherence Theorem 1.5 of \cite{GPS} the items 5)-7) above can be ignored); 
\bigskip


\u{(TD3):} That $\ot: \Tau(q,r)\times\Tau(p,q)\to\Tau(p,r)$ is a pseudofunctor, for every $p,q,r\in\Ob\Tau$, it comprises the following items: 
\begin{enumerate} [1)]
\item for 1-cells $p \stackrel{A}{\to} q \stackrel{B}{\to} r$ there is a composition 1-cell $p \stackrel{B\ot A}{\to} r$; 
\item {\em horizontal composition} of 2- and 3-cells denoted by $\ot$ such that $(a)$ $\Id_G\ot\Id_F=\Id_{G\ot F}$ and 
$(b)$ $(\beta'\odot\beta)\ot(\alpha'\odot\alpha)=(\beta'\ot\alpha')\odot(\beta\ot\alpha)$ (concretely, 
for two 2-cells $F,G$ between two composable pairs of 1-cells 
$
\xymatrix@C=8pt{ p \ar@/^1pc/[rr]^{A}\ar@{}[rr]| {\Downarrow F} \ar@/_1pc/[rr]_{ A'} &  & q 
\ar@/^1pc/[rr]^{B}\ar@{}[rr]|{\Downarrow G} \ar@/_1pc/[rr]_{ B'} &  & r} 
$
there is a horizontal composition 2-cell 
$\xymatrix{
p 
{\ar @{} [rr] |{\Downarrow F\ot G}} 
\ar@/^1.3pc/[rr]^{A\ot B} 
\ar@/_1.3pc/[rr]_{A'\ot B'}
&& r
}$ and for two 3-cells $\alpha, \beta$ between two horizontally composable pairs of 2-cells 
$
\xymatrix@C=8pt{ p \ar@/^1pc/[rr]^{A}\ar@{}[rr]| {\Downarrow F} \ar@/_1pc/[rr]_{ A'} &  & q 
\ar@/^1pc/[rr]^{B}\ar@{}[rr]|{\Downarrow G} \ar@/_1pc/[rr]_{ B'} &  & r} 
$ and 
$
\xymatrix@C=8pt{ p \ar@/^1pc/[rr]^{A}\ar@{}[rr]| {\Downarrow F'} \ar@/_1pc/[rr]_{ A'} &  & q 
\ar@/^1pc/[rr]^{B}\ar@{}[rr]|{\Downarrow G'} \ar@/_1pc/[rr]_{ B'} &  & r} 
$
there is a horizontal composition 3-cell $\beta\ot\alpha: G\ot F\Rrightarrow G'\ot F\s'$); 
\item for 2-cells in $\Tau$ composable pairwise horizontally and vertically: 
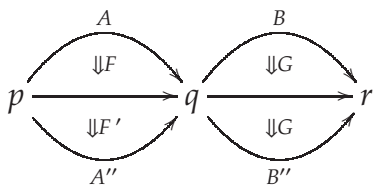
\[
\xymatrix{ 
p \ar@{:=}@/^1pc/[rr] |{\Downarrow F} \ar@/^{2pc}/[rr]^{A}  \ar[rr] 
  \ar@{:=}@/_1pc/[rr] |{\Downarrow F\s'}  \ar@/_{2pc}/[rr]_{A''}  &&
	q \ar@{:=}@/^1pc/[rr] |{\Downarrow G} \ar@/^{2pc}/[rr]^{B}  \ar[rr] 
  \ar@{:=}@/_1pc/[rr] |{\Downarrow G}  \ar@/_{2pc}/[rr]_{B''} && r
}\]
there are natural isomorphism 3-cells 
$$\xi: (G'\ot F\s') \odot (G\ot F) \Rrightarrow (G'\odot G)\ot(F\s'\odot F)$$ 
$$\xi_0: \id_{B\ot A} \Rrightarrow \id_B \ot \id_A, $$
so that the corresponding hexagonal constraint for $\xi$, the square for $\xi\x\xi_0\x\lambda$ and the square for $\xi\x\xi_0\x\rho$ 
commute (these will be ignored by the coherence Theorem 1.6 of \cite{GPS}); 
\end{enumerate}


(TD4) states that for each 0-cell $p$ in $\Tau$ there is a 1-cell $I_p$ and a 2-cell $\iota_P:I_p\to I_p$ so that there is an isomorphism 
3-cell $Id_{I_p}\iso \iota_p$;  

(TD5) states that we have associativity of $\ot$ on the levels of 1-, 2- and 3-cells; 

\noindent for three horizontally composable 1-cells $C,B,A$ there are 2-cells $a_{C,B,A}: (C\ot B)\ot A\Rightarrow C\ot(B\ot A)$ and its quasi-inverse $a'_{C',B',A'}$ 
so that there are invertible 3-cells $a_{C,B,A}\odot a_{C,B,A}'\Rrightarrow\Id_{(C\ot B)\ot A}$ and $\Id_{C\ot(B\ot A)}\Rrightarrow a_{C,B,A}'\odot a_{C,B,A}$; 

\noindent for each pair of triples of composable 1-cells $C,B,A$ and $C',B',A'$ and three 2-cells acting between them $H,G,F$, there are invertible 
3-cells (natural in $H,G,F$): 
$$a_{C',B',A'}\odot\big((H\ot G)\ot F\big) \stackrel{a_{H,G,F}}{\Rrightarrow} \big(H\ot (G\ot F)\big) \odot a_{C,B,A}$$
$$a_{C',B',A'}'\odot \big(H\ot (G\ot F)\big) \stackrel{a'_{H,G,F}}{\Rrightarrow} \big((H\ot G)\ot F\big)\odot a_{C,B,A}';$$

\noindent 
for three horizontally composable 3-cells $\gamma, \beta, \alpha$ the following diagram of (the transversal composition of) 3-cells commutes:
$$
\bfig
 \putmorphism(-150,400)(1,0)[p`q `A]{360}1a
\putmorphism(380,400)(1,0)[q`r`B]{360}1a
 \putmorphism(740,400)(1,0)[`s `C]{360}1a

\putmorphism(-150,420)(0,-1)[\phantom{Y_2}`p `=]{330}1l
\putmorphism(1090,420)(0,-1)[\phantom{Y_2}`s`=]{330}1l

 \putmorphism(-130,100)(1,0)[`q`A]{360}1b
 \putmorphism(230,100)(1,0)[`r`B]{360}1b
 \putmorphism(730,100)(1,0)[r``C]{360}1b
\put(350,230){\fbox{$a_{C,B,A}$}}

\putmorphism(-150,100)(0,-1)[\phantom{Y_2}``=]{380}1l
\putmorphism(230,100)(0,-1)[q` `=]{380}1l
\putmorphism(580,100)(0,-1)[\phantom{Y_2}``=]{380}1l
\putmorphism(720,100)(0,-1)[\phantom{Y_2}``=]{380}1l
\putmorphism(1090,100)(0,-1)[\phantom{Y_2}``=]{380}1l

\putmorphism(-150,-260)(1,0)[p`q `A']{360}1b
\putmorphism(230,-260)(1,0)[`r `B']{360}1b
 \putmorphism(730,-260)(1,0)[r`s`C']{360}1b

\put(-40,-120){\fbox{$F$}}
\put(310,-120){\fbox{$G$}}
\put(820,-120){\fbox{$H$}}

\putmorphism(100,-570)(1,0)[` `\Downarrow]{360}0a 
\putmorphism(550,-570)(1,0)[` `\big(\gamma\ot(\beta\ot\alpha)\big)\odot\Id]{360}0a

\efig
\quad\quad\stackrel{a_{H,G,F}}{\Rrightarrow}\quad
\bfig
 \putmorphism(-150,400)(1,0)[p`q `A]{360}1a
\putmorphism(380,400)(1,0)[q`r`B]{360}1a
 \putmorphism(740,400)(1,0)[`s `C]{360}1a

\putmorphism(-150,420)(0,-1)[\phantom{Y_2}`p `=]{380}1l
\putmorphism(210,420)(0,-1)[` `=]{380}1l
\putmorphism(370,420)(0,-1)[\phantom{Y_2}``=]{380}1l
\putmorphism(720,420)(0,-1)[\phantom{Y_2}``=]{380}1l
\putmorphism(1090,420)(0,-1)[\phantom{Y_2}`s`=]{380}1l

 \putmorphism(-130,50)(1,0)[`q`A']{360}1a
 \putmorphism(380,50)(1,0)[q``B']{360}1a
 \putmorphism(730,50)(1,0)[r``C']{360}1a
\put(350,-120){\fbox{$a_{C',B',A'}$}}

\putmorphism(-150,100)(0,-1)[\phantom{Y_2}``=]{380}1l
\putmorphism(1090,100)(0,-1)[\phantom{Y_2}``=]{380}1l

\putmorphism(-150,-260)(1,0)[p`q `A']{360}1b
\putmorphism(230,-260)(1,0)[`r `B']{360}1b
 \putmorphism(730,-260)(1,0)[r`s`C']{360}1b

\put(-40,210){\fbox{$F$}}
\put(450,210){\fbox{$G$}}
\put(820,210){\fbox{$H$}}

\putmorphism(-50,-570)(1,0)[` `\Downarrow]{360}0a 
\putmorphism(400,-570)(1,0)[` `\Id\odot\big((\gamma\ot\beta)\ot\alpha\big)]{360}0a

\efig
$$
\vspace{-0,1cm}
$$
\bfig
 \putmorphism(-150,500)(1,0)[p`q `A]{360}1a
\putmorphism(380,500)(1,0)[q`r`B]{360}1a
 \putmorphism(740,500)(1,0)[`s `C]{360}1a

\putmorphism(-150,520)(0,-1)[\phantom{Y_2}`p `=]{330}1l
\putmorphism(1090,520)(0,-1)[\phantom{Y_2}`s`=]{330}1l

 \putmorphism(-130,200)(1,0)[`q`A]{360}1b
 \putmorphism(230,200)(1,0)[`r`B]{360}1b
 \putmorphism(730,200)(1,0)[r``C]{360}1b
\put(350,330){\fbox{$a_{C,B,A}$}}

\putmorphism(-150,200)(0,-1)[\phantom{Y_2}``=]{380}1l
\putmorphism(230,200)(0,-1)[q` `=]{380}1l
\putmorphism(580,200)(0,-1)[\phantom{Y_2}``=]{380}1l
\putmorphism(720,200)(0,-1)[\phantom{Y_2}``=]{380}1l
\putmorphism(1090,200)(0,-1)[\phantom{Y_2}``=]{380}1l

\putmorphism(-150,-160)(1,0)[p`q `A']{360}1b
\putmorphism(230,-160)(1,0)[`r `B']{360}1b
 \putmorphism(730,-160)(1,0)[r`s`C']{360}1b

\put(-60,-20){\fbox{$F\s'$}}
\put(290,-20){\fbox{$G'$}}
\put(800,-20){\fbox{$H'$}}

\efig
\quad\stackrel{a_{H',G',F\s'}}{\Rrightarrow}\quad
\bfig
 \putmorphism(-150,500)(1,0)[p`q `A]{360}1a
\putmorphism(380,500)(1,0)[q`r`B]{360}1a
 \putmorphism(740,500)(1,0)[`s `C]{360}1a

\putmorphism(-150,520)(0,-1)[\phantom{Y_2}`p `=]{380}1l
\putmorphism(210,520)(0,-1)[` `=]{380}1l
\putmorphism(370,520)(0,-1)[\phantom{Y_2}``=]{380}1l
\putmorphism(720,520)(0,-1)[\phantom{Y_2}``=]{380}1l
\putmorphism(1090,520)(0,-1)[\phantom{Y_2}`s`=]{380}1l

 \putmorphism(-130,150)(1,0)[`q`A']{360}1a
 \putmorphism(380,150)(1,0)[q``B']{360}1a
 \putmorphism(730,150)(1,0)[r``C']{360}1a
\put(350,-20){\fbox{$a_{C',B',A'}$}}

\putmorphism(-150,200)(0,-1)[\phantom{Y_2}``=]{380}1l
\putmorphism(1090,200)(0,-1)[\phantom{Y_2}``=]{380}1l

\putmorphism(-150,-160)(1,0)[p`q `A']{360}1b
\putmorphism(230,-160)(1,0)[`r `B']{360}1b
 \putmorphism(730,-160)(1,0)[r`s`C']{360}1b

\put(-60,310){\fbox{$F\s'$}}
\put(430,310){\fbox{$G'$}}
\put(800,310){\fbox{$H'$}}

\efig
$$
(in terms of equation: $a_{H',G',F\s'} \cdot \big(\big(\gamma\ot(\beta\ot\alpha)\big)\odot\Id\big)
 = \big(\Id\odot\big((\gamma\ot\beta)\ot\alpha\big)\big) \cdot a_{H,G,F}$ \quad); 

(TD6) gives unity laws for $\ot$; \\
for a 1-cell $A:p\to q$ there are 2-cells $l_A: I_q\ot A\Rightarrow A$ and $r_A: A\ot I_p\Rightarrow A$ and their quasi-inverses 
$l_A'$ and $r_A'$ so that there are invertible 3-cells $l_A\odot l_A'\Rrightarrow \id_A, \hspace{0,2cm} \id_{I_q\ot A}\Rrightarrow l_A'\odot l_A, 
\hspace{0,2cm} r_A\odot r_A'\Rrightarrow \id_A, \hspace{0,2cm} \id_{A\ot I_p}\Rrightarrow l_A'\odot l_A$; \\
moreover, for any 2-cell $F:A\Rightarrow B$ 
there are invertible 3-cells (natural in $F$): 
$$l_B\odot(\Id_{I_q}\ot F)\stackrel{l_F}{\Rrightarrow} F\odot l_A, \quad l_B'\odot F\stackrel{l_F'}{\Rrightarrow} (\Id_{I_q}\ot F)\odot l'_A \hspace{0,14cm}, $$
$$r_B\odot(F\ot\Id_{I_p})\stackrel{r_F}{\Rrightarrow} F\odot r_A, \quad r_B'\odot F\stackrel{r_F'}{\Rrightarrow} (F\ot\Id_{I_p})\odot r'_A;$$
for a 3-cell $\alpha:F\Rrightarrow G$ one has the identities: 
$$l_G\cdot\big(\Id_{l_B}\odot(\Id_{\Id_{I_q}}\ot \alpha)\big)=(\alpha\odot\Id_{l_A})\cdot l_F$$
$$r_G\cdot\big(\Id_{r_B}\odot(\alpha\ot\Id_{\Id_{I_p}})\big)=(\alpha\odot\Id_{r_A})\cdot r_F;$$

(TD7) for every four composable 1-cells $D,C,B,A$ there is an invertible 3-cell 
$$
\bfig
 \putmorphism(-510,400)(1,0)[` `A]{360}1a
 \putmorphism(-150,400)(1,0)[(` `B]{360}1a
\putmorphism(280,400)(1,0)[``C]{360}1a
 \putmorphism(620,400)(1,0)[` )`D]{360}1a

\putmorphism(-490,420)(0,-1)[\phantom{Y_2}` `=]{330}1l
\putmorphism(950,420)(0,-1)[\phantom{Y_2}``=]{330}1l

  \putmorphism(-510,100)(1,0)[` `A]{360}1b
\putmorphism(-130,100)(1,0)[(``B]{360}1b
 \putmorphism(200,100)(1,0)[``C]{360}1b
 \putmorphism(620,100)(1,0)[`)`D]{360}1b
\put(-50,230){\fbox{$a_{D,C,B}\ot id_A$}}
\putmorphism(-490,100)(0,-1)[\phantom{Y_2}` `=]{330}1l
\putmorphism(950,100)(0,-1)[\phantom{Y_2}``=]{330}1l

 \putmorphism(-510,-200)(1,0)[(` `A]{360}1b
\putmorphism(-80,-200)(1,0)[``B]{360}1b
 \putmorphism(260,-200)(1,0)[` )`C]{360}1b
 \putmorphism(620,-200)(1,0)[``D]{360}1b

\putmorphism(-490,-200)(0,-1)[\phantom{Y_2}` `=]{330}1l
\putmorphism(950,-200)(0,-1)[\phantom{Y_2}``=]{330}1l
\put(150,-90){\fbox{$a_{D,CB,A}$}}

\put(-50,-400){\fbox{$id_D\ot a_{C,B,A}$}}

 \putmorphism(-510,-500)(1,0)[(` `A]{360}1b  
\putmorphism(-180,-500)(1,0)[``B]{360}1b 
 \putmorphism(240,-500)(1,0)[` )`C]{360}1b
 \putmorphism(620,-500)(1,0)[``D]{360}1b
\efig
\quad
\stackrel{\pi}{\Rrightarrow}
\quad
\bfig
 \putmorphism(-510,400)(1,0)[` `A]{360}1a
 \putmorphism(-150,400)(1,0)[(` `B]{360}1a
\putmorphism(280,400)(1,0)[``C]{360}1a
 \putmorphism(620,400)(1,0)[` )`D]{360}1a

\putmorphism(-490,420)(0,-1)[\phantom{Y_2}` `=]{330}1l
\putmorphism(950,420)(0,-1)[\phantom{Y_2}``=]{330}1l
\put(70,260){\fbox{$a_{DC,B,A}$}}

 \putmorphism(-510,100)(1,0)[` `A]{360}1a
 \putmorphism(-150,100)(1,0)[` `B]{360}1a
\putmorphism(280,100)(1,0)[(``C]{360}1a
 \putmorphism(620,100)(1,0)[` )`D]{360}1a

\putmorphism(-490,120)(0,-1)[\phantom{Y_2}` `=]{330}1l
\putmorphism(950,120)(0,-1)[\phantom{Y_2}``=]{330}1l
\put(-70,-60){\fbox{$a_{D,C,BA}$}}

 \putmorphism(-510,-200)(1,0)[(` `A]{360}1b  
\putmorphism(-180,-200)(1,0)[``B]{360}1b 
 \putmorphism(240,-200)(1,0)[` )`C]{360}1b
 \putmorphism(620,-200)(1,0)[``D.]{360}1b

\efig
$$ 
so that for four horizontally composable 2-cells
$
\xymatrix@C=8pt{ 
 \ar@/^1pc/[rr]^{A}\ar@{}[rr]| {\Downarrow F} \ar@/_1pc/[rr]_{ A'} &  &  
\ar@/^1pc/[rr]^{B}\ar@{}[rr]|{\Downarrow G} \ar@/_1pc/[rr]_{ B'} &  &   
\ar@/^1pc/[rr]^{C}\ar@{}[rr]|{\Downarrow H} \ar@/_1pc/[rr]_{ C'} &  &   
\ar@/^1pc/[rr]^{D}\ar@{}[rr]|{\Downarrow J} \ar@/_1pc/[rr]_{ D'} &  & }  
$ the following two transversal compositions of 3-cells coincide: \\
$$
\fourfrac{a\ot\id}{a_{\bullet,\bullet\bullet,\bullet}}{\id\ot a}{J\ot(H\ot(G\ot F))} 
\stackrel{\displaystyle{\frac{\pi}{\Id}}}{\Rrightarrow}
\threefrac{a_{\bullet\bullet,\bullet,\bullet}}{a_{\bullet,\bullet,\bullet\bullet}}{J\ot(H\ot(G\ot F))} 
\stackrel{ \displaystyle{\frac{\Id}{a_{J,H,GF}} }}{\Rrightarrow}
\threefrac{a_{\bullet\bullet,\bullet,\bullet}}{(J\ot H)\ot(G\ot F)}{a_{\bullet,\bullet,\bullet\bullet}} 
\stackrel{ \displaystyle{\frac{a_{JH,G,F}}{\Id}} }{\Rrightarrow}
\threefrac{((J\ot H)\ot G)\ot F}{a_{\bullet\bullet,\bullet,\bullet}}{a_{\bullet,\bullet,\bullet\bullet}} 
$$
$$\Downarrow \displaystyle{\threefrac{\Id}{\Id}{id\ot a}} \hspace{11,4cm} \interleave$$
$
\fourfrac{a\ot\id}{a_{\bullet,\bullet\bullet,\bullet}}{J\ot((H\ot G)\ot F)} {\id\ot a}
\stackrel{\displaystyle{\threefrac{\Id}{a_{\bullet,\bullet\bullet,\bullet}}{\Id}}}{\Rrightarrow}
\fourfrac{a\ot\id}{(J\ot(H\ot G))\ot F} {a_{\bullet,\bullet\bullet,\bullet}}{\id\ot a}
\stackrel{\displaystyle{ \threefrac{\Id}{\Id}{\pi* }}}{\Rrightarrow}
\fivefrac{a\ot\id}{(J\ot(H\ot G))\ot F}{a^{-1}\ot\id} {a_{\bullet\bullet,\bullet,\bullet}}{a_{\bullet,\bullet,\bullet\bullet}}
\stackrel{ \displaystyle{ \frac{a_{J,H,G}\ot\id}{\Id_{\equiv}} }}{\Rrightarrow}
\fivefrac{((J\ot H)\ot G)\ot F} {a\ot\id}{a^{-1}\ot\id} {a_{\bullet\bullet,\bullet,\bullet}}{a_{\bullet,\bullet,\bullet\bullet}} 
$\\
where the fractions denote vertical compositions of both 2- and 3-cells and the 2-cells $a_{\bullet,\bullet,\bullet}$ are 
evaluated at 1-cells $A,B,C,D$; 

(TD8) for composable 1-cells $p\stackrel{A}{\to}q\stackrel{B}{\to}r$ there exist 3-cells: 
$$
\bfig
 \putmorphism(-510,150)(1,0)[` `A]{360}1a
 \putmorphism(-150,150)(1,0)[(` `\Id]{360}1a
\putmorphism(210,150)(1,0)[`)`B]{360}1a

\putmorphism(-490,170)(0,-1)[\phantom{Y_2}` `=]{330}1l
\putmorphism(570,170)(0,-1)[\phantom{Y_2}``=]{330}1l
\put(-180,-20){\fbox{$r_B\ot\id$}}

 \putmorphism(-510,-150)(1,0)[` `A]{540}1a
 \putmorphism(20,-150)(1,0)[` `B]{540}1a
\efig
\quad\stackrel{\displaystyle{\mu_{B,A}}}{\Rrightarrow}
\bfig
 \putmorphism(-510,300)(1,0)[` `A]{360}1a
 \putmorphism(-150,300)(1,0)[(` `\Id]{360}1a
\putmorphism(210,300)(1,0)[`)`B]{360}1a

\putmorphism(-490,320)(0,-1)[\phantom{Y_2}` `=]{330}1l
\putmorphism(570,320)(0,-1)[\phantom{Y_2}``=]{330}1l
\put(-70,140){\fbox{$a_{B,\Id,A}$}}

 \putmorphism(-490,0)(1,0)[(` `A]{360}1b  
\putmorphism(-160,0)(1,0)[`)`\Id]{360}1b 
 \putmorphism(200,0)(1,0)[` `B]{360}1b

\putmorphism(-490,20)(0,-1)[\phantom{Y_2}` `=]{330}1l
\putmorphism(570,20)(0,-1)[\phantom{Y_2}``=]{330}1l
\put(-150,-210){\fbox{$\id\ot l_A$}}

 \putmorphism(-510,-300)(1,0)[` `A]{540}1a
 \putmorphism(20,-300)(1,0)[` `B]{540}1a

\efig
$$

$$
\bfig
 \putmorphism(-510,150)(1,0)[` `A]{360}1a
 \putmorphism(-150,150)(1,0)[(` `B]{360}1a
\putmorphism(210,150)(1,0)[`)`\Id]{360}1a

\putmorphism(-490,170)(0,-1)[\phantom{Y_2}` `=]{330}1l
\putmorphism(570,170)(0,-1)[\phantom{Y_2}``=]{330}1l
\put(-180,-20){\fbox{$l_B\ot\id$}}

 \putmorphism(-510,-150)(1,0)[` `A]{540}1a
 \putmorphism(20,-150)(1,0)[` `B]{540}1a
\efig
\quad\stackrel{\displaystyle{\lambda_{B,A}}}{\Rrightarrow}
\bfig
 \putmorphism(-510,300)(1,0)[` `A]{360}1a
 \putmorphism(-150,300)(1,0)[(` `B]{360}1a
\putmorphism(210,300)(1,0)[`)`\Id]{360}1a

\putmorphism(-490,320)(0,-1)[\phantom{Y_2}` `=]{330}1l
\putmorphism(570,320)(0,-1)[\phantom{Y_2}``=]{330}1l
\put(-70,140){\fbox{$a_{\Id,B,A}$}}

 \putmorphism(-490,0)(1,0)[(` `A]{360}1b  
\putmorphism(-160,0)(1,0)[`)`B]{360}1b 
 \putmorphism(200,0)(1,0)[` `\Id]{360}1b

\putmorphism(-490,20)(0,-1)[\phantom{Y_2}` `=]{330}1l
\putmorphism(570,20)(0,-1)[\phantom{Y_2}``=]{330}1l
\put(0,-210){\fbox{$l_{BA}$}}

 \putmorphism(-510,-300)(1,0)[` `A]{540}1a
 \putmorphism(20,-300)(1,0)[` `B]{540}1a

\efig
$$

$$
\bfig
 \putmorphism(-510,300)(1,0)[` `\Id]{360}1a
 \putmorphism(-150,300)(1,0)[(` `A]{360}1a
\putmorphism(210,300)(1,0)[`)`B]{360}1a

\putmorphism(-490,320)(0,-1)[\phantom{Y_2}` `=]{330}1l
\putmorphism(570,320)(0,-1)[\phantom{Y_2}``=]{330}1l
\put(-70,140){\fbox{$a_{B,A,\Id}$}}

 \putmorphism(-490,0)(1,0)[(` `\Id]{360}1b  
\putmorphism(-160,0)(1,0)[`)`A]{360}1b 
 \putmorphism(200,0)(1,0)[` `B]{360}1b

\putmorphism(-490,20)(0,-1)[\phantom{Y_2}` `=]{330}1l
\putmorphism(570,20)(0,-1)[\phantom{Y_2}``=]{330}1l
\put(-150,-210){\fbox{$\id\ot r_A$}}

 \putmorphism(-510,-300)(1,0)[` `A]{540}1a
 \putmorphism(20,-300)(1,0)[` `B]{540}1a

\efig
\quad\stackrel{\displaystyle{\rho_{B,A}}}{\Rrightarrow}
\bfig
 \putmorphism(-510,150)(1,0)[` `\Id]{360}1a
 \putmorphism(-150,150)(1,0)[(` `A]{360}1a
\putmorphism(210,150)(1,0)[`)`B]{360}1a

\putmorphism(-490,170)(0,-1)[\phantom{Y_2}` `=]{330}1l
\putmorphism(570,170)(0,-1)[\phantom{Y_2}``=]{330}1l
\put(-80,-20){\fbox{$r_{BA}$}}

 \putmorphism(-510,-150)(1,0)[` `A]{540}1a
 \putmorphism(20,-150)(1,0)[` `B]{540}1a
\efig
$$
so that for two horizontally composable 2-cells
$
\xymatrix@C=8pt{ 
 \ar@/^1pc/[rr]^{A}\ar@{}[rr]| {\Downarrow F} \ar@/_1pc/[rr]_{ A'} &  &  
\ar@/^1pc/[rr]^{B}\ar@{}[rr]|{\Downarrow G} \ar@/_1pc/[rr]_{ B'} &  & }  
$ the following three pairs of transversal compositions of 3-cells coincide, the first one involving $\mu$: \\
$$
\left\{\frac{(G\ot\Id)\ot F}{r_{B'}\ot\id_{A'}}\right\} \hspace{0,3cm}
\stackrel{\displaystyle{\frac{\Id}{\mu_{B', A'}}}}{\Rrightarrow} \hspace{0,3cm}
\left\{\threefrac{(G\ot\Id)\ot F}{a_{\bullet, \id,\bullet}}{\id_{B'}\ot l_{A'}}\right\} \hspace{0,3cm}
\stackrel{\displaystyle{\frac{a^{-1}_{G, \Id,F}}{\Id}}}{\Rrightarrow} \hspace{0,2cm}
\left\{\threefrac{a_{\bullet, \id,\bullet}}{G\ot(\Id\ot F)}{\id_{B'}\ot l_{A'}}\right\} \hspace{0,3cm}
\stackrel{\displaystyle{\frac{\Id}{\xi}}}{\Rrightarrow} \hspace{0,3cm}
\left\{\frac{a_{\bullet, \id,\bullet}}{[\frac{G}{\id_{B'}}]\ot[\frac{\Id\ot F}{l_{A'}}]}\right\}
%
$$
$$\hspace{0,8cm} \Downarrow  \displaystyle{\xi} \hspace{12,6cm} \Downarrow \displaystyle{\frac{\Id}{\omega_G\ot l_F}} $$
$$
\left\{[\frac{G\ot\Id}{r_{B'}}]\ot[\frac{F}{\id_{A'}}]\right\} 
\stackrel{\displaystyle{r_G\ot\omega_F}}{\Rrightarrow} 
\left\{[\frac{r_{B}}{G}]\ot[\frac{\id_{A}}{F}]\right\} 
\stackrel{\displaystyle{\xi^{-1}}}{\Rrightarrow} 
\left\{\frac{r_B\ot\id_A}{G\ot F}\right\} 
\stackrel{\displaystyle{\frac{\mu_{B,A}}{\Id}}}{\Rrightarrow} 
\left\{\threefrac{a_{\bullet, \id,\bullet}}{\id_{B}\ot l_A}{G\ot F}\right\}
\stackrel{ \displaystyle{\frac{\Id}{\xi}} } {\Rrightarrow}
\left\{\frac{a_{\bullet, \id,\bullet}}{[\frac{\id_{B}}{G}]\ot[\frac{l_{A}}{F}]}\right\}
$$
the second one involving $\lambda$: 
$$
\left\{\fourfrac{a_{\id,\bullet,\bullet}}{\Id\ot(G\ot F)}{a^{-1}_{\id,\bullet,\bullet}}{l_{B'}\ot\id_{A'}}\right\} \hspace{0,3cm}
\stackrel{\displaystyle{\frac{a_{\Id,G,F}}{\Id_{=}}}}{\Rrightarrow} \hspace{0,3cm}
\left\{\fourfrac{(\Id\ot G)\ot F}{a_{\id,\bullet,\bullet}}{a^{-1}_{\id,\bullet,\bullet}}{l_{B'}\ot\id_{A'}}\right\} \hspace{0,3cm}
\stackrel{\displaystyle{\equiv}}{\Rrightarrow} \hspace{0,3cm}
\left\{\frac{(\Id\ot G)\ot F}{l_{B'}\ot\id_{A'}}\right\} \hspace{0,3cm}
\stackrel{\displaystyle{\xi}}{\Rrightarrow} \hspace{0,3cm}
\left\{[\frac{\Id\ot G}{l_{B'}}]\ot[\frac{F}{\id_{A'}}] \right\}
%
$$
$$ \hspace{1cm}\Downarrow \displaystyle{\frac{\Id_{=}}{\lambda^*_{B'A'}}} \hspace{10,6cm} \Downarrow \displaystyle{ l_G\ot\omega_F} $$
$$
\left\{\threefrac{a_{\id,\bullet, \bullet}}{\Id\ot(G\ot F)}{l_{B'A'}}\right\}  \hspace{0,5cm} 
\stackrel{\displaystyle{\frac{\Id}{l_{GF}}}}{\Rrightarrow} \hspace{0,5cm} 
\left\{\fourfrac{a_{\id,\bullet,\bullet}}{a^{-1}_{\id,\bullet,\bullet}}{l_{B}\ot\id_{A}}{G\ot F}\right\} \hspace{0,5cm} 
\displaystyle{\equiv} \hspace{0,5cm} 
\left\{\frac{l_{B}\ot\id_{A}}{G\ot F}\right\} \hspace{0,5cm} 
\stackrel{\displaystyle{\xi}}{\Rrightarrow} \hspace{0,5cm} 
\left\{[\frac{l_{B}}{G}]\ot[\frac{\id_{A}}{F}] \right\}
$$
the third one involving $\rho$: 
$$
\left\{\threefrac{(G\ot F)\ot\Id}{a_{\bullet,\bullet, \id}}{\id_{B'}\ot r_{A'}}\right\} \hspace{1cm}
\stackrel{\displaystyle{\frac{a^{-1}_{G,F,\Id}}{\Id}}}{\Rrightarrow} \hspace{1cm}
\left\{\threefrac{a_{\bullet,\bullet, \id}}{G\ot (F\ot\Id)}{\id_{B'}\ot r_{A'}}\right\} \hspace{1cm}
\stackrel{\displaystyle{\frac{\Id}{\xi}}}{\Rrightarrow} \hspace{1cm}
\left\{\frac{a_{\bullet,\bullet, \id}} {[\frac{G}{\id_{B'}}]\ot[\frac{F\ot\Id}{r_{A'}}]} \right\}
$$
$$\hspace{1cm} \Downarrow \displaystyle{\frac{\Id}{\rho_{B'A'}}} \hspace{10,2cm} \Downarrow \displaystyle{\frac{\Id}{\omega_G\ot r_F}} $$
$$
\left\{\frac{(G\ot F)\ot\Id}{\r_{B'A'}}\right\} \hspace{0,6cm}
\stackrel{\displaystyle{r_{GF} }}{\Rrightarrow}  \hspace{0,6cm}
\left\{\frac{r_{BA}}{G\ot F}\right\} \hspace{0,6cm}
\stackrel{\displaystyle{\frac{\rho^{-1}_{BA}}{\Id}}}{\Rrightarrow} \hspace{0,5cm}
\left\{\threefrac{a_{\bullet,\bullet, \id}}{\id_{B}\ot r_{A}}{G\ot F}\right\} \hspace{0,5cm}
\stackrel{\displaystyle{\frac{\Id}{\xi}}}{\Rrightarrow} \hspace{0,5cm}
\left\{\frac{a_{\bullet,\bullet, \id}}{[\frac{\id_{B}}{G}]\ot[\frac{r_{A}}{F}] }\right\},
$$
here $\omega$ is the appropriate composition of unity constraints for the vertical composition mapping  
$\omega_F:\frac{F}{\id_{A'}}\to \frac{\id_{A}}{F}$ and similarly for $G$.

\section{Monoidal double categories into which monoidal bicategories embed} \selabel{mon}

Although bicategories embed into pseudo double categories, this embedding is not monoidal, 
if one takes for a definition of a monoidal double category any of the ones in \cite{BMM} 
(a monoid in the category of strict double categories and strict double functors) and in \cite{Shul} 
(a pseudomonoid in the 2-category of pseudo double categories, pseudo double functors and say vertical transformations). 
Namely, a monoidal bicategory is a one object tricategory, so its 0-cells have a product associative up to an {\em equivalence}. 
This is far from what happens in the mentioned two definitions of a monoidal double category. 
Even if we consider the triequivalence due to \cite{GPS} of a monoidal bicategory with a one object Gray-category, that is, a Gray monoid, 
one does not have monoidal embeddings, as we will show. Nevertheless, a Gray monoid, which is in fact a monoid in the monoidal category 
$(Gray,\ot)$ of 2-categories, 2-functors with the monoidal product due to Gray \cite{Gray}, can be seen as a monoid in the monoidal category 
$(Dbl,\ot)$ of strict double categories and double functors given by isomorpisms in both direction with the monoidal product constructed in \cite[Section 4.3]{Gabi}. The latter double functors were named later {\em double pseudo functors} in \cite[Definition 6.1]{Shul1}.  
We will show in this section that $(Gray,\ot)$ embeds monoidaly into $(Dbl,\ot)$.

\subsection{The monoidal structure in $(Dbl,\ot)$ } \sslabel{mon str}

The monoidal structure in $(Dbl,\ot)$ is constructed in the analogous way as in \cite{Gray}. For two double categories $\Aa, \Bb$ 
a double category $\llbracket\Aa,\Bb\rrbracket$ is defined in \cite[Section 2.2]{Gabi} which induces a functor $\llbracket-,-\rrbracket: Dbl^{op}\times Dbl\to Dbl$.  
Representability of the functor $Dbl(\Aa, \llbracket\Bb,-\rrbracket): Dbl\to Set$ is proved, which induces a functor 
$-\ot-: Dbl\times Dbl\to Dbl$. 
For two double categories $\Aa,\Bb$ we will give a full description of the double category $\Aa\ot\Bb$. We will do this using the natural isomorphism 
\begin{equation} \eqlabel{nat iso}
Dbl(\Aa\ot\Bb, \Cc)\iso Dbl(\Aa, \llbracket\Bb,\Cc\rrbracket),
\end{equation}
that is, characterizing a double functor $F:\Aa\to\llbracket\Bb,\Cc\rrbracket$ for another double category $\Cc$ and reading off 
the structure of the image double category $F(\Aa)(\Bb)$, setting $\Cc=\Aa\ot\Bb$. 

\medskip

Let us fix the notation in a double category $\Dd$. 
Objects we denote by $A,B, \dots$, horizontal 1-cells we will call 
for brevity 1h-cells and denote them by $f, f', g, F, \dots$, vertical 1-cells we will call 1v-cells and denote by $u,v, U, \dots$, 
and squares we will call just 2-cells and denote them by $\omega, \zeta, \dots$. We denote 
the horizontal identity 1-cell by $1_A$, vertical identity 1-cell by $1^A$ for an object $A\in\Dd$, 
horizontal identity 2-cell on a 1v-cell $u$ by $Id^u$, and vertical identity 2-cell on a 1h-cell $f$ by $Id_f$ (with subindices we denote 
those identity 1- and 2-cells which come from the horizontal 2-category lying in $\Dd$, for this reason we flip the r\^oles of sub and 
supra indices in identity 2-cells with respect to those for identity 1-cells). 
The composition of 1h-cells as well as the 
horizontal composition of 2-cells we will denote by $\odot$, while the composition of 1v-cells as well the 
vertical composition of 2-cells we will denote by juxtaposition. 

We start by noticing that a strict double functor $F:\Cc\to\Dd$ is given by 1) the data: images on objects, 1h-, 1v- and 2-cells of $\Cc$, and 2) rules (in $\Dd$): 

$F(u'u)=F(u')F(u), \quad F(1^A)=1^{F(A)},$ 

$F(\omega\zeta)=F(\omega)F(\zeta), \quad F(1_f)=1_{F(f)},$ 

$F(g\odot f)=F(g)\odot F(f), \quad F(\omega\odot\zeta)=F(\omega)\odot F(\zeta),$ 

$F(1_A)=1_{F(A)},  \quad F(Id^u)=Id^{F(u)}.$

Having in mind the definition of a double category $\llbracket\Aa,\Bb\rrbracket$ from \cite[Section 2.2]{Gabi}, writing out the list of the data and relations 
that determine a double functor $F:\Aa\to\llbracket\Bb,\Cc\rrbracket$, one gets the following characterization of it:

\begin{prop} \prlabel{char df}
A double functor $F:\Aa\to\llbracket\Bb,\Cc\rrbracket$ of double categories consists of the following: \\
1. double functors 
$$(-,A):\Bb\to\Cc\quad\text{ and}\quad (B,-):\Aa\to\Cc$$ 
such that $(-,A)\vert_B=(B,-)\vert_A=(B,A)$, 
for objects $A\in\Aa, B\in\Bb$, \\
2. given 1h-cells $A\stackrel{F}{\to} A'$ and $B\stackrel{f}{\to} B'$ and 1v-cells $A\stackrel{U}{\to} \tilde A$ and $B\stackrel{u}{\to} \tilde B$ 
there are 2-cells
$$
\bfig
 \putmorphism(-150,500)(1,0)[(B,A)`(B,A')`(B,F)]{600}1a
 \putmorphism(450,500)(1,0)[\phantom{A\ot B}`(B', A') `(f, A')]{680}1a
 \putmorphism(-150,50)(1,0)[(B,A)`(B', A)`(f, A)]{600}1a
 \putmorphism(450,50)(1,0)[\phantom{A\ot B}`(B', A') `(B', F)]{680}1a
\putmorphism(-180,500)(0,-1)[\phantom{Y_2}``=]{450}1r
\putmorphism(1100,500)(0,-1)[\phantom{Y_2}``=]{450}1r
\put(350,260){\fbox{$(f,F)$}}
\efig
$$

$$
\bfig
\putmorphism(-150,50)(1,0)[(B,A)`(B,A')`(B,F)]{600}1a
\putmorphism(-150,-400)(1,0)[(\tilde B, A)`(\tilde B,A') `(\tilde B,F)]{640}1a
\putmorphism(-180,50)(0,-1)[\phantom{Y_2}``(u,A)]{450}1l
\putmorphism(450,50)(0,-1)[\phantom{Y_2}``(u,A')]{450}1r
\put(0,-180){\fbox{$(u, F)$}}
\efig
\quad
\bfig
\putmorphism(-150,50)(1,0)[(B,A)`(B',A)`(f,A)]{600}1a
\putmorphism(-150,-400)(1,0)[(B, \tilde A)`(B', \tilde A) `(f,\tilde A)]{640}1a
\putmorphism(-180,50)(0,-1)[\phantom{Y_2}``(B,U)]{450}1l
\putmorphism(450,50)(0,-1)[\phantom{Y_2}``(B',U)]{450}1r
\put(0,-180){\fbox{$(f,U)$}}
\efig
$$

$$
\bfig
 \putmorphism(-150,500)(1,0)[(B,A)`(B,A) `=]{600}1a
\putmorphism(-180,500)(0,-1)[\phantom{Y_2}`(B, \tilde A) `(B,U)]{450}1l
\put(0,50){\fbox{$(u,U)$}}
\putmorphism(-150,-400)(1,0)[(\tilde B, \tilde A)`(\tilde B, \tilde A) `=]{640}1a
\putmorphism(-180,50)(0,-1)[\phantom{Y_2}``(u,\tilde A)]{450}1l
\putmorphism(450,50)(0,-1)[\phantom{Y_2}``(\tilde B, U)]{450}1r
\putmorphism(450,500)(0,-1)[\phantom{Y_2}`(\tilde B, A) `(u,A)]{450}1r
\efig
$$
of which $(f,F)$ is vertically invertible  and $(u,U)$ is horizontally invertible, which satisfy: 
\begin{enumerate} [a)]
\item (11)\quad $(1_B,F)=Id_{(B,F)}\quad\text{and}\quad (f,1_A)=Id_{(f,A)}$ 

(21)\quad  $(1^B,F)=Id_{(B,F)}\quad\text{and}\quad (u,1_A)=Id^{(u,A)}$ 

(12)\quad  $(1_B,U)=Id^{(B,U)}\quad\text{and}\quad (f,1^A)=Id_{(f,A)}$ 

(22)\quad  $(1^B,U)=Id^{(B,U)}\quad\text{and}\quad (u,1^A)=Id^{(u,A)};$

\item (11)  
$(f'f, F)=
\bfig

 \putmorphism(-150,0)(1,0)[(B,A)`(B,A')`(B,F)]{600}1a
 \putmorphism(450,0)(1,0)[\phantom{A\ot B}`(B', A') `(f,A')]{680}1a

 \putmorphism(-150,-450)(1,0)[(B,A)`(B',A)`(f,A)]{600}1a
 \putmorphism(450,-450)(1,0)[\phantom{A\ot B}`(B', A') `(B',F)]{680}1a
 \putmorphism(1100,-450)(1,0)[\phantom{A'\ot B'}`(B'', A') `(f', A')]{660}1a

\putmorphism(-180,0)(0,-1)[\phantom{Y_2}``=]{450}1r
\putmorphism(1100,0)(0,-1)[\phantom{Y_2}``=]{450}1r
\put(350,-240){\fbox{$(f,F)$}}
\put(1000,-700){\fbox{$(f',F)$}}

 \putmorphism(450,-900)(1,0)[(B', A)` (B'', A) `(f', A)]{680}1a
 \putmorphism(1100,-900)(1,0)[\phantom{A''\ot B'}`(B'', A') ` (B'', F)]{660}1a

\putmorphism(450,-450)(0,-1)[\phantom{Y_2}``=]{450}1l
\putmorphism(1750,-450)(0,-1)[\phantom{Y_2}``=]{450}1r
\efig
$ \\
and \\
$(f,F\s'F)=
\bfig
 \putmorphism(450,500)(1,0)[(B, A') `(B, A'') `(B, F')]{680}1a
 \putmorphism(1140,500)(1,0)[\phantom{A\ot B}`(B', A'') ` (f, A'')]{680}1a

 \putmorphism(-150,50)(1,0)[(B, A) `(B, A')`(B, F)]{600}1a
 \putmorphism(450,50)(1,0)[\phantom{A\ot B}`(B', A') `(f, A')]{680}1a
 \putmorphism(1130,50)(1,0)[\phantom{A\ot B}`(B', A'') ` (B', F\s')]{680}1a

\putmorphism(450,500)(0,-1)[\phantom{Y_2}``=]{450}1r
\putmorphism(1750,500)(0,-1)[\phantom{Y_2}``=]{450}1r
\put(1020,270){\fbox{$ (f,F\s')$}}

 \putmorphism(-150,-400)(1,0)[(B, A)`(B', A) `(f,A)]{640}1a
 \putmorphism(480,-400)(1,0)[\phantom{A'\ot B'}`(B', A') `(B', F)]{680}1a

\putmorphism(-180,50)(0,-1)[\phantom{Y_2}``=]{450}1l
\putmorphism(1120,50)(0,-1)[\phantom{Y_3}``=]{450}1r
\put(310,-200){\fbox{$ (f,F)$}}

\efig
$ \\

(21)\quad $(u'u, F)=(u', F)(u,F)\quad\text{and}\quad (u, F\s'F)=(u,F\s')\odot(u,F)$ 

(12)\quad $(f'f, U)=(f', U)\odot(f, U)\quad\text{and}\quad (f,U'U)=(f,U')(f,U)$ 

(22)\quad 
$$(u,U'U)=
\bfig
 \putmorphism(-150,500)(1,0)[(B,A)`(B,A) `=]{600}1a
\putmorphism(-180,500)(0,-1)[\phantom{Y_2}`(B, \tilde A) `(B,U)]{450}1l
\put(0,50){\fbox{$(u,U)$}}
\putmorphism(-150,-400)(1,0)[(\tilde B, \tilde A)`(\tilde B, \tilde A) `=]{640}1a
\putmorphism(-180,50)(0,-1)[\phantom{Y_2}``(u,\tilde A)]{450}1l
\putmorphism(450,50)(0,-1)[\phantom{Y_2}``(\tilde B, U)]{450}1r
\putmorphism(450,500)(0,-1)[\phantom{Y_2}`(\tilde B, A) `(u,A)]{450}1r
\putmorphism(-820,50)(1,0)[(B, \tilde A)``=]{520}1a
\putmorphism(-820,50)(0,-1)[\phantom{(B, \tilde A')}``(B,U')]{450}1l
\putmorphism(-820,-400)(0,-1)[(B, \tilde A')`(\tilde B, \tilde A')`(u,\tilde A')]{450}1l
\putmorphism(-820,-850)(1,0)[\phantom{(B, \tilde A)}``=]{520}1a
\putmorphism(-150,-400)(0,-1)[(\tilde B, \tilde A)`(\tilde B, \tilde A') `(\tilde B, U')]{450}1r
\put(-650,-630){\fbox{$(u,U')$}}
\efig
$$
and
$$(u'u,U)=
\bfig
 \putmorphism(-150,500)(1,0)[(B,A)`(B,A) `=]{600}1a
\putmorphism(-180,500)(0,-1)[\phantom{Y_2}`(B, \tilde A) `(B,U)]{450}1l
\put(0,50){\fbox{$(u,U)$}}
\putmorphism(-150,-400)(1,0)[(\tilde B, \tilde A)` `=]{500}1a
\putmorphism(-180,50)(0,-1)[\phantom{Y_2}``(u,\tilde A)]{450}1l
\putmorphism(450,50)(0,-1)[\phantom{Y_2}`(\tilde B, \tilde A)`(\tilde B, U)]{450}1r
\putmorphism(450,500)(0,-1)[\phantom{Y_2}`(\tilde B, A) `(u,A)]{450}1r
\putmorphism(450,50)(1,0)[\phantom{(B, \tilde A)}`(\tilde B, A)`=]{620}1a
\putmorphism(1070,50)(0,-1)[\phantom{(B, \tilde A')}``(u',A)]{450}1r
\putmorphism(1070,-400)(0,-1)[(\tilde B', A)`(\tilde B', \tilde A)`(\tilde B', U)]{450}1r
\putmorphism(450,-850)(1,0)[\phantom{(B, \tilde A)}``=]{520}1a
\putmorphism(450,-400)(0,-1)[\phantom{(B, \tilde A)}`(\tilde B', \tilde A) `(U',\tilde A)]{450}1l
\put(600,-630){\fbox{$(u',U)$}}
\efig
$$
\item (11)\quad 
$$
\bfig
 \putmorphism(-150,500)(1,0)[(B,A)`(B,A')`(B,F)]{600}1a
 \putmorphism(450,500)(1,0)[\phantom{A\ot B}`(B', A') `(f, A')]{680}1a
 \putmorphism(-150,50)(1,0)[(B,A)`(B', A)`(f, A)]{600}1a
 \putmorphism(450,50)(1,0)[\phantom{A\ot B}`(B', A') `(B', F)]{680}1a

\putmorphism(-180,500)(0,-1)[\phantom{Y_2}``=]{450}1r
\putmorphism(1100,500)(0,-1)[\phantom{Y_2}``=]{450}1r
\put(350,260){\fbox{$(f,F)$}}
\put(650,-180){\fbox{$(v,F)$}}

\putmorphism(-150,-400)(1,0)[(\tilde B,A)`(\tilde B',A) `(g,A)]{640}1a
 \putmorphism(450,-400)(1,0)[\phantom{A'\ot B'}` (\tilde B',A') `(\tilde B', F)]{680}1a

\putmorphism(-180,50)(0,-1)[\phantom{Y_2}``(u,A)]{450}1l
\putmorphism(450,50)(0,-1)[\phantom{Y_2}``]{450}1l
\putmorphism(610,50)(0,-1)[\phantom{Y_2}``(v,A)]{450}0l 
\putmorphism(1120,50)(0,-1)[\phantom{Y_3}``(v,A')]{450}1r
\put(-40,-180){\fbox{$(\omega,A)$}} 

\efig
=
\bfig
\putmorphism(-150,500)(1,0)[(B,A)`(B,A')`(B,F)]{600}1a
 \putmorphism(450,500)(1,0)[\phantom{A\ot B}`(B', A') `(f, A')]{680}1a

 \putmorphism(-150,50)(1,0)[(\tilde B,A)`(\tilde B,A')`(\tilde B,F)]{600}1a
 \putmorphism(450,50)(1,0)[\phantom{A\ot B}`(\tilde B',A') `(g, A')]{680}1a

\putmorphism(-180,500)(0,-1)[\phantom{Y_2}``(u,A)]{450}1l
\putmorphism(450,500)(0,-1)[\phantom{Y_2}``]{450}1r
\putmorphism(300,500)(0,-1)[\phantom{Y_2}``(u,A')]{450}0r
\putmorphism(1100,500)(0,-1)[\phantom{Y_2}``(v,A')]{450}1r
\put(-20,260){\fbox{$(u,F)$}}
\put(660,260){\fbox{$(\omega,A')$}}

\putmorphism(-150,-400)(1,0)[(\tilde B,A)`(\tilde B',A) `(g, A)]{640}1a
 \putmorphism(450,-400)(1,0)[\phantom{A'\ot B'}` (\tilde B,A') `(\tilde B',F)]{680}1a

\putmorphism(-180,50)(0,-1)[\phantom{Y_2}``=]{450}1l
\putmorphism(1120,50)(0,-1)[\phantom{Y_3}``=]{450}1r
\put(300,-200){\fbox{$(g,F)$}}

\efig
$$
and 
$$
\bfig
 \putmorphism(-150,500)(1,0)[(B,A)`(B,A')`(B,F)]{600}1a
 \putmorphism(450,500)(1,0)[\phantom{A\ot B}`(B', A') `(f, A')]{680}1a
 \putmorphism(-150,50)(1,0)[(B,A)`(B', A)`(f, A)]{600}1a
 \putmorphism(450,50)(1,0)[\phantom{A\ot B}`(B', A') `(B', F)]{680}1a

\putmorphism(-180,500)(0,-1)[\phantom{Y_2}``=]{450}1r
\putmorphism(1100,500)(0,-1)[\phantom{Y_2}``=]{450}1r
\put(350,260){\fbox{$(f,F)$}}
\put(650,-180){\fbox{$(B', \zeta)$}}

\putmorphism(-150,-400)(1,0)[(B,\tilde A)`(B',\tilde A) `(f,\tilde A)]{640}1a
 \putmorphism(450,-400)(1,0)[\phantom{A'\ot B'}` (B',\tilde A') `(B', G)]{680}1a

\putmorphism(-180,50)(0,-1)[\phantom{Y_2}``(B,U)]{450}1l
\putmorphism(450,50)(0,-1)[\phantom{Y_2}``]{450}1l
\putmorphism(610,50)(0,-1)[\phantom{Y_2}``(B',U)]{450}0l 
\putmorphism(1120,50)(0,-1)[\phantom{Y_3}``(B',V)]{450}1r
\put(-40,-180){\fbox{$(f,U)$}} 

\efig
=
\bfig
\putmorphism(-150,500)(1,0)[(B,A)`(B,A')`(B,F)]{600}1a
 \putmorphism(450,500)(1,0)[\phantom{A\ot B}`(B', A') `(f, A')]{680}1a

 \putmorphism(-150,50)(1,0)[(B,\tilde A)`(B,\tilde A')`(B,G)]{600}1a
 \putmorphism(450,50)(1,0)[\phantom{A\ot B}`(B',\tilde A') `(f,\tilde A')]{680}1a

\putmorphism(-180,500)(0,-1)[\phantom{Y_2}``(B,U)]{450}1l
\putmorphism(450,500)(0,-1)[\phantom{Y_2}``]{450}1r
\putmorphism(300,500)(0,-1)[\phantom{Y_2}``(B,V)]{450}0r
\putmorphism(1100,500)(0,-1)[\phantom{Y_2}``(B',V)]{450}1r
\put(-20,260){\fbox{$(B,\zeta)$}}
\put(660,260){\fbox{$(F,V)$}}

\putmorphism(-150,-400)(1,0)[(B,\tilde A)`(B',\tilde A) `(f, \tilde A)]{640}1a
 \putmorphism(450,-400)(1,0)[\phantom{A'\ot B'}` (B',\tilde A') `(B',G)]{680}1a

\putmorphism(-180,50)(0,-1)[\phantom{Y_2}``=]{450}1l
\putmorphism(1120,50)(0,-1)[\phantom{Y_3}``=]{450}1r
\put(300,-200){\fbox{$(f,G)$}}

\efig
$$
(22)\quad
$$
\bfig
 \putmorphism(-150,500)(1,0)[(B,A)`(B,A) `=]{600}1a
 \putmorphism(450,500)(1,0)[(B,A)` `(f,A)]{450}1a
\putmorphism(-180,500)(0,-1)[\phantom{Y_2}`(B, \tilde A) `(B,U)]{450}1l
\put(0,50){\fbox{$(u,U)$}}
\putmorphism(-150,-400)(1,0)[(\tilde B, \tilde A)` `=]{500}1a
\putmorphism(-180,50)(0,-1)[\phantom{Y_2}``(u,\tilde A)]{450}1l
\putmorphism(450,50)(0,-1)[\phantom{Y_2}`(\tilde B, \tilde A)`(\tilde B, U)]{450}1l
\putmorphism(450,500)(0,-1)[\phantom{Y_2}`(\tilde B, A) `(u,A)]{450}1l
\put(600,260){\fbox{$(\omega,A)$}}
\putmorphism(450,50)(1,0)[\phantom{(B, \tilde A)}``(g, A)]{500}1a
\putmorphism(1070,50)(0,-1)[\phantom{(B, A')}`(\tilde B', \tilde A)`(\tilde B',U)]{450}1r
\putmorphism(1070,500)(0,-1)[(B', A)`(\tilde B', A)`(v,A)]{450}1r
\putmorphism(450,-400)(1,0)[\phantom{(B, \tilde A)}``(g, \tilde A)]{500}1a
\put(600,-170){\fbox{$(g,U)$}}
\efig=
\bfig
 \putmorphism(-150,500)(1,0)[(B,A)`(B',A) `(f,A)]{600}1a
 \putmorphism(450,500)(1,0)[\phantom{(B,A)}` `=]{450}1a
\putmorphism(-180,500)(0,-1)[\phantom{Y_2}`(B, \tilde A) `(B,U)]{450}1l
\put(620,50){\fbox{$(v,U)$}}
\putmorphism(-150,-400)(1,0)[(\tilde B, \tilde A)` `(g, \tilde A)]{500}1a
\putmorphism(-180,50)(0,-1)[\phantom{Y_2}``(u,\tilde A)]{450}1l
\putmorphism(450,50)(0,-1)[\phantom{Y_2}`(\tilde B', \tilde A)`(v,\tilde A)]{450}1r
\putmorphism(450,500)(0,-1)[\phantom{Y_2}`(B', \tilde A) `(B',U)]{450}1r
\put(0,260){\fbox{$(f,U)$}}
\putmorphism(-150,50)(1,0)[\phantom{(B, \tilde A)}``(f, \tilde A)]{500}1a
\putmorphism(1070,50)(0,-1)[\phantom{(B, A')}`(\tilde B', \tilde A)`(\tilde B',U)]{450}1r
\putmorphism(1070,500)(0,-1)[(B', A)`(\tilde B', A)`(v,A)]{450}1r
\putmorphism(450,-400)(1,0)[\phantom{(B, \tilde A)}``=]{500}1b
\put(0,-170){\fbox{$(\omega, \tilde A)$}}
\efig
$$
and
$$
\bfig
 \putmorphism(-150,500)(1,0)[(B,A)`(B,A) `=]{600}1a
 \putmorphism(550,500)(1,0)[` `(B,F)]{400}1a
\putmorphism(-180,500)(0,-1)[\phantom{Y_2}`(B, \tilde A) `(B,U)]{450}1l
\put(0,50){\fbox{$(u,U)$}}
\putmorphism(-150,-400)(1,0)[(\tilde B, \tilde A)` `=]{500}1a
\putmorphism(-180,50)(0,-1)[\phantom{Y_2}``(u,\tilde A)]{450}1l
\putmorphism(450,50)(0,-1)[\phantom{Y_2}`(\tilde B, \tilde A)`(\tilde B, U)]{450}1l
\putmorphism(450,500)(0,-1)[\phantom{Y_2}`(\tilde B, A) `(u,A)]{450}1l
\put(620,280){\fbox{$(u,F)$}}
\putmorphism(450,50)(1,0)[\phantom{(B, \tilde A)}``(\tilde B,F)]{500}1a
\putmorphism(1070,50)(0,-1)[\phantom{(B, A')}`(\tilde B, \tilde A')`(\tilde B,V)]{450}1r
\putmorphism(1070,500)(0,-1)[(B, A')`(\tilde B, A')`(u,A')]{450}1r
\putmorphism(450,-400)(1,0)[\phantom{(B, \tilde A)}``(\tilde B, G)]{500}1a
\put(620,-170){\fbox{$ (\tilde{B},\zeta)$ } } 
\efig=
\bfig
 \putmorphism(-150,500)(1,0)[(B,A)`(B,A') `(B,F)]{600}1a
 \putmorphism(450,500)(1,0)[\phantom{(B,A)}` `=]{450}1a
\putmorphism(-180,500)(0,-1)[\phantom{Y_2}`(B, \tilde A) `(B,U)]{450}1l
\put(620,50){\fbox{$(u,V)$}}
\putmorphism(-150,-400)(1,0)[(\tilde B, \tilde A)` `(\tilde B, G)]{500}1a
\putmorphism(-180,50)(0,-1)[\phantom{Y_2}``(u,\tilde A)]{450}1l
\putmorphism(450,50)(0,-1)[\phantom{Y_2}`(\tilde B, \tilde A')`(u,\tilde A')]{450}1r
\putmorphism(450,500)(0,-1)[\phantom{Y_2}`(B, \tilde A') `(B,V)]{450}1r
\put(0,260){\fbox{$(B,\zeta)$}}
\putmorphism(-150,50)(1,0)[\phantom{(B, \tilde A)}``(B,G)]{500}1a
\putmorphism(1070,50)(0,-1)[\phantom{(B, A')}`(\tilde B, \tilde A')`(\tilde B,V)]{450}1r
\putmorphism(1070,500)(0,-1)[(B, A')`(\tilde B, A')`(u,A')]{450}1r
\putmorphism(450,-400)(1,0)[\phantom{(B, \tilde A)}``=]{500}1b
\put(0,-170){\fbox{$(u,G)$}}
\efig
$$
for any 2-cells 
\begin{equation} \eqlabel{omega-zeta}
\bfig
\putmorphism(-150,50)(1,0)[B` B'`f]{450}1a
\putmorphism(-150,-300)(1,0)[\tilde B`\tilde B' `g]{440}1b
\putmorphism(-170,50)(0,-1)[\phantom{Y_2}``u]{350}1l
\putmorphism(280,50)(0,-1)[\phantom{Y_2}``v]{350}1r
\put(0,-140){\fbox{$\omega$}}
\efig
\quad\text{and}\quad
\bfig
\putmorphism(-150,50)(1,0)[A` A'`F]{450}1a
\putmorphism(-150,-300)(1,0)[\tilde A`\tilde A'. `G]{440}1b
\putmorphism(-170,50)(0,-1)[\phantom{Y_2}``U]{350}1l
\putmorphism(280,50)(0,-1)[\phantom{Y_2}``V]{350}1r
\put(0,-140){\fbox{$\zeta$}}
\efig
\end{equation}
in $\Bb$, respectively $\Aa$. 
\end{enumerate}
\end{prop}


In analogy to \cite[Section 4.2]{GPS} we set: 

\begin{defn} \delabel{H dbl}
The characterization in the above Proposition gives rise to an application from the Cartesian product of double categories 
$H: \Aa\times\Bb\to\Cc$ such that $H(A,-)=(-, A)$ and $H(-, B)=(B,-)$. Such an application of double categories we 
will call {\em cubical double functor}. 
\end{defn}

We may now describe a double category $\Aa\ot\Bb$ by reading off the structure of the image double category $F(\Aa)(\Bb)$  
for any double functor $F:\Aa\to\llbracket\Bb,\Aa\ot\Bb\rrbracket$ in the right hand-side of \equref{nat iso} using the 
characterization of a double functor before \prref{char df}. With the notation $F(x)(y)=(y,x)=:x\ot y$ for any 0-, 1h-, 1v- or 2-cells 
$x$ of $\Aa$ and $y$ of $\Bb$ we obtain that a double category $\Aa\ot\Bb$ consists of the following: \\
\u{objects}: $A\ot B$ for objects $A\in\Aa, B\in\Bb$; \\ 
\u{1h-cells}: $A\ot f, F\ot B$ and horizontal compositions of such (modulo associativity and unity constraints) obeying the following rules: 
$$(A\ot f')\odot(A\ot f)=A\ot (f'\odot f), \quad (F\s'\ot B)\odot(F\ot B)=(F\s'\odot F)\ot B, \quad A\ot 1_B=1_{A\ot B}=1_A\ot B$$
where $f,f'$ are 1h-cells of $\Bb$ and $F, F\s'$ 1h-cells of $\Aa$; \\
\u{1v-cells}: $A\ot u, U\ot B$ and vertical compositions of such obeying the following rules: 
$$(A\ot u')(A\ot u)=A\ot u'u, \quad (U'\ot B)(U\ot B)=U'U\ot B, \quad A\ot 1^B=1^{A\ot B}=1^A\ot B$$
where $u,u'$ are 1v-cells of $\Bb$ and $U,U'$ 1v-cells of $\Aa$; \\
\u{2-cells}: $A\ot\omega, \zeta\ot B$:
$$
\bfig
\putmorphism(-150,50)(1,0)[A\ot B`A\ot B'`A\ot f]{600}1a
\putmorphism(-150,-400)(1,0)[A\ot\tilde B`A'\ot\tilde B' `A\ot g]{640}1a
\putmorphism(-180,50)(0,-1)[\phantom{Y_2}``A\ot u]{450}1l
\putmorphism(450,50)(0,-1)[\phantom{Y_2}``A\ot v]{450}1r
\put(0,-180){\fbox{$A\ot \omega$}}
\efig
\qquad
\bfig
\putmorphism(-150,50)(1,0)[A\ot B`A'\ot B`F\ot B]{600}1a
\putmorphism(-150,-400)(1,0)[\tilde A\ot B`\tilde A'\ot B `G\ot B]{640}1a
\putmorphism(-180,50)(0,-1)[\phantom{Y_2}``U\ot B]{450}1l
\putmorphism(450,50)(0,-1)[\phantom{Y_2}``V\ot B]{450}1r
\put(0,-180){\fbox{$\zeta\ot B$}}
\efig
$$
where $\omega$ and $\zeta$ are as in \equref{omega-zeta},
and four types of 2-cells coming from the 2-cells of point 2. in \prref{char df}: 
vertically invertible globular 2-cell $F\ot f: (A'\ot f)\odot(F\ot B) \Rightarrow(F\ot B')\odot(A\ot f)$, 
horizontally invertible globular 2-cell $U\ot u: (\tilde A\ot u)(U\ot B)\Rightarrow(U\ot\tilde B)(A\ot u)$, 
2-cells $F\ot u$ and $U\ot f$, and the horizontal and vertical compositions of these 
(modulo associativity and unity constraints in the horizontal direction and the interchange law) 
subject to the rules induced by a), b) and c) of point 2. in \prref{char df} and the following ones: 
$$A\ot(\omega'\odot\omega)=(A\ot\omega')\odot(A\ot\omega), \quad (\zeta'\odot\zeta)\ot B=(\zeta'\ot B)\odot(\zeta\ot B),$$
$$A\ot(\omega'\omega)=(A\ot\omega')(A\ot\omega), \quad (\zeta'\zeta)\ot B=(\zeta'\ot B)(\zeta\ot B),$$
$$A\ot\Id_f=\Id_{A\ot f}, \quad \Id_F\ot B=\Id_{F\ot B}, \quad A\ot\Id^f=\Id^{A\ot f}, \quad \Id^F\ot B=\Id^{F\ot B}.$$

\subsection{A monoidal embedding of $(Gray,\ot)$ into $(Dbl,\ot)$ } \sslabel{mon embed}

Let $E: (Gray,\ot)\hookrightarrow(Dbl,\ot)$ denote the embedding functor which to a 2-category assigns a strict double category whose 
all vertical 1-cells are identities and whose 2-cells are vertically globular cells. Then $E$ is a left adjoint to the functor that to a strict double category 
assigns its underlying horizontal 2-category. Let us denote by $\C=(Gray,\ot)$ and by $\D=\im(E)\subseteq(Dbl,\ot)$, the image category by $E$, 
then the corestriction of $E$ to $\D$ is the identity functor 
\begin{equation}\eqlabel{id fun}
F:\C\to\D.
\end{equation}

In order to examine the monoidality of $F$ let us first consider 
an assignment $t: F(\A\ot\B)\to F(\A)* F(\B)$ for two 2-categories $\A$ and $\B$, where $*$ denotes some monoidal product in the category of 
strict double categories which a priori could be the Cartesian one or the one from the monoidal category $(Dbl,\ot)$. 

Observe that given 1-cells $f:A\to A'$ in $\A$ and $g: B\to B'$ in $\B$ the composition 1-cells 
$(f\ot B')\odot(A\ot g)$ and $(A'\ot g)\odot(f\ot B)$ in $\A\ot\B$ are not equal both in $(Gray,\ot)$ and in $(Dbl,\ot)$. 
This means that their images $F\big((f\ot B')\odot(A\ot g)\big)$ and $F\big((A'\ot g)\odot(f\ot B)\big)$ are different as 1h-cells of the double category 
$F(\A\ot\B)$. Now if we map these two images by $t$ into the Cartesian product $F(\A)\times F(\B)$, we will get in both cases the 1h-cell 
$(f,g)$. Then $t$ with the codomain in the Cartesian product is a bad candidate for the monoidal structure of the identity functor $F$. 
This shows that the Cartesian monoidal product on the category of strict double categories is not a good choice for a monoidal structure 
if one wants to embed the Gray category of 2-categories into the latter category. In contrast, if the codomain of $t$ is the monoidal product 
of $(Dbl,\ot)$, we see that $t$ is identity on these two 1-cells. 

Similar considerations and comparing the monoidal product from \cite[Theorem I.4.9]{Gray} 
in $(Gray,\ot)$ to the one after \deref{H dbl} above in $(Dbl,\ot)$, show that for the candidate for (the one part of) a monoidal structure 
on the identity functor $F$ we may take the identity $s=\Id: F(\A\ot\B)\to F(\A)\ot F(\B)$, and that it is indeed a strict double functor of 
strict double categories. 
(Observe that the hexagonal and two square relations for the double functor $s$ concern 1h-cells, and that the associativity and unity constraints 
appearing in these relations are those 
for the composition of 1-cells in a 2-category and 1h-cells in a strict double category and that they all are identities). 
For the other part of a monoidal structure on $F$, namely $s_0: F(*_2)\to *_{Dbl}$, where $*_2$ is the trivial 2-category with a single object, 
and similarly $*_{Dbl}$ is the trivial double category, it is clear that we again may take identity. Now the hexagonal and two square relations 
for the monoidality of the functor $(F, s, s_0)$ concern 0-cells, and since both $s$ and $s_0$ are identities they come down to checking if 
$$F(\alpha_1)=\alpha_2, \quad F(\lambda_1)=\lambda_2, \quad\text{and}\quad F(\rho_1)=\rho_2$$
where the monoidal constraints with indexes 1 are those from $\C$ and those with indexes 2 from $\D$. 

\medskip

Given any monoidal closed category $(\M, \ot, I, \alpha, \lambda, \rho)$ in \cite[Section 4.1]{Gabi} the author constructs a mate 
\begin{equation} \eqlabel{mate a}
a^C_{A,B}: [A\ot B, C]\to[A,[B,C]]
\end{equation} 
for $\alpha$ under the adjunctions $(-\ot X, [X,-])$, for $X$ taking to be $A,B$ and $A\ot B$, 
and then she constructs a mate of $a$: 
\begin{equation} \eqlabel{l-a}
l^C_{A,B}: \big([A,B]\stackrel{[\Epsilon^C_A,1]}{\rightarrow} [[C,A]\ot C, B] \stackrel{a^B}{\rightarrow} [[C,A],[C,B]]\big)
\end{equation}
where $\Epsilon$ is the counit of the adjunction. By the mate correspondence one gets: 
\begin{equation} \eqlabel{a-l}
a^C_{A,B}: \big( [A\ot B, C]\stackrel{l^B_{A\ot B,C} }{\rightarrow} [[B,A\ot B],[B,C]] \stackrel{[\eta^B_A,1]}{\rightarrow} [A,[B,C]]\big)
\end{equation}
where $\eta$ is the unit of the adjunction. As above for the monoidal constraints, let us write $l_i, a_i, \Epsilon_i$ and $\eta_i$ 
with $i=1,2$ for the corresponding 2-functors in $\C$ (with $i=1$), respectively double functors in $\D$ (with $i=2$). 
Comparing the description of $l_1$ from \cite[Section 4.7]{Gabi}, obtained as indicated above: $\alpha_1$ determines $a_1$, which in turn determines 
$l_1$ by \equref{l-a}, to the construction of $l_2$ in \cite[Section 2.4]{Gabi}, on one hand, and the well-known 2-category $[\A,\B]=\Fun(\A,\B)$ 
of 2-functors between 2-categories $\A$ and $\B$, pseudo natural transformations and modifications (see {\em e.g.} \cite[Section 5.1]{GP:Fram}, \cite{Ben}) 
to the definition of the double category $\llbracket\Aa,\Bb\rrbracket$ from \cite[Section 2.2]{Gabi} for double categories $\Aa,\Bb$, 
on the other hand, one immediately obtains:

\begin{lma}
For two 2-categories $\A$ and $\B$, functor $F$ from \equref{id fun} and $l_1$ and $l_2$ as above, it is: 
\begin{itemize}
\item $F(l_1)=l_2$,
\item $F([\A,\B])=\llbracket F(\A),F(\B)\rrbracket$.
\end{itemize}
\end{lma}

Because of the extent of the definitions and the detailed proofs we will omit them, we only record that the counits of the adjunctions 
$\Epsilon_i, i=1,2$ are basically given as evaluations and it is $F(\Epsilon_1)=\Epsilon_2$. The counits $\eta_i, i=1,2$ are defined in the natural way 
and it is also clear that $F(\eta_1)=\eta_2$. Now by the above Lemma and \equref{a-l} we get: $F(a_1)=a_2$. Then from the next Lemma 
we get that $F(\alpha_1)=\alpha_2$:

\begin{lma}
Suppose that there is an embedding functor $F:\C\to\D$ between monoidal closed categories which fulfills:
\begin{enumerate} [a)]
\item $F(X)\ot F(Y)=F(X\ot Y)$ for objects $X,Y\in\C$,
\item $F([X,Y])=[F(X),F(Y)]$,
\item $F(a_\C)=a_\D$, where the respective  $a$'s are given through \equref{mate a},
\end{enumerate} 
then it is $F(\alpha_\C)=\alpha_\D$, being $\alpha$'s the respective associativity constraints. 
\end{lma}

\begin{proof}
By the mate construction in \equref{mate a} we have a commuting diagram:
$$
\bfig
\putmorphism(-150,50)(1,0)[\C(A\ot(B\ot C),D)`\C(A, [B\ot C,D])`\iso]{1000}1a
\putmorphism(-150,-400)(1,0)[\C((A\ot B)\ot C,D) `\C(A,[B,[C,D]]) `\iso]{1040}{-1}a
\putmorphism(-120,50)(0,-1)[\phantom{Y_2}``\C(\alpha_\C,id)]{450}1l
\putmorphism(870,50)(0,-1)[\phantom{Y_2}``\C(id,a_\C)]{450}1r
\put(300,-180){\fbox{$(1)$}}
\efig
$$
Applying $F$ to it, by the assumptions $a)$ and $b)$ we obtain a commuting diagram:
$$
\bfig
\putmorphism(-150,50)(1,0)[\D(F(A)\ot(F(B)\ot F(C)),F(D))`\D(F(A), [F(B)\ot F(C),F(D)])`\iso]{1600}1a
\putmorphism(-150,-400)(1,0)[\D((F(A)\ot F(B))\ot F(C),F(D)) `\D(F(A),[F(B),[F(C),F(D)]]) `\iso]{1640}{-1}a
\putmorphism(-80,50)(0,-1)[\phantom{Y_2}``\D(F(\alpha_\C),id)]{450}1l
\putmorphism(1490,50)(0,-1)[\phantom{Y_2}``\D(id,F(a_\C))]{450}1r
\put(580,-180){\fbox{$(2)$}}
\efig
$$
Now by the assumption $c)$ and the mate construction in \equref{mate a} it follows $F(\alpha_\C)=\alpha_\D$. 
\end{proof}

So far we have proved that for the categories $\C$ and $\D$ as in \equref{id fun} we have $F(\alpha_\C)=\alpha_\D$. 
For the unity constraints $\rho_i, \lambda_i, i=1,2$ in the cases of both categories 
(see Sections 3.3 and 4.7 of \cite{Gabi}) it is: 
$$\rho_i^A=\Epsilon_{i,A}^{1_i}\comp(c_i\ot id_{1_i})\quad\text{and}\quad\lambda_i^A=\Epsilon_{i,A}^A\comp(1_A\ot id_A)$$
where $c_i: A\to[1_i,A]$ is the canonical isomorphism and $1_A: 1_i\to[A,A]$ the 2-functor (pseudofunctor) 
sending the single object of the terminal 2-category $1_1$ (double category $1_2$) to the identity 2-functor (pseudofunctor) $A\to A$ 
(here we have used the same notation for objects $A$ and inner home objects both in $\C$ and in $\D$). Then it is clear that also 
$F(\lambda_1)=\lambda_2$ and $F(\rho_1)=\rho_2$, which finishes the proof that the functor $F:\C\to\D$ is a monoidal embedding. 
Motivated by our findings from now on we will call a monoid in $(Dbl,\ot)$ a {\em monoidal double category due to B\"ohm}.

\begin{prop} \prlabel{mon embed}
The category $(Gray,\ot)$ monoidally embeds into $(Dbl,\ot)$, where the respective monoidal structures are those from \cite{Gray} and \cite{Gabi}. 
Consequently, a monoid in $(Gray,\ot)$ is a monoid in $(Dbl,\ot)$, and a monoidal bicategory can be seen as a monoidal double category 
due to B\"ohm. 
\end{prop}

\bigskip


\subsection{A monoid in $(Dbl,\ot)$ } \sslabel{Gabi's monoid}

In \cite[Section 4.3]{Gabi} a complete list of data 
and conditions defining the structure of a monoid $\Aa$ in $(Dbl,\ot)$ is given. As a part of this structure we have the following occurrence. 
As a monoid in $(Dbl,\ot)$, we have that $\Aa$ is equipped with a strict double functor 
$M: \Aa\ot\Aa\to\Aa$. Since in the monoidal product $\Aa\ot\Aa$ 
horizontal and vertical 1-cells of the type $(f\ot 1)(1\ot g)$ and $(1\ot g)(f\ot 1)$ are not equal (here juxtaposition denotes the 
corresponding composition of the 1-cells), one can fix a choice for how to define an image 1-cell $f\oast g$ by $M$ 
(either $M\big((f\ot 1)(1\ot g)\big)$ or $M\big((1\ot g)(f\ot 1)\big)$). Any of the two choices 
yields a {\em double pseudo} functor from the {\em Cartesian} product double category 
\begin{equation} \eqlabel{oast}
\oast:\Aa\times\Aa\to\Aa.
\end{equation} 
Let us see this. 
If we take two pairs of horizontal 1-cells $(h,k),(h',k')$ in $A\times A$, for the images under $\oast$, fixing the second choice above, 
we get $(h'h)\oast(k'k)=M\big((1\ot k'k)(h'h\ot 1)\big)=
M(1\ot k')M(1\ot k)M(h'\ot 1)M(h\ot 1)$, whereas 
$(h'\oast k')(h\oast k)=M\big((1\ot k')(h'\ot 1)\big)M\big((1\ot k)(h\ot 1)\big)=M(1\ot k')M(h'\ot 1)M(1\ot k)M(h\ot 1)$. 
So, the two images differ in the flip on the middle factors. The analogous situation happens on the vertical level, thus the functor 
$\oast$ preserves both vertical and horizontal 1-cells only up to an isomorphism 2-cell. This makes it a double pseudo functor due to 
\cite[Definition 6.1]{Shul1}. 

As outlined at the end of \cite[Section 4.3]{Gabi}, monoids in (the non-Cartesian monoidal category) $(Dbl,\ot)$ are monoids in the Cartesian 
monoidal category $(Dbl,\oast)$ of strict double categories and double pseudo functors (in the sense of \cite{Shul1}).

\subsection{Monoidal double categories as intercategories and beyond} \sslabel{beyond}

A monoidal double category in \cite{Shul} is a pseudomonoid in the 2-category $PsDbl$ of pseudo double categories, pseudo double functors 
and vertical transformations, seen as a monoidal 2-category with the Cartesian product. As such it is a particular case of an intercategory \cite{GP}. 

An intercategory is a pseudocategory ({\em i.e.} weakly internal category) in the 2-category $LxDbl$ of pseudo 
double categories, lax double functors and horizontal transformations. 
It consists of pseudodouble categories $\Dd_0$ and $\Dd_1$ and pseudo double functors 
$s,t: \Dd_1\to\Dd_0, u: \Dd_0\to\Dd_1, m:\Dd_1\times_{\Dd_0}\Dd_1\to\Dd_1$ satisfying the corresponding properties. One may denote this structure formally by 
$$
 \Dd_1\times_{\Dd_0}\Dd_1\triarrows \Dd_1\tripplearrow \Dd_0
$$
where $\Dd_1\times_{\Dd_0}\Dd_1$ is a certain 2-pullback and the additional two arrows $\Dd_1\times_{\Dd_0}\Dd_1\to\Dd_1$ stand for the two projections. 
When $\Dd_0$ is the trivial double category 1 (the terminal object in $LxDbl$, consisting of a single object *), setting $\Dd_1=\Dd$ one has that 
$\Dd\times\Dd$ is the Cartesian product of pseudo double categories. 

As a pseudomonoid in $PsDbl$, a monoidal double category of Shulman consists of a pseudo double category $\Dd$ and pseudo 
double functors $m: \Dd\times\Dd\to\Dd$ and $u: 1\to\Dd$ which satisfy properties that make $\Dd$ precisely an intercategory 
$$ \Dd\times\Dd\triarrows \Dd\tripplearrow 1,$$
as explained in \cite[Section 3.1]{GP:Fram}. 

\bigskip

Motivated by \prref{mon embed} and the observation at the end of \cite[Section 4.3]{Gabi}, that there seems to be no easy way to regard a monoid in 
$(Dbl,\ot)$ as a suitably degenerated intercategory of \cite{GP}, we want now to upgrade the Cartesian monoidal category $(Dbl,\oast)$ 
from the end of \ssref{Gabi's monoid} to a 2-category, so to obtain an intercategory-type notion which would 
include monoidal double categories of B\"ohm. 

\medskip

This motivates our next section, in which we will introduce 2-cells and ``unfortunately'' rather than a 2-category we will obtain a 
tricategory of strict double categories whose 1-cells are double pseudo functors of Shulman. Since 1-cells of the 2-category $LxDbl$ 
(considered by Grandis and Par\'e to define intercategories) are lax double functors (they are lax in one and strict in the other direction),  and our 1-cells are double pseudo functors, that is, they are given by isomorphisms in both directions, we can not 
generalize intercategories this way, rather, we will propose an alternative notion to intercategories which will include 
most of the examples of intercategories treated in \cite{GP:Fram} but not duoidal categories, as they rely on lax functors, rather than pseudo ones.




\section{Tricategory of strict double categories and double pseudo functors} 

Let us denote the tricategory from the title of this section by $\DblPs$. 
As we are going to use double pseudo functors of \cite{Shul1}, which preserve compositions of 1-cells and identity 1-cells in 
both horizontal and vertical direction up to an isomorphism (and not strictly in one of the two directions as pseudo double functors 
say in \cite{Shul}), 
we have to introduce accordingly horizontal and vertical transformations. 
A pair consisting of a horizontal and a vertical pseudonatural transformation, which we define next, will be a part of the data 
constituting a 2-cell of the tricategory $\DblPs$. 
Note that while $PsDbl$ usually denotes a category or a 2-category of {\em pseudo} double categories and pseudo double functors, 
that is, in which 0- and 1-cells are weakened, in the notation $\DblPs$ we wish to stress that both 1- and 2-cells are weakened in both 
directions so to deal with {\em double pseudo} functors.

\subsection{Towards the 2-cells}

For the structure of a double pseudo functor we use the same notation as in \cite[Definition 6.1]{Shul1} with the only difference 
that 0-cells we denote by $A,B...$ and 1v-cells by $u,v...$. To simplify the notation, we will denote by juxtaposition the 
compositions of both 1h- and 1v-cells, from the notation of the 1-cells it will be clear which kind of 1-cells and therefore 
composition is meant. Let $\Aa,\Bb,\Cc$ be strict double categories throughout. 

\begin{defn} \delabel{hor psnat tr}
A {\em horizontal pseudonatural transformation} between double pseudo functors $F,G: \Aa\to\Bb$ consists of the following:
\begin{itemize}
\item for every 0-cell $A$ in $\Aa$ a 1h-cell $\alpha(A):F(A)\to G(A)$ in $\Bb$,
\item for every 1v-cell $u:A\to A'$ in $\Aa$ a 2-cell in $\Bb$:
$$
\bfig
\putmorphism(-150,50)(1,0)[F(A)`G(A)`\alpha(A)]{560}1a
\putmorphism(-150,-320)(1,0)[F(A')`G(A')`\alpha(A')]{600}1a
\putmorphism(-180,50)(0,-1)[\phantom{Y_2}``F(u)]{370}1l
\putmorphism(410,50)(0,-1)[\phantom{Y_2}``G(u)]{370}1r
\put(30,-110){\fbox{$\alpha_u$}}
\efig
$$
\item 
for every 1h-cell $f:A\to B$  in $\Aa$ there is a 2-cell in $\Bb$:
$$
\bfig
 \putmorphism(-170,500)(1,0)[F(A)`F(B)`F(f)]{540}1a
 \putmorphism(360,500)(1,0)[\phantom{F(f)}`G(B) `\alpha(B)]{560}1a
 \putmorphism(-170,120)(1,0)[F(A)`G(A)`\alpha(A)]{540}1a
 \putmorphism(360,120)(1,0)[\phantom{G(B)}`G(A) `G(f)]{560}1a
\putmorphism(-180,500)(0,-1)[\phantom{Y_2}``=]{380}1r
\putmorphism(940,500)(0,-1)[\phantom{Y_2}``=]{380}1r
\put(280,310){\fbox{$\delta_{\alpha,f}$}}
\efig
$$
\end{itemize}
so that the following are satisfied:
\begin{enumerate}
\item pseudonaturality of 2-cells:
for every 2-cell in $\Aa$
$\bfig
\putmorphism(-150,50)(1,0)[A` B`f]{400}1a
\putmorphism(-150,-270)(1,0)[A'`B' `g]{400}1b
\putmorphism(-170,50)(0,-1)[\phantom{Y_2}``u]{320}1l
\putmorphism(250,50)(0,-1)[\phantom{Y_2}``v]{320}1r
\put(0,-140){\fbox{$a$}}
\efig$ 
the following identity in $\Bb$ must hold:
$$
\bfig
\putmorphism(-150,500)(1,0)[F(A)`F(B)`F(f)]{600}1a
 \putmorphism(450,500)(1,0)[\phantom{F(A)}`G(B) `\alpha(B)]{640}1a

 \putmorphism(-150,50)(1,0)[F(A')`F(B')`F(g)]{600}1a
 \putmorphism(450,50)(1,0)[\phantom{F(A)}`G(B') `\alpha(B')]{640}1a

\putmorphism(-180,500)(0,-1)[\phantom{Y_2}``F(u)]{450}1l
\putmorphism(450,500)(0,-1)[\phantom{Y_2}``]{450}1r
\putmorphism(300,500)(0,-1)[\phantom{Y_2}``F(v)]{450}0r
\putmorphism(1100,500)(0,-1)[\phantom{Y_2}``G(v)]{450}1r
\put(0,260){\fbox{$F(a)$}}
\put(700,270){\fbox{$\alpha_v$}}

\putmorphism(-150,-400)(1,0)[F(A')`G(A') `\alpha(A')]{640}1a
 \putmorphism(450,-400)(1,0)[\phantom{A'\ot B'}` G(B') `G(g)]{680}1a

\putmorphism(-180,50)(0,-1)[\phantom{Y_2}``=]{450}1l
\putmorphism(1120,50)(0,-1)[\phantom{Y_3}``=]{450}1r
\put(320,-200){\fbox{$\delta_{\alpha,g}$}}

\efig
\quad=\quad
\bfig
\putmorphism(-150,500)(1,0)[F(A)`F(B)`F(f)]{600}1a
 \putmorphism(450,500)(1,0)[\phantom{F(A)}`G(B) `\alpha(B)]{680}1a
 \putmorphism(-150,50)(1,0)[F(A)`G(A)`\alpha(A)]{600}1a
 \putmorphism(450,50)(1,0)[\phantom{F(A)}`G(B) `G(f)]{680}1a

\putmorphism(-180,500)(0,-1)[\phantom{Y_2}``=]{450}1r
\putmorphism(1100,500)(0,-1)[\phantom{Y_2}``=]{450}1r
\put(350,260){\fbox{$\delta_{\alpha,f}$}}
\put(650,-180){\fbox{$G(a)$}}

\putmorphism(-150,-400)(1,0)[F(A')`G(A') `\alpha(A')]{640}1a
 \putmorphism(490,-400)(1,0)[\phantom{F(A')}` G(B') `G(g)]{640}1a

\putmorphism(-180,50)(0,-1)[\phantom{Y_2}``F(u)]{450}1l
\putmorphism(450,50)(0,-1)[\phantom{Y_2}``]{450}1l
\putmorphism(610,50)(0,-1)[\phantom{Y_2}``G(u)]{450}0l 
\putmorphism(1120,50)(0,-1)[\phantom{Y_3}``G(v)]{450}1r
\put(40,-180){\fbox{$\alpha_u$}} 

\efig
$$

\item vertical functoriality: for any composable 1v-cells $u$ and $v$ in $\Aa$:
$$
\bfig
 \putmorphism(-150,500)(1,0)[F(A)`F(A) `=]{500}1a
 \putmorphism(450,500)(1,0)[` `\alpha(A)]{380}1a
\putmorphism(-180,500)(0,-1)[\phantom{Y_2}`F(A') `F(u)]{450}1l
\put(20,250){\fbox{$F^{vu}$}}
\putmorphism(-150,-400)(1,0)[F(A'')`F(A'') `=]{500}1a
\putmorphism(-180,50)(0,-1)[\phantom{Y_2}``F(v)]{450}1l
\putmorphism(380,500)(0,-1)[\phantom{Y_2}` `F(vu)]{900}1l
\put(600,45){\fbox{$\alpha^{vu}$} }
\putmorphism(940,500)(0,-1)[G(A)`G(A'')`G(vu)]{900}1r
\putmorphism(450,-400)(1,0)[``\alpha(A'')]{360}1a
\efig= 
\bfig
 \putmorphism(-150,500)(1,0)[F(A)`G(A)  `\alpha(A)]{600}1a
 \putmorphism(450,500)(1,0)[\phantom{(B,A)}` `=]{450}1a
\putmorphism(-180,500)(0,-1)[\phantom{Y_2}`F(A') `F(u)]{450}1l
\put(620,50){\fbox{$G^{vu}$}}
\putmorphism(-150,-400)(1,0)[F(A'')` `\alpha(A'')]{500}1a
\putmorphism(-180,50)(0,-1)[\phantom{Y_2}``F(v)]{450}1l
\putmorphism(450,50)(0,-1)[\phantom{Y_2}`G(A'')`G(v)]{450}1r
\putmorphism(450,500)(0,-1)[\phantom{Y_2}`G(A') `G(u)]{450}1r
\put(40,260){\fbox{$\alpha_u$}}
\putmorphism(-150,50)(1,0)[\phantom{(B, \tilde A)}``\alpha(A')]{500}1a
\putmorphism(1000,500)(0,-1)[G(A)`G(A'')`G(vu)]{900}1r
\putmorphism(480,-400)(1,0)[\phantom{F(A)}`\phantom{F(A)}`=]{500}1b
\put(40,-170){\fbox{$\alpha_v$}}
\efig
$$
and  
$$
\bfig
 \putmorphism(-150,250)(1,0)[F(A)`F(A)`=]{600}1a 
 \putmorphism(450,250)(1,0)[\phantom{A\ot B}`G(A) `\alpha(A)]{680}1a
\putmorphism(500,250)(0,-1)[\phantom{Y_2}``F(id_A)]{450}1l

 \putmorphism(-150,-200)(1,0)[F(A)`F(A)`=]{600}1a 
 \putmorphism(450,-200)(1,0)[\phantom{A\ot B}`G(A) `\alpha(A)]{680}1a
\putmorphism(-180,250)(0,-1)[\phantom{Y_2}``=]{450}1r
\putmorphism(1100,250)(0,-1)[\phantom{Y_2}``G(id_A)]{450}1r
\put(0,0){\fbox{$F^A$}}
\put(700,30){\fbox{$\alpha_{id_A}$}}
\efig
\quad=
\bfig
 \putmorphism(-100,250)(1,0)[F(A)`G(A)`\alpha(A)]{550}1a
 \putmorphism(450,250)(1,0)[\phantom{A\ot B}`G(A) `=]{680}1a
\putmorphism(500,250)(0,-1)[\phantom{Y_2}``=]{450}1l

 \putmorphism(-150,-200)(1,0)[F(A)`G(A)`\alpha(A)]{600}1a
 \putmorphism(450,-200)(1,0)[\phantom{F( B)}`G(A) `=]{680}1a 
\putmorphism(-180,250)(0,-1)[\phantom{Y_2}``=]{450}1r
\putmorphism(1100,250)(0,-1)[\phantom{Y_2}``G(id_A)]{450}1r
\put(0,0){\fbox{$\Id_{\alpha(A)}$}}
\put(700,30){\fbox{$G^A$}}
\efig
$$

\item horizontal functoriality for $\delta_{\alpha,-}$: for any composable 1h-cells $f$ and $g$ in $\Aa$ the 2-cell 
$\delta_{\alpha,gf}$ is given by: 
$$
\bfig
 \putmorphism(-130,500)(1,0)[F(A)`F(C)`F(gf)]{580}1a
 \putmorphism(450,500)(1,0)[\phantom{F(B)}`G(C) `\alpha(C)]{580}1a
 \putmorphism(-130,135)(1,0)[F(A)`G(A)`\alpha(A)]{580}1a
 \putmorphism(450,135)(1,0)[\phantom{F(B)}`G(C) `G(gf)]{580}1a
\putmorphism(-180,520)(0,-1)[\phantom{Y_2}``=]{380}1r
\putmorphism(1030,520)(0,-1)[\phantom{Y_2}``=]{380}1r
\put(330,300){\fbox{$\delta_{\alpha,gf}$}}
\efig= 
\bfig
\putmorphism(-150,900)(1,0)[F(A)`F(C)`F(gf)]{1200}1a

 \putmorphism(-150,450)(1,0)[F(A)`F(B)`F(f)]{600}1a
 \putmorphism(450,450)(1,0)[\phantom{F(B)}`F(C) `F(g)]{680}1a
 \putmorphism(1120,450)(1,0)[\phantom{F(B)}`G(C) `\alpha(C)]{600}1a

\putmorphism(-180,900)(0,-1)[\phantom{Y_2}``=]{450}1r
\putmorphism(1060,900)(0,-1)[\phantom{Y_2}``=]{450}1r
\put(350,650){\fbox{$F_{gf}$}}
\put(1000,200){\fbox{$\delta_{\alpha,g}$}}

  \putmorphism(-150,0)(1,0)[F(A)` F(B) `F(f)]{600}1a
\putmorphism(450,0)(1,0)[\phantom{F(A)}` G(B) `\alpha(B)]{680}1a
 \putmorphism(1120,0)(1,0)[\phantom{F(A)}`G(C) ` G(g)]{620}1a

\putmorphism(450,450)(0,-1)[\phantom{Y_2}``=]{450}1l
\putmorphism(1710,450)(0,-1)[\phantom{Y_2}``=]{450}1r

 \putmorphism(-150,-450)(1,0)[F(A)`G(A)`\alpha(A)]{600}1a
 \putmorphism(450,-450)(1,0)[\phantom{F(B)}`G(B) `G(f)]{680}1a
 \putmorphism(1120,-450)(1,0)[\phantom{F(B)}`G(C) `G(g)]{620}1a

\putmorphism(-180,0)(0,-1)[\phantom{Y_2}``=]{450}1r
\putmorphism(1040,0)(0,-1)[\phantom{Y_2}``=]{450}1r
\put(350,-240){\fbox{$\delta_{\alpha,f}$}}
\put(1000,-660){\fbox{$G^{-1}_{gf}$}}

 \putmorphism(450,-900)(1,0)[G(A)` G(C) `G(gf)]{1300}1a

\putmorphism(450,-450)(0,-1)[\phantom{Y_2}``=]{450}1l
\putmorphism(1750,-450)(0,-1)[\phantom{Y_2}``=]{450}1r
\efig
$$ 
and unit:
$$
\bfig
 \putmorphism(-150,420)(1,0)[F(A)`F(A)`=]{500}1a
\putmorphism(-180,420)(0,-1)[\phantom{Y_2}``=]{370}1l
\putmorphism(320,420)(0,-1)[\phantom{Y_2}``=]{370}1r
 \putmorphism(-150,50)(1,0)[F(A)`F(A)`F(id_A)]{500}1a
 \put(-80,250){\fbox{$(F_A)^{-1}$}} 
\putmorphism(330,50)(1,0)[\phantom{F(A)}`G(A) `\alpha(A)]{560}1a
 \putmorphism(-170,-350)(1,0)[F(A)`G(A)`\alpha(A)]{520}1a
 \putmorphism(350,-350)(1,0)[\phantom{F(A)}`G(A) `G(id_A)]{560}1a

\putmorphism(-180,50)(0,-1)[\phantom{Y_2}``=]{400}1r
\putmorphism(910,50)(0,-1)[\phantom{Y_2}``=]{400}1r
\put(240,-170){\fbox{$\delta_{\alpha,id_A}$}}
\put(540,-550){\fbox{$G_A$}}

\putmorphism(350,-350)(0,-1)[\phantom{Y_2}``=]{350}1l
\putmorphism(920,-350)(0,-1)[\phantom{Y_3}``=]{350}1r
 \putmorphism(350,-700)(1,0)[G(A)` G(A) `=]{580}1a
\efig
\quad=\quad
\bfig
\putmorphism(-150,50)(1,0)[F(A)` G(A) `\alpha(A)]{450}1a
\putmorphism(-150,-300)(1,0)[F(A)` G(A) `\alpha(A)]{440}1b
\putmorphism(-170,50)(0,-1)[\phantom{Y_2}``=]{350}1l
\putmorphism(280,50)(0,-1)[\phantom{Y_2}``=]{350}1r
\put(-80,-140){\fbox{$\Id_{\alpha(A)}$}}
\efig
$$
\end{enumerate}
\end{defn}

\begin{rem} \rmlabel{hor tr as subcase}
Recall horizontal transformations from \cite[Section 2.2]{GP:Adj} and their version when $R=S=\Id, 
\Aa=\Bb, \Cc=\Dd$, which constitute the 2-cells of the 2-category $\mathcal{L}x\mathcal{D}bl$ from \cite{GP}. Considering them 
as acting between pseudo double functors of strict double categories, one has that they are particular cases of our 
horizontal pseudonatural transformations so that the 2-cells $\delta_{\alpha,f}$ and $F_{gf}, F_A$ are identities. 
Similarly, vertical transformations from \cite{Shul} and \cite{GG} are particular cases of the vertical analogon of 
\deref{hor psnat tr}. 
\end{rem}

\begin{rem}
Our above definition generalizes also horizontal pseudotransformations from \cite[Section 2.2]{Gabi} to the case of {\em double 
pseudo functors} instead of strict double functors. In \cite[Section 2.2]{Gabi} horizontal pseudotransformations appear 
as 1h-cells of the double category $\llbracket\Aa,\Bb\rrbracket$ defined therein, which we mentioned in \ssref{mon str}. 
On the other hand, our definition of horizontal pseudotransformations differs from {\em strong horizontal transformations} from 
\cite[Section 7.4]{GP:Limits} in that therein the authors work with {\em pseudo} double categories (whereas we work here with strict ones), 
and they work with double functors which are lax in one direction and strict in the other, whereas we work with double functors 
which are pseudo in both directions. 
\end{rem}

Given a 1h-cell $f:A\to B$ and two horizontally composable horizontal pseudonatral transformations $\alpha_1$ and $\beta_1$, 
then setting in the axiom 1) of the above definition $u=v=id_A$ and $a=\delta_{\alpha_1,f}$ (and using axiom 2) ) one gets the identity: 
$$
\bfig

 \putmorphism(-650,0)(1,0)[F\s'F(A)` F\s'G(B)` F\s'\big(\alpha_1(B)F(f)\big)]{1160}1a
 \putmorphism(640,0)(1,0)[` G'G(B) `\beta_1(G(B))]{840}1a

\putmorphism(-730,0)(0,-1)[\phantom{Y_2}``=]{450}1r
\putmorphism(1600,0)(0,-1)[\phantom{Y_2}``=]{450}1r

 \putmorphism(-640,-450)(1,0)[F\s'F(A) ` G'F(A) ` ]{700}1a
 \putmorphism(-620,-420)(1,0)[` `\beta_1(F(A)) ]{700}0a

 \putmorphism(-20,-450)(1,0)[\phantom{F''G''G'(A)} ` `= ]{570}1a 
 \putmorphism(50,-450)(1,0)[\phantom{F''G''G'(A)} ` G'F(A) ` ]{640}0a
	
\putmorphism(840,-450)(1,0)[ `G'G(B) `]{780}1a
 \putmorphism(740,-430)(1,0)[ ``G'\big(\alpha_1(B)F(f)\big)]{780}0a
	
 \putmorphism(1620,-450)(1,0)[\phantom{A'\ot B'}` G'G(B) ` =]{720}1a

\put(40,-220){\fbox{$ \delta_{\beta_1,\alpha_1(B)F(f)}$  }}
\put(960,-690){\fbox{$G'(\delta_{\alpha_1,f})$}}

\putmorphism(50,-420)(0,-1)[\phantom{Y_2}``=]{450}1l
\putmorphism(750,-420)(0,-1)[``]{450}1l
\putmorphism(780,-420)(0,-1)[\phantom{Y_2}``G'(id)]{450}0l
\putmorphism(1600,-420)(0,-1)[\phantom{Y_2}``]{450}1r
\putmorphism(1570,-420)(0,-1)[\phantom{Y_2}``G'(id)]{450}0r
\putmorphism(2300,-420)(0,-1)[\phantom{Y_2}``=]{450}1r
\put(130,-690){\fbox{$ G'^\bullet$ }}
\put(1880,-690){\fbox{$  (G'^\bullet)^{-1} $}}

 \putmorphism(30,-880)(1,0)[G'F(A)` G'F(A) ` = ]{670}1b
 \putmorphism(690,-880)(1,0)[\phantom{A''\ot B'}`G'G(B) ` G'(G(f)\alpha_1(A)) ]{930}1b
 \putmorphism(1600,-880)(1,0)[\phantom{A''\ot B'}`G'G(B) ` = ]{740}1b
\efig
$$
\begin{equation} \eqlabel{Axiom 4}
\end{equation}
=
$$
\bfig

 \putmorphism(-650,0)(1,0)[F\s'F(A)` F\s'F(A)` =]{660}1a
 \putmorphism(150,0)(1,0)[`F\s'G(B) ` F\s'\big(\alpha_1(B)F(f)\big)]{1030}1a
 \putmorphism(1340,0)(1,0)[` F\s'G(B) `=]{620}1a

\putmorphism(-730,0)(0,-1)[\phantom{Y_2}``=]{450}1r
\putmorphism(50,0)(0,-1)[\phantom{Y_2}``F\s'(id)]{450}1l
\putmorphism(1150,0)(0,-1)[``]{450}1r
\putmorphism(1130,0)(0,-1)[\phantom{Y_2}``F\s'(id)]{450}0r
\putmorphism(2000,0)(0,-1)[\phantom{Y_2}``=]{450}1r

\put(-500,-250){\fbox{$ F\s'^\bullet$ }}
\put(1480,-250){\fbox{$  (F\s'^\bullet)^{-1} $}}

 \putmorphism(-650,-450)(1,0)[F\s'F(A) ` F\s'F(A) `= ]{700}1a

 \putmorphism(-50,-450)(1,0)[\phantom{F''G''G'(A)} ` `F\s'(G(f)\alpha_1(A)) ]{1080}1b
 \putmorphism(500,-450)(1,0)[\phantom{F''G''G'(A)} ` F\s'G(B) ` ]{680}0a
	
\putmorphism(1350,-450)(1,0)[ `F\s'G(B) `=]{630}1b
	
 \putmorphism(1980,-450)(1,0)[\phantom{A'\ot B'}` G'G(B) ` \beta_1(G(B))]{800}1b

\put(370,-220){\fbox{$  F\s'(\delta_{\alpha_1,f}) $}}
\put(1280,-690){\fbox{$ \delta_{\beta_1,G(f)\alpha_1(A)}$ }}

\putmorphism(50,-420)(0,-1)[\phantom{Y_2}``=]{450}1l
\putmorphism(2800,-440)(0,-1)[\phantom{Y_2}``=]{450}1r

 \putmorphism(30,-880)(1,0)[F\s'F(A)` G\s'F(A) ` \beta_1(F(A)) ]{1300}1b
 \putmorphism(1320,-880)(1,0)[\phantom{A''\ot B'}`G'G(B). ` G'(G(f)\alpha_1(A)) ]{1440}1b
\efig
$$

In particular, given three horizontally composable horizontal pseudonatural transformations $\alpha_1, \beta_1, \gamma_1$ 
(see \leref{horiz comp hor.ps.tr.} below for the definition of this composition), 
substituting: $f\mapsto \alpha_1(A), F\mapsto F\s', F\s'\mapsto F'', B\mapsto G(A), \alpha_1\mapsto\beta_1, \beta_1\mapsto\gamma_1, 
G\mapsto G', A\mapsto F(A), G'\mapsto G''$ one gets to:
$$
\bfig

 \putmorphism(-650,0)(1,0)[F''F\s'F(A)` F''G''G'(A)` F''\big((\beta_1\comp\alpha_1)(A)\big)]{1160}1a
 \putmorphism(740,0)(1,0)[` G''G''G'(A) `\gamma_1(G'G(A))]{760}1a

\putmorphism(-730,0)(0,-1)[\phantom{Y_2}``=]{450}1r
\putmorphism(1600,0)(0,-1)[\phantom{Y_2}``=]{450}1r

 \putmorphism(-640,-450)(1,0)[F''F\s'F(A) ` G''F\s'F(A) ` ]{700}1a
 \putmorphism(-620,-420)(1,0)[` `\gamma_1(F\s'F(A)) ]{700}0a

 \putmorphism(20,-450)(1,0)[\phantom{F''G''G'(A)} ` `= ]{470}1a 
 \putmorphism(50,-450)(1,0)[\phantom{F''G''G'(A)} ` G''F\s'F(A) ` ]{640}0a
	
\putmorphism(880,-450)(1,0)[ `G''G'G(A) `]{740}1a
 \putmorphism(710,-430)(1,0)[ ``G''\big((\beta_1\comp\alpha_1)(A)\big)]{740}0a
	
 \putmorphism(1670,-450)(1,0)[\phantom{A'\ot B'}` G''G'G(A) ` =]{660}1a

\put(40,-220){\fbox{$ \delta_{\gamma_1,(\beta_1\comp\alpha_1)(A)}$  }}
\put(880,-690){\fbox{$G''(\delta_{\beta_1,\alpha_1(A)})$}}

\putmorphism(50,-420)(0,-1)[\phantom{Y_2}``=]{450}1l
\putmorphism(750,-420)(0,-1)[``]{450}1l
\putmorphism(780,-420)(0,-1)[\phantom{Y_2}``G''(id)]{450}0l
\putmorphism(1600,-420)(0,-1)[\phantom{Y_2}``]{450}1r
\putmorphism(1570,-420)(0,-1)[\phantom{Y_2}``G''(id)]{450}0r
\putmorphism(2300,-420)(0,-1)[\phantom{Y_2}``=]{450}1r
\put(130,-690){\fbox{$ G''^\bullet$ }}
\put(1880,-690){\fbox{$  (G''^\bullet)^{-1}$ }}

 \putmorphism(30,-880)(1,0)[G''F\s'F(A)` G''F\s'G(A) ` = ]{670}1b
 \putmorphism(730,-880)(1,0)[\phantom{A''\ot B'}`G''G'G(A) ` G''(\crta{(\beta_1\comp\alpha_1)(A)}) ]{890}1b
 \putmorphism(1680,-880)(1,0)[\phantom{A''\ot B'}`G''G'G(A) ` = ]{660}1b
\efig
$$
\begin{equation} \eqlabel{3d's}
\end{equation}
=
$$
\bfig

 \putmorphism(-650,0)(1,0)[F''F\s'F(A)` F''F\s'F(A)` =]{660}1a
 \putmorphism(200,0)(1,0)[`F''G'G(A) ` F''\big((\beta_1\comp\alpha_1)(A)\big)]{980}1a
 \putmorphism(1400,0)(1,0)[` F''G'G(A) `=]{560}1a

\putmorphism(-730,0)(0,-1)[\phantom{Y_2}``=]{450}1r
\putmorphism(50,0)(0,-1)[\phantom{Y_2}``F''(id)]{450}1l
\putmorphism(1150,0)(0,-1)[``]{450}1r
\putmorphism(1130,0)(0,-1)[\phantom{Y_2}``F''(id)]{450}0r
\putmorphism(2000,0)(0,-1)[\phantom{Y_2}``=]{450}1r

\put(-500,-250){\fbox{$ F''^\bullet$ }}
\put(1480,-250){\fbox{$  (F''^\bullet)^{-1} $}}

 \putmorphism(-650,-450)(1,0)[F''F\s'F(A) ` F''F\s'F(A) `= ]{700}1a

 \putmorphism(20,-450)(1,0)[\phantom{F''G''G'(A)} ` `F''(\crta{(\beta_1\comp\alpha_1)(A)}) ]{920}1b
 \putmorphism(500,-450)(1,0)[\phantom{F''G''G'(A)} ` F''G\s'G(A) ` ]{680}0a
	
\putmorphism(1380,-450)(1,0)[ `F''G'G(A) `=]{590}1b
	
 \putmorphism(2020,-450)(1,0)[\phantom{A'\ot B'}` G''G'G(A) ` \gamma_1(G'G(A))]{760}1b

\put(340,-220){\fbox{$  F''(\delta_{\beta_1,\alpha_1(A)})$ }}
\put(1280,-690){\fbox{$ \delta_{\gamma_1,(\beta_1\comp\alpha_1)(A)}$ }}

\putmorphism(50,-420)(0,-1)[\phantom{Y_2}``=]{450}1l
\putmorphism(2800,-440)(0,-1)[\phantom{Y_2}``=]{450}1r

 \putmorphism(30,-880)(1,0)[F''F\s'F(A)` G''F\s'F(A) ` \gamma_1(F\s'F(A)) ]{1300}1b
 \putmorphism(1380,-880)(1,0)[\phantom{A''\ot B'}`G''G'G(A) ` G''(\crta{(\beta_1\comp\alpha_1)(A)}) ]{1380}1b
\efig
$$
where $\crta{(\beta_1\comp\alpha_1)(A)}=\crta{\beta_1(G(A))\comp F\s'(\alpha_1(A))} = G'(\alpha_1(A))\comp\beta_1(F(A));$
observe that the top 1h-cell in the above identity is $\gamma_1\comp(\beta_1\comp\alpha_1)$.

\begin{lma} \lelabel{delta F(alfa)}
For a double pseudofunctor $H:\Bb\to\Cc$ and a horizontal pseudonatural transformation $\alpha: F\to G$ of double pseudofunctors 
$F,G: \Aa\to\Bb$, $H(\alpha)$ is a horizontal pseudonatural transformation with $(H(\alpha))_u=H(\alpha_u)$ and $\delta_{H(\alpha),f}$ 
satisfying:
$$
\bfig
\putmorphism(-150,500)(1,0)[HF(A)`HG(B)`H(\alpha(B)F(f))]{1240}1a

 \putmorphism(-150,50)(1,0)[HF(A)`HF(B)`HF(f)]{600}1a
 \putmorphism(450,50)(1,0)[\phantom{F(A)}`HG(B) `H(\alpha(B))]{640}1a

\putmorphism(-180,500)(0,-1)[\phantom{Y_2}``=]{450}1l
\putmorphism(1100,500)(0,-1)[\phantom{Y_2}``=]{450}1r
\put(320,270){\fbox{$H_{\alpha(B),F(f)}$}}

\putmorphism(-150,-400)(1,0)[HF(A)`HG(A) `H(\alpha(A))]{640}1a
 \putmorphism(450,-400)(1,0)[\phantom{A'\ot B'}` HG(B) `HG(f)]{680}1a

\putmorphism(-180,50)(0,-1)[\phantom{Y_2}``=]{450}1l
\putmorphism(1120,50)(0,-1)[\phantom{Y_3}``=]{450}1r
\put(320,-200){\fbox{$\delta_{H(\alpha),f}$}}

\efig
\quad=\quad
\bfig
\putmorphism(-580,500)(0,-1)[\phantom{Y_2}``=]{450}1l
\putmorphism(1520,500)(0,-1)[\phantom{Y_2}``=]{450}1r
\putmorphism(-550,500)(1,0)[HF(A)`\phantom{HF(A)}`=]{550}1a
\putmorphism(-550,50)(1,0)[HF(A)`\phantom{HF(A)}`=]{550}1a
\put(-500,260){\fbox{$H^{F(A)}$}}
\put(1060,260){\fbox{$(H^{G(B)})^{-1}$}}

 \putmorphism(980,500)(1,0)[\phantom{HF(A)}`HG(B) `=]{550}1a
 \putmorphism(1000,50)(1,0)[\phantom{HF(A)}`HG(B) `=]{550}1a

\putmorphism(-20,500)(1,0)[HF(A)`HG(B)`H(\alpha(B)F(f))]{1000}1a
 \putmorphism(-20,50)(1,0)[HF(A)`HG(B)`H\big(G(f)\alpha(A)\big)]{1000}1a

\putmorphism(-80,500)(0,-1)[\phantom{Y_2}``]{450}1r
\putmorphism(-100,500)(0,-1)[\phantom{Y_2}``H(id)]{450}0r

\putmorphism(1020,500)(0,-1)[\phantom{Y_2}``]{450}1l
\putmorphism(1050,500)(0,-1)[\phantom{Y_2}``H(id)]{450}0l

\put(300,310){\fbox{$H(\delta_{\alpha,f})$}}
\put(300,-180){\fbox{$H_{G(f),\alpha(A)}$}}

\putmorphism(-150,-400)(1,0)[HF(A)`HG(A) `H(\alpha(A))]{600}1a
 \putmorphism(450,-400)(1,0)[\phantom{HF(A)}` HG(B). `HG(f)]{600}1a

\putmorphism(-80,50)(0,-1)[\phantom{Y_2}``=]{450}1l
\putmorphism(1020,50)(0,-1)[\phantom{Y_3}``=]{450}1r

\efig
$$
\end{lma}

In view of the above Lemma, the identity \equref{Axiom 4} is equivalent to:
$$
\bfig
 \putmorphism(450,450)(1,0)[F\s'F(B)` F\s'G(B) `F\s'(\alpha_1(B))]{1100}1a
 \putmorphism(1590,450)(1,0)[\phantom{F(B)}`G'G(B) `\beta_1(G(B))]{960}1a

\put(1300,220){\fbox{$\delta_{\beta_1,\alpha_1(B)}$}}

  \putmorphism(-550,0)(1,0)[F\s'F(A) ` F\s'F(B) `F\s'F(f)]{1000}1a
\putmorphism(480,0)(1,0)[\phantom{F(A)}` G\s'F(B) `\beta_1(F(B))]{1070}1a
 \putmorphism(1600,0)(1,0)[\phantom{F(A)}` G'G(B) ` G'(\alpha_1(B))]{960}1a

\putmorphism(450,450)(0,-1)[\phantom{Y_2}``=]{450}1l
\putmorphism(2510,450)(0,-1)[\phantom{Y_2}``=]{450}1r

 \putmorphism(-550,-450)(1,0)[F\s'F(A) ` G'F(A)` \beta_1(F(A))]{1000}1a
 \putmorphism(500,-450)(1,0)[\phantom{F(B)}` G'F(B) ` G'F(f)]{1050}1a

\putmorphism(-550,0)(0,-1)[\phantom{Y_2}``=]{450}1r
\putmorphism(1540,0)(0,-1)[\phantom{Y_2}``=]{450}1r
\putmorphism(2510,0)(0,-1)[\phantom{Y_2}``=]{900}1r
\put(250,-240){\fbox{$\delta_{\beta_1,F(f)}$}}
\put(1780,-620){\fbox{$\delta_{G'(\alpha_1),f}$}}

 \putmorphism(550,-900)(1,0)[G'F(A)` G'G(A) ` G'(\alpha_1(A))]{1000}1a
 \putmorphism(1500,-900)(1,0)[\phantom{G''G'F(A)}` G'G(B) ` G'G(f)]{1060}1a

\putmorphism(450,-450)(0,-1)[\phantom{Y_2}``=]{450}1l
\efig
$$
\begin{equation}\eqlabel{Cor}
\end{equation}
=
$$
\bfig
   \putmorphism(-550,450)(1,0)[F\s'F(A) ` F\s'F(B) `F\s'F(f)]{1000}1a
\putmorphism(590,450)(1,0)[` F\s'G(B) `F\s'(\alpha_1(B))]{960}1a
\putmorphism(-550,450)(0,-1)[\phantom{Y_2}``=]{450}1r
\putmorphism(1540,450)(0,-1)[\phantom{Y_2}``=]{450}1l

\put(250,220){\fbox{$\delta_{F\s'(\alpha_1), f}$}}

  \putmorphism(-550,0)(1,0)[F\s'F(A) ` F\s'G(A) `F\s'(\alpha_1(A))]{1000}1a
\putmorphism(490,0)(1,0)[\phantom{F(A)}` F\s'G(B) `F\s'G(f)]{1060}1a 
 \putmorphism(1590,0)(1,0)[\phantom{F(A)}` G'G(B) ` \beta_1(G(B))]{960}1a

 \putmorphism(450,-450)(1,0)[F\s'G(A) ` G'G(A)` \beta_1(G(A))]{1000}1a 
 \putmorphism(1520,-450)(1,0)[\phantom{F(B)}` G'G(B) ` G'G(f)]{1040}1a

\putmorphism(-550,0)(0,-1)[\phantom{Y_2}``=]{900}1r
\putmorphism(450,0)(0,-1)[\phantom{Y_2}``=]{450}1r
\putmorphism(2510,0)(0,-1)[\phantom{Y_2}``=]{450}1r

\put(1340,-240){\fbox{$\delta_{\beta_1,G(f)}$}}
\put(-300,-620){\fbox{$\delta_{\beta_1, \alpha_1(A)}$}}

 \putmorphism(-550,-900)(1,0)[F\s'F(A) `\phantom{G''G'F(A)} ` \beta_1(F(A))]{1060}1a
 \putmorphism(450,-900)(1,0)[G\s'F(A)` G'G(A). ` G'(\alpha_1(A))]{1000}1a

\putmorphism(1550,-450)(0,-1)[\phantom{Y_2}``=]{450}1l
\efig
$$

We outlined the above identities \equref{3d's} and \equref{Cor} for the following reason. The horizontal composition of the 
2-cell components $t^\alpha_f$ of {\em double pseudonatural transformations} from \deref{double 2-cells}
is defined in \prref{horiz comp 2-cells}. In the proof that this 
horizontal composition of $t$'s  satisfies axiom (T3-1) of \deref{double 2-cells} the vertical version of the identity \equref{Cor} 
is used. In order for it to be associative the vertical version of the identity \equref{3d's} is used.

\begin{lma} \lelabel{horiz comp hor.ps.tr.}
Horizontal composition of two horizontal pseudonatural transformations $\alpha_1: F\Rightarrow G: \Aa\to\Bb$ and $\beta_1: F\s'\Rightarrow G':\Bb\to\Cc$, denoted by $\beta_1\comp\alpha_1$, is well-given by: 
\begin{itemize}
\item for every 0-cell $A$ in $\Aa$ a 1h-cell in $\Cc$:
$$(\beta_1\comp\alpha_1)(A)=\big( F\s'F(A)\stackrel{F\s'(\alpha_1(A))}{\longrightarrow}F\s'G(A) \stackrel{\beta_1(G(A))}{\longrightarrow} G'G(A) \big),$$ 
\item for every 1v-cell $u:A\to A'$ in $\Aa$ a 2-cell in $\Cc$:
$$(\beta_1\comp\alpha_1)_u=
\bfig
\putmorphism(-320,500)(1,0)[F\s'F(A)`F\s'G(A)`F\s'(\alpha_1(A))]{770}1a
 \putmorphism(500,500)(1,0)[\phantom{F(A)}`G'G(A) `\beta_1(G(A))]{720}1a

 \putmorphism(-320,50)(1,0)[F\s'F(A')`F\s'G(A')`F\s'(\alpha_1(A'))]{770}1a
 \putmorphism(520,50)(1,0)[\phantom{F(A)}`G'G(A') `\beta_1(G(A'))]{730}1a

\putmorphism(-280,500)(0,-1)[\phantom{Y_2}``F\s'F(u)]{450}1l
\putmorphism(450,500)(0,-1)[\phantom{Y_2}``]{450}1r
\putmorphism(300,500)(0,-1)[\phantom{Y_2}``F\s'G(u)]{450}0r
\putmorphism(1180,500)(0,-1)[\phantom{Y_2}``G'G(u)]{450}1r
\put(-160,290){\fbox{$F\s'((\alpha_1)_u)$}}
\put(700,300){\fbox{$(\beta_1)_{G(u)}$}}
\efig
$$
\item for every 1h-cell $f:A\to B$  in $\Aa$ a 2-cell in $\Cc$:
$$\delta_{\beta_1\comp\alpha_1,f}=
\bfig

 \putmorphism(-250,0)(1,0)[F\s'F(A)`F\s'F(B)` F\s'F(f)]{700}1a
 \putmorphism(450,0)(1,0)[\phantom{F\s'F(B)}`F\s'G(B) `F\s'(\alpha_1(B))]{760}1a

 \putmorphism(-250,-450)(1,0)[F\s'F(A)`F\s'G(A)`F\s'(\alpha_1(A))]{700}1a
 \putmorphism(480,-450)(1,0)[\phantom{A\ot B}`F\s'G(B) `F\s'G(f)]{740}1a
 \putmorphism(1200,-450)(1,0)[\phantom{A'\ot B'}` G'G(B) `\beta_1(G(B))]{760}1a

\putmorphism(-280,0)(0,-1)[\phantom{Y_2}``=]{450}1r
\putmorphism(1180,0)(0,-1)[\phantom{Y_2}``=]{450}1r
\put(350,-240){\fbox{$ \delta_{F\s'(\alpha_1),f}  $}}
\put(1000,-700){\fbox{$\delta_{\beta_1,G(f)}$}}

 \putmorphism(450,-900)(1,0)[F\s'G(A)` G'G(A) `\beta_1(G(A))]{760}1a
 \putmorphism(1180,-900)(1,0)[\phantom{A''\ot B'}`G'G(B) ` G'G(f)]{760}1a

\putmorphism(450,-450)(0,-1)[\phantom{Y_2}``=]{450}1l
\putmorphism(1950,-450)(0,-1)[\phantom{Y_2}``=]{450}1r
\efig
$$
where $\delta_{F\s'(\alpha_1),f}$ is from \leref{delta F(alfa)}. 
\end{itemize}
\end{lma}

{\em Vertical pseudonatural transformations} between double pseudo functors $F,G: \Aa\to\Bb$ are defined in an analogous way, 
consisting of a 1v-cell $\alpha(A):F(A)\to G(A)$ in $\Bb$ for every 0-cell $A$ in $\Aa$, 
for every 1h-cell $f:A\to B$ in $\Aa$ a 2-cell on the left hand-side below and for every 1v-cell $u:A\to A'$  in $\Aa$ a 2-cell 
on the right hand-side below, both in $\Bb$:
$$
\bfig
\putmorphism(-150,180)(1,0)[F(A)`F(B)`F(f)]{560}1a
\putmorphism(-150,-190)(1,0)[G(A)`G(B)`G(f)]{600}1a
\putmorphism(-180,180)(0,-1)[\phantom{Y_2}``\alpha(A)]{370}1l
\putmorphism(410,180)(0,-1)[\phantom{Y_2}``\alpha(B)]{370}1r
\put(30,40){\fbox{$\alpha_f$}}
\efig
\qquad\qquad
\bfig
 \putmorphism(-90,500)(1,0)[F(A)`F(A) `=]{540}1a
\putmorphism(-120,500)(0,-1)[\phantom{Y_2}`F(A') `F(u)]{400}1l
\putmorphism(-90,-300)(1,0)[G(A')`G(A') `=]{540}1a
\putmorphism(-120,100)(0,-1)[\phantom{Y_2}``\alpha(A')]{400}1l
\putmorphism(450,100)(0,-1)[\phantom{Y_2}``G(u)]{400}1r
\putmorphism(450,500)(0,-1)[\phantom{Y_2}`G(A) `\alpha(A)]{400}1r
\put(60,50){\fbox{$\delta_{\alpha,u}$}}
\efig
$$
Observe that we use the same notation for the 2-cells $\alpha_{\bullet}$ and $\delta_{\alpha,\bullet}$ both for a horizontal and a vertical 
pseudonatural transformation $\alpha$, the difference is indicated by the notation for the respective 1-cell, recall that horizontal ones are denoted by $f,g..$ and vertical ones by $u,v...$. 

For vertical pseudonatural transformations results analogous to \leref{delta F(alfa)} and \equref{Cor} hold, as well as to \leref{horiz comp hor.ps.tr.} 
which we state here in order to fix the structures that we use:

\begin{lma} \lelabel{horiz comp vert.ps.tr.}
Horizontal composition of two vertical pseudonatural transformations $\alpha_0: F\Rightarrow G: \Aa\to\Bb$ and 
$\beta_0: F\s'\Rightarrow G':\Bb\to\Cc$, denoted by $\beta_0\comp\alpha_0$, is well-given by: 
\begin{itemize}
\item for every 0-cell $A$ in $\Aa$ a 1v-cell on the left below, and 
for every 1h-cell $f:A\to B$ in $\Aa$ a 2-cell on the right below, both in $\Cc$:
$$(\beta_0\comp\alpha_0)(A)=
\bfig
\putmorphism(-280,500)(0,-1)[F\s'F(A)`F\s'G(A) `F\s'(\alpha_0(A))]{450}1l
 \putmorphism(-280,70)(0,-1)[\phantom{F(A)}`G'G(A) `\beta_0(G(A))]{450}1l
\efig \qquad
(\beta_0\comp\alpha_0)_f= 
\bfig
\putmorphism(-250,500)(1,0)[F\s'F(A)`F\s'F(B)` F\s'F(f)]{700}1a
 \putmorphism(-250,50)(1,0)[F\s'G(A)`F\s'G(B)` F\s'G(f)]{700}1a
 \putmorphism(-250,-400)(1,0)[G'G(A)`G'G(B)` G'G(f)]{700}1a

\putmorphism(-280,500)(0,-1)[\phantom{Y_2}``F\s'(\alpha_0(A))]{450}1l
 \putmorphism(-280,70)(0,-1)[\phantom{F(A)}` `\beta_0(G(A))]{450}1l

\putmorphism(450,500)(0,-1)[\phantom{Y_2}``F\s'(\alpha_0(B))]{450}1r
\putmorphism(450,70)(0,-1)[\phantom{Y_2}``\beta_0(G(B))]{450}1r
\put(-160,290){\fbox{$F\s'((\alpha_0)_f)$}}
\put(-120,-150){\fbox{$(\beta_0)_{G(f)}$}}
\efig
$$
\item for every 1v-cell $u:A\to A'$ in $\Aa$ a 2-cell in $\Cc$: 
$$\delta_{\beta_0\comp\alpha_0,u}=
\bfig
 \putmorphism(-150,500)(1,0)[F\s'F(A)`F\s'F(A) `=]{550}1a
\putmorphism(-180,500)(0,-1)[\phantom{Y_2}`F\s'F(A') `F\s'F(u)]{450}1l
\put(-80,-160){\fbox{$\delta_{F\s'(\alpha_0),u}$}}
\putmorphism(-150,-400)(1,0)[F\s'G(A')`F\s'G(A') `=]{550}1a
\putmorphism(-180,50)(0,-1)[\phantom{Y_2}``F\s'(\alpha_0(A'))]{450}1l
\putmorphism(380,500)(0,-1)[\phantom{Y_2}` `F\s'(\alpha_0(A))]{450}1r
\putmorphism(380,50)(0,-1)[F\s'G(A)` `F\s'G(u)]{450}1r
\putmorphism(520,60)(1,0)[`F\s'G(A)`=]{460}1a
\putmorphism(370,-850)(1,0)[\phantom{G'G(A)}``=]{440}1b
\putmorphism(960,50)(0,-1)[\phantom{(B, \tilde A')}``\beta_0(G(A))]{450}1r
\putmorphism(960,-400)(0,-1)[G'G(A)`G'G(A')` G'G(u)]{450}1r
\putmorphism(400,-400)(0,-1)[\phantom{(B, \tilde A)}`G'G( A') `\beta_0(G(A'))]{450}1l
\put(500,-630){\fbox{$\delta_{\beta_0,G(u)}$}}
\efig
$$
where $\delta_{F\s'(\alpha_0),u}$ is defined analogously as in \leref{delta F(alfa)}. 
\end{itemize}
\end{lma}

Horizontal compositions of horizontal and of vertical pseudonatural transformations are not strictly associative. 

\medskip

We proceed by defining vertical compositions of horizontal and of vertical pseudonatural transformations. From the respective definitions 
it will be clear that these vertical compositions are strictly associative.

\begin{lma} \lelabel{vert comp hor.ps.tr.}
Vertical composition of two horizontal pseudonatural transformations $\alpha_1: F\Rightarrow G: \Aa\to\Bb$ and 
$\beta_1: G\Rightarrow H:\Aa\to\Bb$, denoted by $\frac{\alpha_1}{\beta_1}$, is well-given by: 
\begin{itemize}
\item for every 0-cell $A$ in $\Aa$ a 1h-cell in $\Bb$:
$$(\frac{\alpha_1}{\beta_1})(A)=\big( F(A)\stackrel{\alpha_1(A)}{\longrightarrow}G(A) \stackrel{\beta_1(A)}{\longrightarrow} H(A) \big),$$ 
\item for every 1v-cell $u:A\to A'$ in $\Aa$ a 2-cell in $\Bb$:
$$(\frac{\alpha_1}{\beta_1})(u)=
\bfig
\putmorphism(-150,50)(1,0)[F(A)`G(A)`\alpha_1(A)]{560}1a
 \putmorphism(430,50)(1,0)[\phantom{F(A)}`H(A) `\beta_1(A)]{620}1a

\putmorphism(-150,-320)(1,0)[F(A')`G(A')`\alpha_1(A')]{600}1a
\putmorphism(-180,50)(0,-1)[\phantom{Y_2}``F(u)]{370}1l
\putmorphism(380,50)(0,-1)[\phantom{Y_2}``G(u)]{370}1r
\put(-30,-130){\fbox{$(\alpha_1)_u$}}

 \putmorphism(460,-320)(1,0)[\phantom{F(A)}`H(A') `\beta_1(A')]{630}1a

\putmorphism(1060,50)(0,-1)[\phantom{Y_2}``H(u)]{370}1r
\put(640,-110){\fbox{$(\beta_1)_u$}}
\efig
$$
\item for every 1h-cell $f:A\to B$  in $\Aa$ a 2-cell in $\Bb$:
$$\delta_{\frac{\alpha_1}{\beta_1},f}=
\bfig

 \putmorphism(-200,-50)(1,0)[F(A)`F(B)` F(f)]{650}1a
 \putmorphism(430,-50)(1,0)[\phantom{F(B)}`G(B) `\alpha_1(B)]{700}1a

 \putmorphism(-200,-450)(1,0)[F(A)`G(A)`\alpha_1(A)]{650}1a
 \putmorphism(430,-450)(1,0)[\phantom{A\ot B}`G(B) `G(f)]{700}1a
 \putmorphism(1050,-450)(1,0)[\phantom{A'\ot B'}` H(B) `\beta_1(B)]{700}1a

\putmorphism(-230,-50)(0,-1)[\phantom{Y_2}``=]{400}1r
\putmorphism(1050,-50)(0,-1)[\phantom{Y_2}``=]{400}1r
\put(300,-240){\fbox{$ \delta_{\alpha_1,f}  $}}
\put(1000,-660){\fbox{$\delta_{\beta_1,f}$}}

 \putmorphism(450,-850)(1,0)[G(A)` H(A) `\beta_1(A)]{700}1a
 \putmorphism(1080,-850)(1,0)[\phantom{A''\ot B'}`G(B). ` H(f)]{700}1a

\putmorphism(450,-450)(0,-1)[\phantom{Y_2}``=]{400}1l
\putmorphism(1750,-450)(0,-1)[\phantom{Y_2}``=]{400}1r
\efig
$$
\end{itemize}
\end{lma}

\bigskip

\begin{lma} \lelabel{vert comp vert.ps.tr.}
Vertical composition of two vertical pseudonatural transformations $\alpha_0: F\Rightarrow G: \Aa\to\Bb$ and 
$\beta_0: G\Rightarrow H:\Aa\to\Bb$, denoted by $\frac{\alpha_0}{\beta_0}$, is well-given by: 
\begin{itemize}
\item for every 0-cell $A$ in $\Aa$ a 1v-cell on the left below, and for every 1h-cell $f:A\to B$ in $\Aa$ a 2-cell on the right below, both in $\Bb$:
$$(\frac{\alpha_0}{\beta_0})(A)=
\bfig
\putmorphism(-280,500)(0,-1)[F(A)`G(A) `\alpha_0(A)]{450}1l
 \putmorphism(-280,70)(0,-1)[\phantom{F(A)}`H(A) `\beta_0(A)]{450}1l
\efig \qquad
(\frac{\alpha_0}{\beta_0})(f)= 
\bfig
\putmorphism(-250,500)(1,0)[F(A)`F(B)` F(f)]{550}1a
 \putmorphism(-250,50)(1,0)[G(A)`G(B)` G(f)]{550}1a
 \putmorphism(-250,-400)(1,0)[H(A)`H(B)` H(f)]{550}1a

\putmorphism(-280,500)(0,-1)[\phantom{Y_2}``\alpha_0(A)]{450}1l
 \putmorphism(-280,70)(0,-1)[\phantom{F(A)}` `\beta_0(A)]{450}1l

\putmorphism(300,500)(0,-1)[\phantom{Y_2}``\alpha_0(B)]{450}1r
\putmorphism(300,70)(0,-1)[\phantom{Y_2}``\beta_0(B)]{450}1r
\put(-120,290){\fbox{$(\alpha_0)_f$}}
\put(-120,-150){\fbox{$(\beta_0)_f$}}
\efig
$$ 
\item for every 1v-cell $u:A\to A'$ in $\Aa$ a 2-cell in $\Bb$: 
$$\delta_{\frac{\alpha_0}{\beta_0},u}=
\bfig
 \putmorphism(-150,500)(1,0)[F(A)`F(A) `=]{460}1a
\putmorphism(-130,500)(0,-1)[\phantom{Y_2}`F(A') `F(u)]{400}1l
\put(0,250){\fbox{$\delta_{\alpha_0,u}$}}
\putmorphism(-150,-300)(1,0)[G(A')`G(A') `=]{460}1a
\putmorphism(-130,110)(0,-1)[\phantom{Y_2}``\alpha_0(A')]{400}1l
\putmorphism(380,500)(0,-1)[\phantom{Y_2}` `\alpha_0(A)]{400}1r
\putmorphism(380,100)(0,-1)[G(A)` `G(u)]{400}1l
\putmorphism(480,110)(1,0)[`G(A)`=]{460}1a
\putmorphism(390,-700)(1,0)[\phantom{G(A)}`H(A').`=]{570}1a
\putmorphism(920,100)(0,-1)[\phantom{(B, \tilde A')}``\beta_0(A)]{400}1r
\putmorphism(920,-300)(0,-1)[H(A)`` H(u)]{400}1r
\putmorphism(400,-300)(0,-1)[\phantom{(B, \tilde A)}`H( A') `\beta_0(A')]{400}1l
\put(530,-190){\fbox{$\delta_{\beta_0,u}$}}
\efig
$$
\end{itemize}
\end{lma}

\subsection{2-cells of the tricategory and their compositions} 

Now we may define what will be the 2-cells of our tricategory $\DblPs$ of strict double categories and double pseudofunctors.

\begin{defn} \delabel{double 2-cells}
A {\em double pseudonatural transformation} $\alpha:F\to G$ between double pseudofunctors 
is a quadruple $(\alpha_0, \alpha_1, t^\alpha, r^\alpha)$, where: 

(T1) $\alpha_0:F\Rightarrow G$ is a vertical pseudonatural transformation, and \\ \indent 
 $\alpha_1:F\Rightarrow G$ is a horizontal pseudonatural transformation, 

(T2) the 2-cells $\delta_{\alpha_1,f}$ and $\delta_{\alpha_0,u}$ are invertible when 
$f$ is a 1h-cell component of a horizontal pseudonatural transformation, and 
$u$ is a 1v-cell component of a vertical pseudonatural transformation;

(T3) for every 1h-cell $f:A\to B$ and 1v-cell $u:A\to A'$ in $\Aa$ there are 2-cells in $\Bb$:
$$
\bfig
\putmorphism(-150,500)(1,0)[F(A)`F(B)`F(f)]{600}1a
 \putmorphism(430,500)(1,0)[\phantom{F(A)}`G(B) `\alpha_1(B)]{610}1a
 \putmorphism(-150,100)(1,0)[G(A)`G(B)`G(f)]{1200}1b
\putmorphism(-180,500)(0,-1)[\phantom{Y_2}``\alpha_0(A)]{400}1l
\putmorphism(1050,500)(0,-1)[\phantom{Y_2}``=]{400}1r
\put(350,260){\fbox{$t^\alpha_f$}}
\efig
\qquad\text{and}\qquad
\bfig
 \putmorphism(-150,500)(1,0)[F(A)`G(A)  `\alpha_1(A)]{500}1a
\putmorphism(-180,500)(0,-1)[\phantom{Y_2}`F(A') `F(u)]{400}1l
\putmorphism(-150,-300)(1,0)[G(A')` `=]{400}1a
\putmorphism(-180,100)(0,-1)[\phantom{Y_2}``\alpha_0(A')]{400}1l
\putmorphism(350,500)(0,-1)[\phantom{Y_2}`G(A') `G(u)]{800}1r
\put(50,100){\fbox{$r^\alpha_u$}}
\efig
$$
satisfying: \\
(T3-1)
$$
\bfig
\putmorphism(-150,500)(1,0)[F(A)`F(B)`F(f)]{600}1a
 \putmorphism(450,500)(1,0)[\phantom{F(A)}`G(B) `\alpha_1(B)]{600}1a

 \putmorphism(-150,50)(1,0)[F(A')`F(B')`F(g)]{600}1a
 \putmorphism(450,50)(1,0)[\phantom{F(A)}`G(B') `\alpha_1(B')]{600}1a

\putmorphism(-180,500)(0,-1)[\phantom{Y_2}``F(u)]{450}1l
\putmorphism(450,500)(0,-1)[\phantom{Y_2}``]{450}1r
\putmorphism(300,500)(0,-1)[\phantom{Y_2}``F(v)]{450}0r
\putmorphism(1050,500)(0,-1)[\phantom{Y_2}``G(v)]{450}1r
\put(0,260){\fbox{$F(a)$}}
\put(630,270){\fbox{$(\alpha_1)_v$}}

\putmorphism(-150,-400)(1,0)[G(A')`G(B') `G(g)]{1200}1a

\putmorphism(-180,50)(0,-1)[\phantom{Y_2}``\alpha_0(A')]{450}1l
\putmorphism(1050,50)(0,-1)[\phantom{Y_3}``=]{450}1r
\put(370,-140){\fbox{$t^\alpha_g$}}

\efig=
\bfig
 \putmorphism(-150,500)(1,0)[F(A)`F(A) `=]{500}1a
\putmorphism(-180,500)(0,-1)[\phantom{Y_2}`F(A') `F(u)]{450}1r
\put(0,30){\fbox{$\delta_{\alpha_0,u}$}}
\putmorphism(-150,-400)(1,0)[G(A')`G(A') `=]{520}1a
\putmorphism(-180,50)(0,-1)[\phantom{Y_2}``\alpha_0(A')]{450}1r
\putmorphism(380,500)(0,-1)[\phantom{Y_2}` `\alpha_0(A)]{450}1r
\putmorphism(380,50)(0,-1)[G(A)` `G(u)]{450}1r
\putmorphism(350,500)(1,0)[F(A)`F(B)`F(f)]{600}1a
 \putmorphism(950,500)(1,0)[\phantom{F(A)}`G(B) `\alpha_1(B)]{650}1a
 \putmorphism(470,50)(1,0)[`G(B)`G(f)]{1150}1b
\putmorphism(1570,500)(0,-1)[\phantom{Y_2}``=]{450}1r
\put(880,260){\fbox{$t^\alpha_f$}}
\putmorphism(480,-400)(1,0)[`G(B') `G(g)]{1140}1a
\putmorphism(1570,50)(0,-1)[\phantom{Y_2}``G(v)]{450}1r
\put(850,-200){\fbox{$G(a)$}}
\efig
$$
and 
$$
\bfig
 \putmorphism(-150,500)(1,0)[F(A)`F(B)`F(f)]{600}1a
 \putmorphism(450,500)(1,0)[\phantom{F(A)}`G(B) `\alpha_1(B)]{600}1a
\putmorphism(-180,500)(0,-1)[\phantom{Y_2}`F(A') `F(u)]{450}1l
\put(680,50){\fbox{$r^\alpha_v$}}
\putmorphism(-150,-400)(1,0)[G(A')` `G(g)]{500}1a
\putmorphism(-180,50)(0,-1)[\phantom{Y_2}``\alpha_0(A')]{450}1l
\putmorphism(450,50)(0,-1)[\phantom{Y_2}`G(B')`\alpha_0(B')]{450}1r
\putmorphism(450,500)(0,-1)[\phantom{Y_2}`F(B') `F(v)]{450}1r
\put(0,260){\fbox{$F(a)$}}
\putmorphism(-150,50)(1,0)[\phantom{(B, \tilde A)}``F(g)]{500}1a
\putmorphism(1000,500)(0,-1)[`G(B')`G(v)]{900}1r
\putmorphism(480,-400)(1,0)[\phantom{F(A)}`\phantom{F(A)}`=]{500}1b
\put(0,-170){\fbox{$(\alpha_0)_g$}}
\efig
=
\bfig
\putmorphism(-150,500)(1,0)[F(A)`F(B)`F(f)]{600}1a
 \putmorphism(450,500)(1,0)[\phantom{F(A)}`G(B) `\alpha_1(B)]{680}1a
 \putmorphism(-150,100)(1,0)[F(A)`G(A)`\alpha_1(A)]{600}1a
 \putmorphism(450,100)(1,0)[\phantom{F(A)}`G(B) `G(f)]{680}1a

\putmorphism(-180,500)(0,-1)[\phantom{Y_2}``=]{400}1r
\putmorphism(1100,500)(0,-1)[\phantom{Y_2}``=]{400}1r
\put(350,260){\fbox{$\delta_{\alpha_1,f}$}}
\putmorphism(-180,100)(0,-1)[\phantom{Y_2}`F(A') `F(u)]{400}1l
\putmorphism(-150,-680)(1,0)[G(A')` `=]{480}1a
\putmorphism(-180,-280)(0,-1)[\phantom{Y_2}``\alpha_0(A')]{400}1l
\putmorphism(450,100)(0,-1)[\phantom{Y_2}`G(A') `G(u)]{760}1l
\put(0,-300){\fbox{$r^\alpha_u$}}
\putmorphism(1100,100)(0,-1)[\phantom{Y_2}`G(B') `G(v)]{760}1r
 \putmorphism(450,-660)(1,0)[\phantom{F(A)}` `G(g)]{540}1a
\put(620,-300){\fbox{$G(a)$}} 

\efig
$$
for every 2-cell $a$ in $\Aa$, 

(T3-2) for every composable 1h-cells $f$ and $g$ and every composable 1v-cells $u$ and $v$ it is: 
$$
\bfig

 \putmorphism(440,400)(1,0)[F(B)`F(C)` F(g)]{550}1a
 \putmorphism(1000,400)(1,0)[\phantom{F(B)}`G(C) `\alpha_1(C)]{550}1a

 \putmorphism(-200,0)(1,0)[F(A)`F(B) `F(f)]{600}1a
 \putmorphism(320,0)(1,0)[\phantom{A'\ot B'}` G(B) `\alpha_1(B)]{640}1a
 \putmorphism(900,0)(1,0)[\phantom{A'\ot B'}` G(C) `G(g)]{680}1a

\putmorphism(400,400)(0,-1)[\phantom{Y_2}``=]{400}1l
\putmorphism(1550,400)(0,-1)[\phantom{Y_2}``=]{400}1r
\put(1000,200){\fbox{$ \delta_{\alpha_1,g}$ }}
\put(560,-210){\fbox{$t^\alpha_f$}}

 \putmorphism(-230,-400)(1,0)[G(A)` G(B) ` G(f)]{1210}1a

\putmorphism(-230,0)(0,-1)[\phantom{Y_2}``\alpha_0(A)]{400}1r
\putmorphism(960,0)(0,-1)[\phantom{Y_2}``=]{400}1r
\efig
=
\bfig
 \putmorphism(-100,200)(1,0)[F(A)`F(B)`F(f)]{520}1a
 \putmorphism(420,200)(1,0)[\phantom{F(B)}`F(C) `F(g)]{520}1a
 \putmorphism(930,200)(1,0)[\phantom{F(B)}`G(C) `\alpha_1(C)]{570}1a
\put(20,20){\fbox{$(\alpha_0)_f$}}
\putmorphism(-100,200)(0,-1)[\phantom{Y_2}``\alpha_0(A)]{400}1l
\put(1100,20){\fbox{$t^\alpha_g$}}

\putmorphism(430,200)(0,-1)[\phantom{Y_2}``\alpha_0(B)]{400}1r
\putmorphism(1510,200)(0,-1)[\phantom{Y_2}``=]{400}1r

  \putmorphism(-150,-200)(1,0)[G(A)` G(B) `G(f)]{600}1a
\putmorphism(450,-200)(1,0)[\phantom{F(A)}`G(C) ` G(g)]{1050}1a
\efig
$$ 
and
$$
\bfig
 \putmorphism(920,450)(1,0)[F(A)`F(A) `\alpha_1(A)]{460}1a
\putmorphism(920,450)(0,-1)[\phantom{Y_2}`F(A') `F(u)]{400}1l
\put(1050,250){\fbox{$r^\alpha_u$}}
\putmorphism(920,-400)(1,0)[G(A')`G(A') `=]{460}1a
\putmorphism(1380,450)(0,-1)[\phantom{Y_2}` `G(u)]{850}1r
\putmorphism(440,50)(0,-1)[F(A')`F(A'') `F(v)]{450}1l
\putmorphism(530,60)(1,0)[``=]{300}1a
\putmorphism(560,-850)(1,0)[`G(A'')`=]{430}1a
\putmorphism(920,50)(0,-1)[\phantom{(B, \tilde A')}``\alpha_0(A')]{450}1r
\putmorphism(920,-400)(0,-1)[`` G(v)]{450}1r
\putmorphism(440,-400)(0,-1)[\phantom{(B, \tilde A)}`G(A'') `\alpha_0(A'')]{450}1l
\put(530,-190){\fbox{$\delta_{\alpha_0,v}$}}
\efig
=
\bfig
 \putmorphism(-150,410)(1,0)[F(A)`G(A)  `\alpha_1(A)]{530}1a
\putmorphism(-160,400)(0,-1)[\phantom{Y_2}`F(A') `F(u)]{380}1l
\putmorphism(370,400)(0,-1)[\phantom{Y_2}`G(A') `G(u)]{380}1r
\putmorphism(-160,50)(0,-1)[\phantom{Y_2}`F(A'')`F(v)]{430}1l
\putmorphism(370,50)(0,-1)[\phantom{Y_2}`G(A')`G(v)]{820}1r
\put(-60,240){\fbox{$(\alpha_1)_u$}}
\putmorphism(-70,20)(1,0)[``\alpha_1(A')]{330}1a
\put(20,-210){\fbox{$r^\alpha_v$ }}
\putmorphism(-160,-350)(0,-1)[\phantom{Y_2}``\alpha_0(A'')]{400}1l
\putmorphism(-180,-750)(1,0)[G(A')` `=]{440}1b
\efig
$$

(T3-3) for every composable 1h-cells $f$ and $g$ and every composable 1v-cells $u$ and $v$ it is: 
$$t^\alpha_{gf}=
\bfig
\putmorphism(-150,850)(1,0)[F(A)`F(C)`F(gf)]{1200}1a

 \putmorphism(-180,450)(1,0)[F(A)`F(B)`F(f)]{600}1a
 \putmorphism(420,450)(1,0)[\phantom{F(B)}`F(C) `F(g)]{680}1a
 \putmorphism(1090,450)(1,0)[\phantom{F(B)}`G(C) `\alpha_1(C)]{570}1a
\putmorphism(-180,840)(0,-1)[\phantom{Y_2}``=]{400}1l
\putmorphism(1060,840)(0,-1)[\phantom{Y_2}``=]{400}1r
\put(350,650){\fbox{$F_{gf}$}}
\put(0,250){\fbox{$(\alpha_0)_f$}}
\putmorphism(-180,450)(0,-1)[\phantom{Y_2}``\alpha_0(A)]{400}1l
\put(1000,270){\fbox{$t^\alpha_g$}}

\putmorphism(450,450)(0,-1)[\phantom{Y_2}``\alpha_0(B)]{400}1r
\putmorphism(1660,450)(0,-1)[\phantom{Y_2}``=]{400}1r

  \putmorphism(-150,50)(1,0)[G(A)` G(B) `G(f)]{600}1a
\putmorphism(450,50)(1,0)[\phantom{F(A)}`G(C) ` G(g)]{1200}1a


\putmorphism(-180,50)(0,-1)[\phantom{Y_2}``=]{400}1r
\putmorphism(1640,50)(0,-1)[\phantom{Y_2}``=]{400}1r
\put(650,-180){\fbox{$G^{-1}_{gf}$}}

 \putmorphism(-150,-350)(1,0)[G(A)` G(C) `G(gf)]{1800}1b
\efig
$$ 
and
$$r^\alpha_{vu}=
\bfig
\putmorphism(-600,500)(0,-1)[F(A)`F(A'')`F(vu)]{850}1r
 \putmorphism(-510,500)(1,0)[` F(A)`=]{400}1a
 \putmorphism(-30,500)(1,0)[`G(A)  `\alpha_1(A)]{450}1a
 \putmorphism(400,500)(1,0)[\phantom{(B,A)}` `=]{450}1a
\putmorphism(-180,500)(0,-1)[\phantom{Y_2}`F(A') `F(u)]{450}1r
\put(650,50){\fbox{$G^{vu}$}}
 \putmorphism(-150,-750)(1,0)[G(A'')` `=]{480}1a
\putmorphism(-180,50)(0,-1)[\phantom{Y_2}``F(v)]{400}1r
\putmorphism(-180,-350)(0,-1)[F(A'')``\alpha_0(A'')]{400}1r
\putmorphism(-500,-340)(1,0)[` `=]{200}1a

\putmorphism(450,50)(0,-1)[\phantom{Y_2}`G(A'')`G(v)]{800}1r
\putmorphism(450,500)(0,-1)[\phantom{Y_2}`G(A') `G(u)]{450}1r
\put(80,260){\fbox{$(\alpha_1)_u$}}
\putmorphism(-150,50)(1,0)[\phantom{(B, \tilde A)}``\alpha_1(A')]{500}1a
\putmorphism(960,500)(0,-1)[G(A)`G(A'').`G(vu)]{1250}1r
\putmorphism(470,-750)(1,0)[\phantom{F(A)}`\phantom{F(A)}`=]{450}1a

\put(120,-230){\fbox{$r^\alpha_v$}}
\put(-580,250){\fbox{$(F^{vu})^{-1}$}}
\efig
$$
\end{defn}

By the axiom (v) of a double pseudofunctor in \cite[Definition 6.1]{Shul1} one has:

\begin{lma}
Given three composable 1h-cells $f,g,h$ and three composable 1v-cells $u,v,w$ for a double pseudonatural transformation $\alpha$ 
it is: $t^\alpha_{(hg)f}=t^\alpha_{h(gf)}$ and $r^\alpha_{(wv)u}=r^\alpha_{w(vu)}$. 
\end{lma}

\begin{rem} \rmlabel{axioms used for t} 
The horizontal and vertical compositions of $t$'s and  $r$'s are defined in the next two Propositions below. 
Axiom (T3-2) in the above definition is introduced in order for $t$'s to satisfy the interchange law (up to isomorphism). 
\end{rem}

The following identity for horizontally composable pseudonatural transformations 
$
\xymatrix{
\Aa \ar@{:= }[r]|{\Downarrow \alpha_i} 
\ar@/^{1pc}/[r]^F \ar@/_{1pc}/[r]_G &  
\Bb \ar@{:= }[r]|{\Downarrow \beta_i} \ar@/^{1pc}/[r]^{F'} \ar@/_{1pc}/[r]_{G'} & \Cc
}
$
for $i=0,1$ and a 1h-cell $f:A\to B$ and a 1v-cell $u:A\to A'$, 
is used in order to prove that the horizontal composition of $t$'s satisfies the axiom (T3-2): 
$$
\bfig
 \putmorphism(-170,400)(1,0)[F\s'F(A)`F\s'F(B)` F\s'F(f)]{600}1a
 \putmorphism(470,400)(1,0)[\phantom{F(B)}`F\s'G(B) `F\s'(\alpha_1(B))]{640}1a
\putmorphism(-170,400)(0,-1)[\phantom{Y_2}``=]{400}1l
\putmorphism(1070,400)(0,-1)[\phantom{Y_2}``=]{400}1r
\put(300,210){\fbox{$\delta_{F\s'(\alpha_1),f}$}}
 \putmorphism(-170,0)(1,0)[F\s'F(A)`F\s'G(A)`F\s'(\alpha_1(A))]{670}1a
 \putmorphism(540,0)(1,0)[\phantom{F(B)}`F\s'G(B) `F\s'G(f)]{550}1a
 \putmorphism(1160,0)(1,0)[\phantom{F(B)}`G'G(B) `\beta_1(G(B))]{570}1a
\put(0,-220){\fbox{$(\beta_0)_{\alpha_1(A)}$}}
\putmorphism(-170,0)(0,-1)[\phantom{Y_2}``\beta_0(F(A))]{400}1l
\put(1020,-220){\fbox{$t^\beta_{G(f)}$}}

\putmorphism(460,0)(0,-1)[\phantom{Y_2}``\beta_0(G(A))]{400}1r
\putmorphism(1700,0)(0,-1)[\phantom{Y_2}``=]{400}1r

  \putmorphism(-170,-400)(1,0)[G'F(A)` G'G(A) `G'(\alpha_1(A))]{670}1b
\putmorphism(550,-400)(1,0)[\phantom{F(A)}`G'G(B) ` G'G(f)]{1200}1b
\efig
$$
$$=
\bfig
\putmorphism(-100,0)(0,-1)[\phantom{Y_2}``=]{400}1l
\putmorphism(1710,0)(0,-1)[\phantom{Y_2}``=]{400}1r
 \putmorphism(-170,400)(1,0)[F\s'F(A)`F\s'F(B)` F\s'F(f)]{600}1a
 \putmorphism(470,400)(1,0)[\phantom{F(B)}`F\s'G(B) `F\s'(\alpha_1(B))]{640}1a
 \putmorphism(1180,400)(1,0)[\phantom{F(B)}`G'G(B) `\beta_1(G(B))]{550}1a
\put(0,210){\fbox{$(\beta_0)_{F(f)}$}}
\putmorphism(-100,400)(0,-1)[\phantom{Y_2}``\beta_0(F(A))]{400}1l
\put(1100,200){\fbox{$t^\beta_{\alpha_1(B)}$}}

\putmorphism(430,400)(0,-1)[\phantom{Y_2}``\alpha_0(B)]{400}1r
\putmorphism(1710,400)(0,-1)[\phantom{Y_2}``=]{400}1r

  \putmorphism(-150,0)(1,0)[G'F(A)` G'F(B) `G'F(f)]{600}1a
\putmorphism(490,0)(1,0)[\phantom{F(A)}`G'G(B) ` G'(\alpha_1(B))]{1260}1a

 \putmorphism(-150,-400)(1,0)[G'F(A)` G'G(A) `G'(\alpha_1(A))]{760}1a
 \putmorphism(600,-400)(1,0)[\phantom{A'\ot B'}` G'G(B) `G'G(f)]{1140}1a
\put(560,-210){\fbox{$\delta_{G'(\alpha_1),f}$}}

\efig
$$

and 

$$
\bfig
\putmorphism(440,670)(0,-1)[`F\s'F(A') `F\s'F(u)]{450}1l
\putmorphism(440,660)(1,0)[F\s'F(A)`F\s'F(A)`=]{500}1a
\putmorphism(1100,660)(1,0)[`G'F(A)`\beta_1(F(A))]{480}1a

\putmorphism(1550,670)(0,-1)[\phantom{Y_2}` `G'(\alpha_0(A))]{450}1r
\putmorphism(1550,200)(0,-1)[\phantom{Y_2}` `G'G(u)]{820}1r
\putmorphism(920,200)(1,0)[F\s'G(A)`G'G(A) `\beta_1(G(A))]{680}1a

\putmorphism(440,-250)(1,0)[F\s'G(A')`F\s'G(A')`=]{550}1a
\putmorphism(920,670)(0,-1)[\phantom{(B, \tilde A')}``F\s'(\alpha_0(A))]{450}1l
\putmorphism(920,200)(0,-1)[`` F\s'G(u)]{450}1r
\putmorphism(440,200)(0,-1)[\phantom{(B, \tilde A)}` `F\s'(\alpha_0(A'))]{450}1l
\put(500,10){\fbox{$\delta_{F\s'(\alpha_0),u}$}}
\put(1050,450){\fbox{$(\beta_1)_{\alpha_0(A)}$}}
\put(1130,-440){\fbox{$r^\beta_{G(u)}$}}

\putmorphism(920,-230)(0,-1)[\phantom{Y_2}` `\beta_0(G(A'))]{400}1l
 \putmorphism(920,-630)(1,0)[G'G(A')`G'G(A') `=]{660}1a
\efig
=
\bfig
 \putmorphism(-270,610)(1,0)[F\s'F(A)`G'F(A)  `\beta_1(F(A))]{670}1a
 \putmorphism(530,610)(1,0)[`G'F(A)  `=]{500}1a

\putmorphism(-160,600)(0,-1)[\phantom{Y_2}` `F\s'F(u)]{380}1l 
\putmorphism(370,600)(0,-1)[\phantom{Y_2}``G'F(u)]{380}1r  
\putmorphism(-160,250)(0,-1)[\phantom{Y_2}`F\s'G(A')`F\s'(\alpha_0(A'))]{430}1l
\putmorphism(370,250)(0,-1)[\phantom{Y_2}``G'(\alpha_0(A'))]{800}1r
\put(-60,440){\fbox{$(\beta_1)_{F(u)}$}}
\putmorphism(-270,220)(1,0)[F\s'F(A')`G'F(A')`\beta_1(F(A'))]{690}1a
\put(-50,10){\fbox{$r^\beta_{\alpha_0(A')} $}}
\putmorphism(-160,-150)(0,-1)[\phantom{Y_2}``\beta_0(G(A'))]{400}1l
\putmorphism(-180,-550)(1,0)[G'G(A')`G'G(A') `=]{630}1b
\putmorphism(600,-550)(1,0)[` G'G(A')`=]{420}1b

\putmorphism(1000,600)(0,-1)[\phantom{Y_2}`G'G(A) `G'(\alpha_0(A))]{380}1r
\putmorphism(1000,250)(0,-1)[\phantom{Y_2}``G'G(u)]{800}1r
\put(520,30){\fbox{$\delta_{G'(\alpha_0),u} $}}
\efig
$$
It is obtained similarly as \equref{Axiom 4}: set $a=\delta_{\alpha_1,f}$ in (T3-1) and use (T3-3), and similarly for the other direction.  

\medskip

For every 1-cell $F$ of $\DblPs$, the identity 2-cell $\Id_F: F\Rightarrow F$ is given by the 2-cells: 
$((\Id_F)_0)_f=\Id_{F(f)}=t^{\Id_F}_f, ((\Id_F)_1)_u=\Id_{F(u)}=r^{\Id_F}_u, \delta_{(\Id_F)_0,u}=\Id_{F(u)}$ and $\delta_{(\Id_F)_1,f}=\Id_{F(f)}$, 
with $(\Id_F)_0(A)$ and $(\Id_F)_1(A)$ being the identity 1v- and 1h-cells on $F(A)$, respectively, $f$ an arbitrary 1h-cell and $u$ an arbitrary 1v-cell. 

For the horizontal and vertical compositions of double pseudonatural transformations we have: 

\begin{prop} \prlabel{horiz comp 2-cells}
A horizontal composition of two double pseudonatural transformations acting between double pseudo functors  
$(\alpha_0, \alpha_1, t^\alpha, r^\alpha): F\Rightarrow G: \Aa\to\Bb$ and 
$(\beta_0, \beta_1, t^\beta, r^\beta): F\s'\Rightarrow G': \Bb\to\Cc$, denoted by $\beta\comp\alpha$, is well-given by: 
\begin{itemize}
\item the horizontal pseudonatural transformation $\beta_1\comp\alpha_1$ from \leref{horiz comp hor.ps.tr.}, 
\item the vertical pseudonatural transformation $\beta_0\comp\alpha_0$ from \leref{horiz comp vert.ps.tr.}, 
\item for every 1h-cell $f:A\to B$ and 1v-cell $u:A\to A'$ in $\Aa$: 2-cells in $\Bb$:
$$t^{\beta\comp\alpha}_f:= t^\beta_f\comp t^\alpha_f = 
\bfig
\putmorphism(-680,450)(1,0)[F\s'F(A)` `=]{360}1a
\putmorphism(-600,450)(0,-1)[\phantom{Y_2}`F\s'G(A)`F\s'(\alpha_0(A))]{400}1l
\putmorphism(-600,50)(0,-1)[\phantom{Y_2}`F\s'G(A)`=]{400}1l
\putmorphism(-600,-320)(0,-1)[\phantom{Y_2}``\beta_0(G(A))]{400}1l
\putmorphism(-640,-700)(1,0)[G'G(A)``=]{360}1b
\put(-580,-160){\fbox{$\delta_{\beta_0,\alpha_0(A)}$}}
 \putmorphism(-180,450)(1,0)[F\s'F(A)`F\s'F(B)`F\s'F(f)]{600}1a
 \putmorphism(450,450)(1,0)[\phantom{F(B)}`F\s'G(B) `F\s'(\alpha_1(B))]{660}1a
 \putmorphism(1180,450)(1,0)[\phantom{F(B)}`G'G(B) `\beta_1(G(B))]{570}1a
\put(0,250){\fbox{$(\beta_0)_{F(f)}$}}
\putmorphism(-110,450)(0,-1)[\phantom{Y_2}``\beta_0(F(A))]{400}1l
\put(1000,270){\fbox{$t^\beta_{\alpha_1(B)}$}}

\putmorphism(450,450)(0,-1)[\phantom{Y_2}``\beta_0(F(B))]{400}1r
\putmorphism(1660,450)(0,-1)[\phantom{Y_2}``=]{400}1r

  \putmorphism(-150,50)(1,0)[G'F(A)` G'F(B) `G'F(f)]{600}1a
\putmorphism(450,50)(1,0)[\phantom{F(A)}`G'G(B) ` G'(\alpha_1(B))]{1200}1a

\putmorphism(-110,50)(0,-1)[\phantom{Y_2}``=]{400}1r
\putmorphism(1640,50)(0,-1)[\phantom{Y_2}``=]{400}1r
\put(800,-160){\fbox{$G'^{-1}_{\alpha_1(B),F(f) }$}}

 \putmorphism(-150,-350)(1,0)[G'F(A)` G'G(B) `G'(\alpha_1(B)F(f))]{1220}1a
 \putmorphism(1200,-350)(1,0)[` G'G(B) `=]{520}1a
\putmorphism(-130,-700)(1,0)[G'G(A)`G'G(B)`G'G(f)]{1200}1b
\putmorphism(1200,-700)(1,0)[`G'G(B)`=]{520}1b
\putmorphism(-110,-340)(0,-1)[\phantom{Y_2}``G'(\alpha_0(A))]{360}1l
\putmorphism(1630,-320)(0,-1)[\phantom{Y_2}``=]{400}1r
\putmorphism(1030,-320)(0,-1)[\phantom{Y_2}``G'(id)]{400}1l
\put(250,-530){\fbox{$G'(t^\alpha_f)$}}
\put(1090,-560){\fbox{$\big(G'^{G(B)}\big)^{-1}$}}

\efig
$$ 
and \vspace{-0,4cm}
$$r^{\beta\comp\alpha}_u:= r^\beta_u\comp r^\alpha_u = 
\bfig
 \putmorphism(-180,850)(1,0)[F\s'F(A)`F\s'G(A)  `F\s'(\alpha_1(A))]{680}1a
 \putmorphism(530,850)(1,0)[\phantom{(B,A)}`F\s'G(A) `=]{490}1a
 \putmorphism(1160,850)(1,0)[`G'G(A) `\beta_1(G(A))]{600}1a 
\putmorphism(-180,850)(0,-1)[\phantom{Y_2}` `=]{350}1l
\putmorphism(1660,850)(0,-1)[``=]{350}1r
\put(540,670){\fbox{$\delta_{\beta_1,\alpha_1(A)}$}}

 \putmorphism(-180,500)(1,0)[F\s'F(A)`G'F(A)  `\beta_1(F(A))]{660}1a
\putmorphism(-180,500)(0,-1)[\phantom{Y_2}`F\s'F(A') `F\s'F(u)]{450}1l
\put(580,-190){\fbox{$G'^{\bullet\bullet}$}}

 \putmorphism(490,500)(1,0)[\phantom{(B,A)}` `=]{350}1a

\putmorphism(-180,70)(0,-1)[\phantom{Y_2}`F\s'G(A')`F\s'(\alpha_0(A'))]{420}1l
\putmorphism(440,70)(0,-1)[\phantom{Y_2}``]{820}1l  
\putmorphism(440,-60)(0,-1)[\phantom{Y_2}``G'(\alpha_0(A'))]{850}0l  

\putmorphism(-180,-330)(0,-1)[\phantom{Y_2}``\beta_0(G(A'))]{420}1l 
\putmorphism(-180,-760)(1,0)[G'G(A')`G'G(A') `=]{660}1a 
\putmorphism(540,-760)(1,0)[\phantom{F(A)}`G'G(A')`=]{490}1a 
\putmorphism(1200,-760)(1,0)[` `=]{300}1a 

\putmorphism(440,500)(0,-1)[\phantom{Y_2}` `G'F(u)]{450}1r
\put(-20,300){\fbox{$(\beta_1)_{F(u)}$}}
\putmorphism(-180,50)(1,0)[\phantom{F\s'F(A')}`G'F(A')`\beta_1(F(A'))]{720}1a
\putmorphism(980,500)(0,-1)[G'F(A)``G'(\bullet\bullet)]{800}1r   
\putmorphism(980,-280)(0,-1)[``=]{470}1r

\putmorphism(1100,-300)(1,0)[G'G(A')`G'G(A')`G'(id)]{600}1a

\putmorphism(1660,-300)(0,-1)[`G'G(A').`=]{460}1r  

\putmorphism(1660,500)(0,-1)[``G'G(u)]{800}1r
 \putmorphism(1140,500)(1,0)[`G'G(A) `G'(\alpha_1(A))]{600}1a 

\put(0,-190){\fbox{$r^\beta_{\alpha_0(A')}$}}
\put(1180,270){\fbox{$G'(r^\alpha_u)$}}
\put(1190,-530){\fbox{$G'_{G(A')}$}}
\efig
$$
\end{itemize}
\end{prop}

The non-strict associativity of the horizontal compositions of horizontal and vertical pseudonatural transformations has as a 
consequence the non-strict associativity of the horizontal composition of double pseudonatural transformations. 
In the proof that the horizontal composition of $t$'s satisfies axiom (T3-1) of \deref{double 2-cells} the vertical version of the 
identity \equref{Cor} is used. 

\begin{prop}
A vertical composition of two double pseudonatural transformations acting between double pseudo functors  
$(\alpha_0, \alpha_1, t^\alpha, r^\alpha): F\Rightarrow G: \Aa\to\Bb$ and 
$(\beta_0, \beta_1, t^\beta, r^\beta): G\Rightarrow H: \Aa\to\Bb$, denoted by $\frac{\alpha}{\beta}$, is well-given by: 
\begin{itemize}
\item the horizontal pseudonatural transformation $\frac{\alpha_1}{\beta_1}$ from \leref{vert comp hor.ps.tr.}, 
\item the vertical pseudonatural transformation $\frac{\alpha_0}{\beta_0}$ from \leref{vert comp vert.ps.tr.}, 
\item for every 1h-cell $f:A\to B$ and 1v-cell $u:A\to A'$ in $\Aa$: 2-cells in $\Bb$:
$$t^{\frac{\alpha}{\beta}}_f:= \frac{t^\alpha_f}{t^\beta_f}= 
\bfig

 \putmorphism(-200,150)(1,0)[F(A)`F(B)` F(f)]{550}1a
 \putmorphism(330,150)(1,0)[\phantom{F(B)}`G(B) `\alpha_1(B)]{550}1a

 \putmorphism(-200,-250)(1,0)[G(A)`G(B) `G(f)]{1100}1a
 \putmorphism(840,-250)(1,0)[\phantom{A'\ot B'}` H(B) `\beta_1(B)]{600}1a

\putmorphism(-230,150)(0,-1)[\phantom{Y_2}``\alpha_0(A)]{400}1r
\putmorphism(850,150)(0,-1)[\phantom{Y_2}``=]{400}1r
\put(500,-40){\fbox{$ t^\alpha_f$} }
\put(1000,-480){\fbox{$t^\beta_f$}}

 \putmorphism(-230,-650)(1,0)[H(A)` H(B) ` H(f)]{1680}1a

\putmorphism(-230,-250)(0,-1)[\phantom{Y_2}``\beta_0(A)]{400}1r
\putmorphism(1420,-250)(0,-1)[\phantom{Y_2}``=]{400}1r
\efig
$$ 
and 
$$r^{\frac{\alpha}{\beta}}_u:= \frac{r^\alpha_u}{r^\beta_u}= 
\bfig
 \putmorphism(-150,500)(1,0)[F(A)`G(A) `\alpha_1(A)]{520}1a
\putmorphism(480,500)(1,0)[`H(A)`\beta_1(A)]{460}1a

\putmorphism(-130,500)(0,-1)[\phantom{Y_2}`F(A') `F(u)]{450}1l
\put(0,250){\fbox{$r^\alpha_u$}}
\putmorphism(-150,-400)(1,0)[G(A')`G(A') `=]{500}1a
\putmorphism(-130,50)(0,-1)[\phantom{Y_2}``\alpha_0(A')]{450}1l
\putmorphism(380,500)(0,-1)[\phantom{Y_2}` `G(u)]{900}1r
\putmorphism(390,-850)(1,0)[\phantom{G(A)}`H(A').`=]{570}1a
\putmorphism(920,500)(0,-1)[`` H(u)]{1350}1r
\putmorphism(400,-400)(0,-1)[\phantom{(B, \tilde A)}`H( A') `\beta_0(A')]{450}1l
\put(530,-190){\fbox{$r^\beta_u$}}
\efig
$$
\end{itemize}
\end{prop}

This composition is clearly strictly associative. The unity constraint 3-cells for the vertical composition of 2-cells ($\lambda$ and $\rho$ from point 6) 
of (TD2) of \seref{tricat} evaluated at 2-cells) will be identities. The unity constraints for the horizontal composition we will discuss in \ssref{unities}.

\subsection{A subclass of the class of 2-cells}

In \cite[Definition 6.3]{Gabi1} {\em double natural transformations} between strict double functors were used, as a 
particular case of {\em generalized natural transformations} from \cite[Definition 3]{BMM}. Adapting the former to the case of 
{\em double pseudo functors} of Shulman, we get the following weakening of \cite[Definition 6.3]{Gabi1}:

\begin{defn} \delabel{Theta}
A {\em $\Theta$-double pseudonatural transformation} between two double pseudofunctors $F\Rightarrow G:\Aa\to\Bb$ is a tripple 
$(\alpha_0, \alpha_1, \Theta^\alpha)$, which we will denote shortly by $\Theta^\alpha$, where: 
\begin{itemize} 
\item $\alpha_0$ is a vertical and $\alpha_1$ a horizontal pseudonatural transformation (from \deref{hor psnat tr} 
and the analogous one), 
\item the axiom (T2) from \deref{double 2-cells} holds, and 
\item for every 0-cell $A$ in $\Aa$ there are 2-cells in $\Bb$:
$$
\bfig
\putmorphism(-150,50)(1,0)[F(A)`G(A)`\alpha_1(A)]{560}1a
\putmorphism(-150,-320)(1,0)[G(A)`G(A)`=]{600}1a
\putmorphism(-180,50)(0,-1)[\phantom{Y_2}``\alpha_0(A)]{370}1l
\putmorphism(410,50)(0,-1)[\phantom{Y_2}``=]{370}1r
\put(30,-120){\fbox{$\Theta^\alpha_A$}}
\efig
$$
so that for every 1h-cell $f:A\to B$ and every 1v-cell $u:A\to A'$ in $\Aa$ the following identities hold: \\
$(\Theta 0)$ 
$$
\bfig
\putmorphism(-150,500)(1,0)[F(A)`F(B)`F(f)]{600}1a
 \putmorphism(450,500)(1,0)[\phantom{F(A)}`G(B) `\alpha_1(B)]{640}1a

 \putmorphism(-150,50)(1,0)[G(A)`G(B)`G(f)]{600}1a
 \putmorphism(450,50)(1,0)[\phantom{F(A)}`G(B) `=]{640}1a

\putmorphism(-180,500)(0,-1)[\phantom{Y_2}``\alpha_0(A)]{450}1l
\putmorphism(450,500)(0,-1)[\phantom{Y_2}``]{450}1r
\putmorphism(300,500)(0,-1)[\phantom{Y_2}``\alpha_0(B)]{450}0r
\putmorphism(1100,500)(0,-1)[\phantom{Y_2}``=]{450}1r
\put(0,260){\fbox{$(\alpha_0)_f$}}
\put(700,270){\fbox{$\Theta^\alpha_B$}}
\efig
\quad=\quad
\bfig
\putmorphism(-150,500)(1,0)[F(A)`F(B)`F(f)]{600}1a
 \putmorphism(450,500)(1,0)[\phantom{F(A)}`G(B) `\alpha_1(B)]{680}1a
 \putmorphism(-150,50)(1,0)[F(A)`G(A)`\alpha_1(A)]{600}1a
 \putmorphism(450,50)(1,0)[\phantom{F(A)}`G(B) `G(f)]{680}1a

\putmorphism(-180,500)(0,-1)[\phantom{Y_2}``=]{450}1r
\putmorphism(1100,500)(0,-1)[\phantom{Y_2}``=]{450}1r
\put(350,260){\fbox{$\delta_{\alpha_1,f}$}}
\putmorphism(-150,-400)(1,0)[F(A')`G(A') `=]{640}1a

\putmorphism(-180,50)(0,-1)[\phantom{Y_2}``\alpha_0(A)]{450}1l
\putmorphism(450,50)(0,-1)[\phantom{Y_2}``]{450}1l
\putmorphism(610,50)(0,-1)[\phantom{Y_2}``=]{450}0l 
\put(20,-180){\fbox{$\Theta^\alpha_A$}} 
\efig
$$
and \\
$(\Theta 1)$ 
$$
\bfig
 \putmorphism(-150,500)(1,0)[F(A)`G(A)  `\alpha_1(A)]{600}1a
\putmorphism(-180,500)(0,-1)[\phantom{Y_2}`F(A') `F(u)]{450}1l
\putmorphism(-150,-400)(1,0)[G(A')` `=]{480}1a
\putmorphism(-180,50)(0,-1)[\phantom{Y_2}``\alpha_0(A')]{450}1l
\putmorphism(450,50)(0,-1)[\phantom{Y_2}`G(A')`=]{450}1r
\putmorphism(450,500)(0,-1)[\phantom{Y_2}`G(A') `G(u)]{450}1r
\put(0,260){\fbox{$(\alpha_1)_u$}}
\putmorphism(-180,50)(1,0)[\phantom{F(A)}``\alpha_1(A')]{500}1a
\put(20,-170){\fbox{$\Theta^\alpha_{A'}$}}
\efig=
\bfig
 \putmorphism(-150,500)(1,0)[F(A)`F(A) `=]{500}1a
 \putmorphism(450,500)(1,0)[` `\alpha_1(A)]{380}1a
\putmorphism(-180,500)(0,-1)[\phantom{Y_2}`F(A') `F(u)]{450}1l
\put(0,-160){\fbox{$\delta_{\alpha_0,u}$}}
\putmorphism(-150,-400)(1,0)[G(A')`G(A'). `=]{500}1a
\putmorphism(-180,50)(0,-1)[\phantom{Y_2}``\alpha_0(A')]{450}1l
\putmorphism(380,500)(0,-1)[\phantom{Y_2}` `\alpha_0(A)]{450}1l
\putmorphism(380,50)(0,-1)[G(A)` `G(u)]{450}1r
\put(550,290){\fbox{$\Theta^\alpha_A$} }
\putmorphism(940,500)(0,-1)[G(A)`G(A)`=]{450}1r
\putmorphism(470,60)(1,0)[``=]{360}1a
\efig
$$
\end{itemize} 
\end{defn}

Let us denote a $\Theta$-double pseudonatural transformation $\Theta^\alpha$ by 
$ \xymatrix{
\Aa  \ar@{:= }[r]|{\Downarrow \Theta^\alpha} \ar@/^{1pc}/[r]^{F} \ar@/_{1pc}/[r]_{G} & \Bb
}$. 

Horizontal composition of $\Theta$-double pseudonatural transformations 
$\xymatrix{
\Aa \ar@{:= }[r]|{\Downarrow \Theta^\alpha} 
\ar@/^{1pc}/[r]^F \ar@/_{1pc}/[r]_G &  
\Bb \ar@{:= }[r]|{\Downarrow \Theta^\beta} \ar@/^{1pc}/[r]^{F'} \ar@/_{1pc}/[r]_{G'} & \Cc
}$ 
is given by
$$\Theta^{\beta\comp\alpha}_A:= \Theta^\beta_A\comp \Theta^\alpha_A = 
\bfig
 \putmorphism(-180,500)(1,0)[F\s'F(A)`F\s'G(A)  ` F\s'(\alpha_1(A))]{670}1a
 \putmorphism(490,500)(1,0)[\phantom{(B,A)}` `=]{350}1a
\putmorphism(-180,500)(0,-1)[\phantom{Y_2}`F\s'G(A) `F\s'(\alpha_0(A))]{450}1l
\putmorphism(440,500)(0,-1)[\phantom{Y_2}` `]{450}1r
\putmorphism(330,410)(0,-1)[\phantom{Y_2}` `F\s'(id)]{450}0r

\put(580,300){\fbox{$ {F\s'^\bullet}^{-1} $}}  
\putmorphism(490,50)(1,0)[\phantom{(B,A)}`F\s'G(A) `=]{550}1a 

\putmorphism(-180,50)(0,-1)[\phantom{Y_2}``=]{450}1l
\putmorphism(440,50)(0,-1)[\phantom{F\s'F(A')}``=]{450}1r 
\put(-40,300){\fbox{$F\s'(\Theta^\alpha_A)$}}

\putmorphism(-200,50)(1,0)[\phantom{F\s'F(A')}`F\s'G(A)`F\s'(id)]{660}1a 
\putmorphism(980,500)(0,-1)[F\s'G(A)``=]{450}1r 
\putmorphism(980,50)(0,-1)[``=]{450}1r

\putmorphism(-200,-400)(1,0)[F\s'G(A)`F\s'G(A)`=]{660}1a 
\putmorphism(500,-400)(1,0)[\phantom{F(A)}`F\s'G(A)`=]{520}1a 

\putmorphism(1660,-380)(0,-1)[``=]{400}1r
 \putmorphism(1170,-400)(1,0)[`G'G(A) `\beta_1(G(A))]{600}1a 
 \putmorphism(1000,-380)(0,-1)[`G'G(A) `\beta_0(G(A))]{400}1l 
 \putmorphism(1160,-800)(1,0)[`G'G(A) `=]{600}1a 

\put(0,-190){\fbox{$F\s'_{G(A)}$}}
\put(650,-190){\fbox{$1$}}
\put(1210,-620){\fbox{$\Theta^\beta_{G(A)}$}}
\efig
$$
and vertical composition of $\Theta$-double pseudonatural transformations 
$
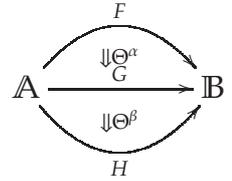
\xymatrix{
\Aa \ar@{:=}@/^1pc/[rr] |{\Downarrow \Theta^\alpha} \ar@/^{2pc}/[rr]^{F}  \ar[rr]^G
  \ar@{:=}@/_1pc/[rr] |{\Downarrow \Theta^\beta}  \ar@/_{2pc}/[rr]_{H}  &&
	\Bb }$ 
is given by
$$\frac{\Theta^\alpha_A}{\Theta^\beta_A}=
\bfig

 \putmorphism(-240,150)(1,0)[F(A)`G(A)` \alpha_1(A)]{500}1a

 \putmorphism(-280,-250)(1,0)[G(A)``=]{420}1a  %
 \putmorphism(210,-250)(1,0)[\phantom{A\ot B}`H(A) `\beta_1(A)]{500}1a

\putmorphism(-230,150)(0,-1)[\phantom{Y_2}``\alpha_0(A)]{400}1l
\putmorphism(230,150)(0,-1)[\phantom{Y_2}`G(A)`=]{400}1l

\put(-160,-40){\fbox{$ \Theta^\alpha_A$ }}
\put(350,-440){\fbox{$ \Theta^\beta_A $}}

\putmorphism(250,-250)(0,-1)[\phantom{Y_2}``\beta_0(A)]{400}1l
\putmorphism(690,-250)(0,-1)[``=]{400}1r

 \putmorphism(250,-640)(1,0)[H(A)` H(A). `=]{450}1a
\efig
$$

The following result is directly proved: 

\begin{prop} \prlabel{teta->double}
A $\Theta$-double pseudonatural transformation $\Theta^\alpha$ gives rise to a double pseudonatural transformation $(\alpha_0,\alpha_1, t^\alpha, r^\alpha)$, 
where 
$$t^\alpha_f=
\bfig
\putmorphism(-150,250)(1,0)[F(A)`F(B)`F(f)]{600}1a
 \putmorphism(450,250)(1,0)[\phantom{F(A)}`G(B) `\alpha_1(B)]{600}1a

 \putmorphism(-150,-200)(1,0)[G(A)` G(B)` G(f)]{600}1a
 \putmorphism(450,-200)(1,0)[\phantom{F(A)}`G(B) ` =]{600}1a

\putmorphism(-180,250)(0,-1)[\phantom{Y_2}``\alpha_0(A)]{450}1l
\putmorphism(450,250)(0,-1)[\phantom{Y_2}``]{450}1r
\putmorphism(300,250)(0,-1)[\phantom{Y_2}``\alpha_0(B)]{450}0r
\putmorphism(1050,250)(0,-1)[\phantom{Y_2}``=]{450}1r
\put(0,10){\fbox{$(\alpha_0)_f$}}
\put(650,20){\fbox{$\Theta^\alpha_B$}}
\efig
\quad\text{and}\quad
r^\alpha_u=
\bfig
 \putmorphism(-150,500)(1,0)[F(A)`G(A)  `\alpha_1(A)]{600}1a
\putmorphism(-180,500)(0,-1)[\phantom{Y_2}`F(A') `F(u)]{450}1l
\putmorphism(-150,-400)(1,0)[G(A')` `=]{480}1a
\putmorphism(-180,50)(0,-1)[\phantom{Y_2}``\alpha_0(A')]{450}1l
\putmorphism(450,50)(0,-1)[\phantom{Y_2}`G(A')`=]{450}1r
\putmorphism(450,500)(0,-1)[\phantom{Y_2}`G(A') `G(u)]{450}1r
\put(0,260){\fbox{$(\alpha_1)_u$}}
\putmorphism(-180,50)(1,0)[\phantom{F(A)}``\alpha_1(A')]{500}1a
\put(20,-170){\fbox{$\Theta^\alpha_{A'}$}}
\efig
$$
for every 1h-cell $f:A\to B$ and 1v-cell $u:A\to A'$. Moreover, the class of all $\Theta$-double pseudonatural transformations 
is a subclass of the class of double pseudonatural transformations. 
\end{prop}

\begin{proof}
By the axiom $(\Theta 1)$, axiom 1. for the horizontal pseudonatural transformation $\alpha_1$ implies 
axiom (T3-1) for $t^\alpha_f$. 
\end{proof}

\medskip

Thus, from the point of view of $\Theta$-double pseudonatural transformations, the axioms (T3-2) and (T3-3) of 
double pseudonatural transformations become redundant. 

Observe also that given a double pseudonatural transformation $\alpha:F\Rightarrow G$ acting between {\em strict} 
double functors, the 2-cells $t^\alpha_{id_A}$ obey the conditions $(\Theta 0)$ and $(\Theta 1)$ for every 0-cell $A$.

\bigskip

The other way around, observe that setting 
$$
\bfig
\putmorphism(-150,150)(1,0)[F(A)`G(A)`\alpha_1(A)]{560}1a
\putmorphism(-150,-220)(1,0)[G(A)`G(A).`=]{600}1a
\putmorphism(-180,150)(0,-1)[\phantom{Y_2}``\alpha_0(A)]{370}1l
\putmorphism(410,150)(0,-1)[\phantom{Y_2}``=]{370}1r
\put(0,-40){\fbox{$t\Theta^\alpha_A$}}
\efig
:=
\bfig
 \putmorphism(-150,420)(1,0)[F(A)`F(A)`=]{500}1a
\putmorphism(-180,420)(0,-1)[\phantom{Y_2}``=]{370}1l
\putmorphism(320,420)(0,-1)[\phantom{Y_2}``=]{370}1r
 \putmorphism(-150,50)(1,0)[F(A)`F(A)`F(id_A)]{500}1a
 \put(-80,250){\fbox{$(F_A)^{-1}$}} 
\putmorphism(330,50)(1,0)[\phantom{F(A)}`G(A) `\alpha_1(A)]{560}1a
 \putmorphism(-170,-350)(1,0)[F(A)`G(A)`G(id_A)]{1070}1a

\putmorphism(-180,50)(0,-1)[\phantom{Y_2}``\alpha_0(A)]{400}1r
\putmorphism(910,50)(0,-1)[\phantom{Y_2}``=]{400}1r
\put(540,-170){\fbox{$t^\alpha_{id_A}$}}
\put(240,-550){\fbox{$G_A$}}

\putmorphism(-180,-350)(0,-1)[\phantom{Y_2}``=]{350}1l
\putmorphism(920,-350)(0,-1)[\phantom{Y_3}``=]{350}1r
 \putmorphism(-180,-700)(1,0)[G(A)` G(A), `=]{1070}1a
\efig
$$
by (T3-4) we get 
$$t^\alpha_f=
\bfig
 \putmorphism(-100,200)(1,0)[F(A)`F(B)`F(f)]{520}1a
 \putmorphism(420,200)(1,0)[\phantom{F(B)}`F(B) `F(id_B)]{520}1a
 \putmorphism(930,200)(1,0)[\phantom{F(B)}`G(B) `\alpha_1(B)]{570}1a
\put(20,20){\fbox{$(\alpha_0)_f$}}
\putmorphism(-100,200)(0,-1)[\phantom{Y_2}``\alpha_0(A)]{400}1l
\put(1100,20){\fbox{$t^\alpha_{id_B}$}}

\putmorphism(430,200)(0,-1)[\phantom{Y_2}``\alpha_0(B)]{400}1r
\putmorphism(1510,200)(0,-1)[\phantom{Y_2}``=]{400}1r

  \putmorphism(-150,-200)(1,0)[G(A)` G(B) `G(f)]{600}1a
\putmorphism(450,-200)(1,0)[\phantom{F(A)}`G(B), ` G(id_B)]{1050}1a
\efig
$$
and analogously, setting 
$$
\bfig
\putmorphism(-150,150)(1,0)[F(A)`G(A)`\alpha_1(A)]{560}1a
\putmorphism(-150,-220)(1,0)[G(A)`G(A).`=]{600}1a
\putmorphism(-180,150)(0,-1)[\phantom{Y_2}``\alpha_0(A)]{370}1l
\putmorphism(410,150)(0,-1)[\phantom{Y_2}``=]{370}1r
\put(0,-40){\fbox{$r\Theta^\alpha_A$}}
\efig
:=
\bfig
\putmorphism(-610,500)(1,0)[F(A)` `=]{370}1a
\putmorphism(-600,500)(0,-1)[` `=]{450}1l
\putmorphism(-610,50)(1,0)[F(A)`F(A) `=]{450}1a
\put(-460,260){\fbox{$F^A$} }

\putmorphism(-150,500)(1,0)[F(A)`G(A) `\alpha_1(A)]{500}1a
 \putmorphism(450,500)(1,0)[` `=]{330}1a
\putmorphism(-180,500)(0,-1)[\phantom{Y_2}` `F(id_A)]{450}1r
\put(0,-160){\fbox{$r^\alpha_{id_A}$}}
\putmorphism(-150,-400)(1,0)[G(A)`G(A) `=]{500}1a
\putmorphism(-180,50)(0,-1)[\phantom{Y_2}``\alpha_0(A)]{450}1l
\putmorphism(380,500)(0,-1)[\phantom{Y_2}` `G(id_A)]{900}1l
\put(460,260){\fbox{$(G^A)^{-1}$} }
  \putmorphism(870,500)(0,-1)[G(A)`G(A),`=]{900}1r
\putmorphism(440,-400)(1,0)[``=]{340}1a
\efig
$$
one gets 
$$r^\alpha_u=
\bfig
 \putmorphism(-150,410)(1,0)[F(A)`G(A)  `\alpha_1(A)]{530}1a
\putmorphism(-160,400)(0,-1)[\phantom{Y_2}`F(A') `F(u)]{380}1l
\putmorphism(370,400)(0,-1)[\phantom{Y_2}`G(A') `G(u)]{380}1r
\putmorphism(-160,50)(0,-1)[\phantom{Y_2}`F(A')`F(id_{A'})]{430}1l
\putmorphism(370,50)(0,-1)[\phantom{Y_2}`G(A').`G(id_{A'})]{820}1r
\put(-60,240){\fbox{$(\alpha_1)_u$}}
\putmorphism(-70,20)(1,0)[``\alpha_1(A')]{330}1a
\put(20,-210){\fbox{$r^\alpha_{id_{A'}} $}}
\putmorphism(-160,-350)(0,-1)[\phantom{Y_2}``\alpha_0(A')]{400}1l
\putmorphism(-180,-750)(1,0)[G(A')` `=]{440}1b
\efig
$$
By successive applications of (T3-2) and axiom 2. for $\alpha_0$ one gets that $t\Theta^\alpha_{\bullet}$ satisfies the axiom $(\Theta 0)$, 
and similarly $r\Theta^\alpha_{\bullet}$ satisfies the axiom $(\Theta 1)$ of \deref{Theta}. Since the 2-cells $t^\alpha_f$ and $r^\alpha_f$ are 
not related, we can not claim that all double pseudonatural transformations are $\Theta$-double pseudotransformations.




\subsection{3-cells of the tricategory } 

We first define modifications for horizontal and vertical pseudonatural transformations. Since we will then define modifications for 
double pseudonatural transformations, for mnemonic reasons we will denote vertical pseudonatural transformations with index 0 and horizontal 
ones with index 1.

\begin{defn} \delabel{modif vert/horiz psnat tr}
A modification between two vertical pseudonatural transformations $\alpha_0$ and $\beta_0$ which act between double psuedofunctors 
$F\Rightarrow G$ is an application $a: \alpha_0\Rrightarrow\beta_0$ such that for each 1h-cell $f:A\to B$ in $\Aa$ there is a horizontally globular 
2-cell $a_0(A):\alpha_0(A)\Rightarrow\beta_0(A)$ satisfying: 
$$
\bfig
\putmorphism(-100,250)(1,0)[F(A)`F(A)`=]{550}1a
 \putmorphism(430,250)(1,0)[\phantom{F(A)}`F(B) `F(f)]{580}1a

 \putmorphism(-100,-200)(1,0)[G(A)`G(A)`=]{550}1a
 \putmorphism(450,-200)(1,0)[\phantom{F(A)}`G(B) `G(f)]{580}1a

\putmorphism(-100,250)(0,-1)[\phantom{Y_2}``\alpha_0(A)]{450}1l
\putmorphism(450,250)(0,-1)[\phantom{Y_2}``]{450}1r
\putmorphism(300,250)(0,-1)[\phantom{Y_2}``\beta_0(A)]{450}0r
\putmorphism(1020,250)(0,-1)[\phantom{Y_2}``\beta_0(B)]{450}1r
\put(0,10){\fbox{$a_0(A)$}}
\put(620,20){\fbox{$(\beta_0)_f$}}
\efig\quad
=\quad
\bfig
\putmorphism(-100,250)(1,0)[F(A)`F(B)`F(f)]{550}1a
 \putmorphism(430,250)(1,0)[\phantom{F(A)}`F(B) `=]{550}1a

 \putmorphism(-100,-200)(1,0)[G(A)`G(B)`G(f)]{550}1a
 \putmorphism(450,-200)(1,0)[\phantom{F(A)}`G(B) `=]{550}1a

\putmorphism(-100,250)(0,-1)[\phantom{Y_2}``\alpha_0(A)]{450}1l
\putmorphism(450,250)(0,-1)[\phantom{Y_2}``]{450}1r
\putmorphism(300,250)(0,-1)[\phantom{Y_2}``\alpha_0(B)]{450}0r
\putmorphism(990,250)(0,-1)[\phantom{Y_2}``\beta_0(B)]{450}1r
\put(0,40){\fbox{$ (\alpha_0)_f$}}
\put(590,30){\fbox{$a_0(B)$}}
\efig
$$
and 

$$
\bfig
 \putmorphism(-400,-400)(1,0)[` `=]{380}1a
 \putmorphism(80,500)(1,0)[F(A)`F(A) `=]{500}1a
\putmorphism(-480,50)(0,-1)[F(A')`G(A')`\alpha_0(A')]{450}1l
\putmorphism(-380,60)(1,0)[``=]{360}1a
\putmorphism(70,500)(0,-1)[\phantom{Y_2}`F(A') `F(u)]{450}1l
\put(200,270){\fbox{$\delta_{\beta_0,u}$}}
\putmorphism(80,-400)(1,0)[G(A')`G(A') `=]{500}1a
\putmorphism(70,50)(0,-1)[\phantom{Y_2}``\beta_0(A')]{450}1r
\putmorphism(580,500)(0,-1)[\phantom{Y_2}` `\beta_0(A)]{450}1r
\putmorphism(580,50)(0,-1)[G(A)` `G(u)]{450}1r
\put(-370,-180){\fbox{$a_0(A')$} }
\efig
=
\bfig
 \putmorphism(-150,500)(1,0)[F(A)`F(A) `=]{500}1a
 \putmorphism(450,500)(1,0)[` `=]{380}1a
\putmorphism(-180,500)(0,-1)[\phantom{Y_2}`F(A') `F(u)]{450}1l
\put(0,-160){\fbox{$\delta_{\alpha_0,u}$}}
\putmorphism(-150,-400)(1,0)[G(A')`G(A'). `=]{500}1a
\putmorphism(-180,50)(0,-1)[\phantom{Y_2}``\alpha_0(A')]{450}1l
\putmorphism(380,500)(0,-1)[\phantom{Y_2}` `\alpha_0(A)]{450}1l
\putmorphism(380,50)(0,-1)[G(A)` `G(u)]{450}1r
\put(520,290){\fbox{$a_0(A)$} }
\putmorphism(940,500)(0,-1)[F(A)`G(A)`\beta_0(A)]{450}1r
\putmorphism(470,60)(1,0)[``=]{360}1a
\efig
$$

A modification between two horizontal pseudonatural transformations $\alpha_1$ and $\beta_1$ which act between double psuedofunctors 
$F\Rightarrow G$ is an application $a: \alpha_1\Rrightarrow\beta_1$ such that for each 1v-cell $u:A\to A'$ in $\Aa$ there is a vertically globular 2-cell 
$a_1(A):\alpha_1(A)\Rightarrow\beta_1(A)$ satisfying: 
$$
\bfig
 \putmorphism(-150,470)(1,0)[F(A)`G(A)  `\alpha_1(A)]{490}1a
\putmorphism(-160,470)(0,-1)[\phantom{Y_2}`F(A) `=]{450}1l
\putmorphism(-160,50)(0,-1)[\phantom{Y_2}``F(u)]{420}1l
\putmorphism(370,50)(0,-1)[\phantom{Y_2}`G(A')`G(u)]{420}1r
\putmorphism(370,470)(0,-1)[\phantom{Y_2}`G(A) `=]{450}1r
\put(-40,260){\fbox{$a_1(A)$}}
\putmorphism(-200,20)(1,0)[\phantom{F(A)}``\beta_1(A)]{460}1a
\putmorphism(-200,-350)(1,0)[G(A')` `\beta_1(A')]{440}1b
\put(-40,-180){\fbox{$(\beta_1)_u$ }}
\efig
=
\bfig
 \putmorphism(-150,470)(1,0)[F(A)`G(A)  `\alpha_1(A)]{470}1a
\putmorphism(-160,470)(0,-1)[\phantom{Y_2}`F(A') `F(u)]{450}1l
\putmorphism(370,470)(0,-1)[\phantom{Y_2}`G(A') `G(u)]{450}1r

\putmorphism(-160,50)(0,-1)[\phantom{Y_2}``=]{420}1l
\putmorphism(370,50)(0,-1)[\phantom{Y_2}`G(A')`=]{420}1r
\put(-40,260){\fbox{$(\alpha_1)_u$}}
\putmorphism(-170,20)(1,0)[\phantom{F(A)}``\alpha_1(A')]{430}1a
\putmorphism(-190,-350)(1,0)[F(A')` `\beta_1(A')]{440}1b
\put(-40,-170){\fbox{$a_1(A') $}}
\efig
$$
and 
$$
\bfig
\putmorphism(-150,500)(1,0)[F(A)`F(B)`F(f)]{600}1a
 \putmorphism(450,500)(1,0)[\phantom{F(A)}`G(B) `\alpha_1(B)]{680}1a
 \putmorphism(-150,50)(1,0)[F(A)`G(A)`\alpha_1(A)]{600}1a
 \putmorphism(450,50)(1,0)[\phantom{F(A)}`G(B) `G(f)]{680}1a

\putmorphism(-180,500)(0,-1)[\phantom{Y_2}``=]{450}1r
\putmorphism(1100,500)(0,-1)[\phantom{Y_2}``=]{450}1r
\put(350,260){\fbox{$\delta_{\alpha_1,f}$}}
\putmorphism(-150,-400)(1,0)[F(A)`G(A) `\beta_1(A)]{640}1a

\putmorphism(-180,50)(0,-1)[\phantom{Y_2}``=]{450}1l
\putmorphism(450,50)(0,-1)[\phantom{Y_2}``]{450}1l
\putmorphism(610,50)(0,-1)[\phantom{Y_2}``=]{450}0l 
\put(20,-180){\fbox{$a_1(A)$}} 
\efig\quad=\quad
\bfig
\putmorphism(-150,0)(1,0)[F(A)`F(B)`F(f)]{600}1a
 \putmorphism(450,0)(1,0)[\phantom{F(A)}`G(B) `\beta_1(B)]{680}1a
 \putmorphism(-150,-450)(1,0)[F(A)`G(A)`\beta_1(A)]{600}1a
 \putmorphism(450,-450)(1,0)[\phantom{F(A)}`G(B). `G(f)]{680}1a

\putmorphism(-180,0)(0,-1)[\phantom{Y_2}``=]{450}1r
\putmorphism(1100,0)(0,-1)[\phantom{Y_2}``=]{450}1r
\put(350,-240){\fbox{$\delta_{\beta_1,f}$}}
\putmorphism(450,450)(1,0)[F(B)`G(B) `\alpha_1(B)]{640}1a

\putmorphism(450,450)(0,-1)[\phantom{Y_2}``=]{450}1l
\putmorphism(1100,450)(0,-1)[\phantom{Y_2}``=]{450}1l 
\put(600,250){\fbox{$a_1(B)$}} 
\efig
$$

\end{defn}

Now, 3-cells for our tricategory $\DblPs$ will be modifications which we define here:

\begin{defn} \delabel{modif 3}
A modification between two double pseudonatural transformations $\alpha=(\alpha_0,\alpha_1,t^\alpha, r^\alpha)$ and 
$\beta=(\beta_0,\beta_1, t^\beta, r^\beta)$ which act between double pseudofunctors $F\Rightarrow G$ is an application 
$a: \alpha\Rrightarrow\beta$ consisting of a modification $a_0$ for vertical pseudonatural transformations and a 
modification $a_1$ for horizontal pseudonatural transformations, such that for each 0-cell $A$ in $\Aa$ it holds: 
\begin{equation} \eqlabel{modif t}
\bfig
\putmorphism(900,650)(1,0)[F(B)`G(B)  `\alpha_1(B)]{520}1a
\putmorphism(890,650)(0,-1)[\phantom{Y_2}` `=]{380}1l
\putmorphism(1420,650)(0,-1)[\phantom{Y_2}` `=]{380}1r
\put(1020,470){\fbox{$a_1(B)$}}

\putmorphism(-100,250)(1,0)[F(A)`F(A)`=]{550}1a
 \putmorphism(430,250)(1,0)[\phantom{F(A)}`F(B) `F(f)]{480}1a
 \putmorphism(900,250)(1,0)[\phantom{F(A)}`G(B) `\beta_1(B)]{520}1a
\putmorphism(1420,250)(0,-1)[\phantom{Y_2}``=]{450}1r

 \putmorphism(-100,-200)(1,0)[G(A)`G(A)`=]{550}1a
 \putmorphism(450,-200)(1,0)[\phantom{F(A)}`G(B) `G(f)]{980}1a

\putmorphism(-100,250)(0,-1)[\phantom{Y_2}``\alpha_0(A)]{450}1l
\putmorphism(450,250)(0,-1)[\phantom{Y_2}``]{450}1r
\putmorphism(300,250)(0,-1)[\phantom{Y_2}``\beta_0(A)]{450}0r

\put(0,10){\fbox{$a_0(A)$}}
\put(1070,20){\fbox{$t^\beta_f$}}
\efig\quad
=\quad
\bfig
\putmorphism(-80,250)(1,0)[F(A)`F(B)`F(f)]{530}1a
 \putmorphism(450,250)(1,0)[\phantom{F(A)}`G(B) `\alpha_1(B)]{580}1a
 \putmorphism(-80,-150)(1,0)[F(A)`G(A)`G(f)]{1120}1b
\putmorphism(-100,250)(0,-1)[\phantom{Y_2}``\alpha_0(A)]{400}1r
\putmorphism(1020,250)(0,-1)[\phantom{Y_2}``=]{400}1r
\put(350,30){\fbox{$t^\alpha_f$}}
\efig
\end{equation}
and
$$
\bfig
 \putmorphism(-150,410)(1,0)[F(A)`G(A)  `\alpha_1(A)]{490}1a
\putmorphism(-160,430)(0,-1)[\phantom{Y_2}`F(A) `=]{410}1l
\putmorphism(370,430)(0,-1)[\phantom{Y_2}`G(A) `=]{410}1r
\putmorphism(-160,50)(0,-1)[\phantom{Y_2}``F(u)]{420}1l
\putmorphism(370,50)(0,-1)[\phantom{Y_2}`G(A')`G(u)]{820}1r
\put(-60,240){\fbox{$a_1(A)$}}

\putmorphism(-180,20)(1,0)[\phantom{F(A)}``\beta_1(A)]{460}1a
\put(20,-180){\fbox{$r^\beta_u $}}

\putmorphism(-160,-350)(0,-1)[\phantom{Y_2}``\beta_0(A')]{400}1r
\putmorphism(-620,-350)(0,-1)[\phantom{Y_2}``\alpha_0(A')]{400}1l
\putmorphism(-620,-350)(1,0)[F(A')`F(A') `=]{440}1b
\putmorphism(-630,-750)(1,0)[G(A')` `=]{340}1b
\putmorphism(-180,-750)(1,0)[G(A')` `=]{440}1b
\put(-540,-560){\fbox{$a_0(A')$}}

\efig
=
\bfig
 \putmorphism(-150,500)(1,0)[F(A)`G(A)  `\alpha_1(A)]{500}1a
\putmorphism(-180,500)(0,-1)[\phantom{Y_2}`F(A') `F(u)]{400}1l
\putmorphism(-150,-300)(1,0)[G(A')` `=]{400}1a
\putmorphism(-180,100)(0,-1)[\phantom{Y_2}``\alpha_0(A')]{400}1l
\putmorphism(350,500)(0,-1)[\phantom{Y_2}`G(A'). `G(u)]{800}1r
\put(50,100){\fbox{$r^\alpha_u$}}
\efig
$$
\end{defn}

{\em Horizontal composition} of the modifications $a: \alpha\Rrightarrow\beta: F\Rightarrow G$ and $b: \alpha'\Rrightarrow\beta': 
F\s'\Rightarrow G'$, acting between horizontally composable double pseudonatural transformations 
$\alpha'\comp\alpha\Rrightarrow\beta'\comp\beta : F\s'\comp F\Rightarrow G'\comp G$ is given for every 0-cell 
$A$ in $\Aa$ by pairs consisting of
$$(b\comp a)_0(A)=
\bfig
 \putmorphism(-170,470)(1,0)[F\s'F(A)`F\s'F(A)  `=]{570}1a
\putmorphism(-180,470)(0,-1)[\phantom{Y_2}` `=]{400}1l
\putmorphism(-180,80)(0,-1)[\phantom{Y_2}``F\s'(\alpha_0(A))]{400}1l
\putmorphism(370,80)(0,-1)[\phantom{Y_2}``F\s'(\beta_0(A))]{400}1r
\putmorphism(370,470)(0,-1)[\phantom{Y_2}``=]{400}1r  
\put(-40,280){\fbox{$(F\s'_\bullet)^{-1}$}}

\putmorphism(-170,70)(1,0)[F\s'F(A)`F\s'F(A) `F\s'(id)]{570}1a
\putmorphism(-170,-320)(1,0)[F\s'G(A)`F\s'G(A) `F\s'(id)]{570}1b
\put(-140,-130){\fbox{$F\s'(a_0(A))$}}

\putmorphism(-180,-310)(0,-1)[\phantom{Y_2}``=]{400}1l
\putmorphism(370,-310)(0,-1)[\phantom{Y_2}``=]{400}1r
\putmorphism(-170,-710)(1,0)[F\s'G(A)`F\s'G(A) `=]{570}1b
\put(-20,-560){\fbox{$F\s'_\bullet$}}

\putmorphism(-180,-700)(0,-1)[\phantom{Y_2}``\alpha'_0(G(A))]{400}1l
\putmorphism(370,-700)(0,-1)[\phantom{Y_2}``\beta'_0(G(A))]{400}1r
\putmorphism(-170,-1100)(1,0)[G'G(A)`G'G(A) `=]{570}1b
\put(-120,-930){\fbox{$b_0(G(A))$}}

\efig
$$
and 
$$(b\comp a)_1(A)=
\bfig
\putmorphism(-580,250)(0,-1)[\phantom{Y_2}``=]{450}1l
\putmorphism(1520,250)(0,-1)[\phantom{Y_2}``=]{450}1r
\putmorphism(-570,250)(1,0)[F\s'F(A)`\phantom{HF(A)}`=]{550}1a
\putmorphism(-570,-200)(1,0)[F\s'F(A)`\phantom{HF(A)}`=]{550}1a
\put(-460,30){\fbox{$F\s'^\bullet$}}
\put(1040,30){\fbox{$(F\s'^\bullet)^{-1}$}}

 \putmorphism(880,250)(1,0)[\phantom{HF(A)}`F\s'G(A) `=]{550}1a
 \putmorphism(1430,250)(1,0)[\phantom{HF(A)}`F\s'G(A) `\alpha_1'(G(A))]{900}1a
 \putmorphism(880,-200)(1,0)[\phantom{HF(A)}`F\s'G(A) `=]{550}1a
 \putmorphism(1450,-200)(1,0)[\phantom{HF(A)}`F\s'G(A). `\beta_1'(G(A))]{900}1a
\putmorphism(2350,250)(0,-1)[\phantom{Y_2}``=]{450}1l
\put(1680,30){\fbox{$b_1(G(A))$}}

\putmorphism(-20,250)(1,0)[F\s'F(A)`F\s'G(A)` F\s'(\alpha_1(A))]{900}1a
 \putmorphism(-20,-200)(1,0)[F\s'F(A)`F\s'G(A)`F\s'(\beta_1(A))]{900}1a

\putmorphism(-120,250)(0,-1)[\phantom{Y_2}``]{450}1r
\putmorphism(-140,250)(0,-1)[\phantom{Y_2}``F\s'(id)]{450}0r

\putmorphism(940,250)(0,-1)[\phantom{Y_2}``]{450}1l
\putmorphism(960,250)(0,-1)[\phantom{Y_2}``F\s'(id)]{450}0l

\put(180,30){\fbox{$F\s'(a_1(A))$}}

\efig
$$








{\em Vertical composition} of the modifications $a: \alpha\Rrightarrow\beta: F\Rightarrow G$ and $b: \alpha'\Rrightarrow\beta': 
G\Rightarrow H$, acting between vertically composable double pseudonatural transformations $F\stackrel{\alpha}{\Rightarrow} G 
\stackrel{\alpha'}{\Rightarrow}H$ and $F\stackrel{\beta}{\Rightarrow} G \stackrel{\beta'}{\Rightarrow}H$, is given for every 0-cell 
$A$ in $\Aa$ by pairs consisting of
$$(\frac{a}{b})_0(A)=
\bfig
 \putmorphism(-150,470)(1,0)[F(A)`F(A)  `=]{440}1a
\putmorphism(-160,470)(0,-1)[\phantom{Y_2}`G(A) `\alpha_0(A)]{400}1l
\putmorphism(-160,80)(0,-1)[\phantom{Y_2}``\alpha'_0(A)]{400}1l
\putmorphism(300,80)(0,-1)[\phantom{Y_2}`H(A)`\beta'_0(A)]{400}1r
\putmorphism(300,470)(0,-1)[\phantom{Y_2}`G(A) `\beta_0(A)]{400}1r
\put(-80,260){\fbox{$a_0(A)$}}

\putmorphism(-180,70)(1,0)[\phantom{F(A)}``=]{380}1a
\putmorphism(-220,-300)(1,0)[H(A)` `\beta_1(A')]{440}1b
\put(-80,-130){\fbox{$a'_0(A)$}}
\efig
\text{and}\quad
(\frac{a}{b})_1(A)=
\bfig
\putmorphism(-100,250)(1,0)[F(A)`G(A)`\alpha_1(A)]{550}1a
 \putmorphism(430,250)(1,0)[\phantom{F(A)}`H(A) `\alpha'_1(A)]{580}1a

 \putmorphism(-100,-200)(1,0)[F(A)`G(A)`\beta_1(A)]{550}1a
 \putmorphism(450,-200)(1,0)[\phantom{F(A)}`H(A). `\beta'_1(A)]{580}1a

\putmorphism(-100,250)(0,-1)[\phantom{Y_2}``=]{450}1r
\putmorphism(420,250)(0,-1)[\phantom{Y_2}``]{450}1r
 \putmorphism(400,250)(0,-1)[\phantom{Y_2}``=]{450}0r
\putmorphism(1020,250)(0,-1)[\phantom{Y_2}``=]{450}1r
\put(40,10){\fbox{$a_1(A)$}}
\put(560,20){\fbox{$a'_1(A)$}}
\efig
$$

{\em Transversal composition} of the modifications 
$\alpha\stackrel{a}{\Rrightarrow}\beta \stackrel{b}{\Rrightarrow}\gamma : F\Rightarrow G$ is given for every 0-cell 
$A$ in $\Aa$ by pairs consisting of
$$(b\cdot a)_0(A)=
\bfig
\putmorphism(-100,250)(1,0)[F(A)`G(A)`=]{550}1a
 \putmorphism(430,250)(1,0)[\phantom{F(A)}`H(A) `=]{580}1a

 \putmorphism(-100,-200)(1,0)[F(A)`G(A)`=]{550}1a
 \putmorphism(450,-200)(1,0)[\phantom{F(A)}`H(A) `=]{580}1a

\putmorphism(-100,250)(0,-1)[\phantom{Y_2}``]{450}1r
\putmorphism(-120,200)(0,-1)[\phantom{Y_2}``\alpha_0(A)]{450}0r

\putmorphism(420,250)(0,-1)[\phantom{Y_2}``]{450}1r
 \putmorphism(400,200)(0,-1)[\phantom{Y_2}``\beta_0(A)]{450}0r

\putmorphism(1020,250)(0,-1)[\phantom{Y_2}``]{450}1r
\putmorphism(1000,200)(0,-1)[\phantom{Y_2}``\gamma_0(A)]{450}0r

\put(100,100){\fbox{$a_0(A)$}}
\put(660,100){\fbox{$b_0(A)$}}
\efig\quad
\text{and}\quad
(b\cdot a)_1(A)=
\bfig
 \putmorphism(-150,470)(1,0)[F(A)`G(A)  `\alpha_1(A)]{440}1a
\putmorphism(-160,470)(0,-1)[\phantom{Y_2}`F(A) `=]{400}1l
\putmorphism(-160,80)(0,-1)[\phantom{Y_2}`F(A)`=]{400}1l
\putmorphism(300,80)(0,-1)[\phantom{Y_2}`G(A).`=]{400}1r
\putmorphism(300,470)(0,-1)[\phantom{Y_2}`G(A) `=]{400}1r
\put(-80,260){\fbox{$a_1(A)$}}

\putmorphism(-160,70)(1,0)[\phantom{F(A)}``\beta_1(A)]{380}1a
\putmorphism(-100,-300)(1,0)[\phantom{Y_2}` `\gamma_1(A)]{300}1b
\put(-80,-130){\fbox{$b_1(A)$}}
\efig
$$

From the definitions it is clear that vertical and transversal composition of the 3-cells is strictly associative. The associativity in 
the horizontal direction we will treat in \ssref{horiz ass modif}.

\subsection{A subclass of the 3-cells} 

For $\Theta$-double pseudonatural transformations we define modifications as follows:

\begin{defn}
A modification between two $\Theta$-double pseudonatural transformations $\Theta^\alpha\equiv(\alpha_0,\alpha_1,\Theta^\alpha)$ and 
$\Theta^\beta\equiv(\beta_0,\beta_1, \Theta^\beta)$ which act between double psuedofunctors $F\Rightarrow G$ is an application 
$a: \alpha\Rrightarrow\beta$ consisting of a modification $a_0$ for vertical pseudonatural transformations and a 
modification $a_1$ for horizontal pseudonatural transformations, such that for each 0-cell $A$ in $\Aa$ it holds:
$$
\bfig
\putmorphism(450,450)(1,0)[`G(A) `\alpha_1(A)]{460}1a
\putmorphism(380,470)(0,-1)[F(A)` `=]{420}1l
\putmorphism(920,470)(0,-1)[\phantom{Y_2}``=]{420}1r
\put(500,250){\fbox{$a_1(A)$} }

 \putmorphism(-150,50)(1,0)[F(A)`F(A) `=]{500}1a
 \putmorphism(450,50)(1,0)[` `\beta_1(A)]{380}1a
\putmorphism(-180,50)(0,-1)[\phantom{Y_2}` `\alpha_0(A)]{450}1l
\put(-50,-190){\fbox{$a_0(A)$}}
\putmorphism(-150,-400)(1,0)[G(A)`G(A) `=]{500}1a
\putmorphism(360,50)(0,-1)[\phantom{Y_2}` `\beta_0(A)]{450}1r
\put(660,-180){\fbox{$\Theta^\beta_A$} }
\putmorphism(940,50)(0,-1)[G(A)`G(A)`=]{450}1r
\putmorphism(470,-400)(1,0)[``=]{360}1a
\efig\quad
=
\bfig
\putmorphism(-150,50)(1,0)[F(A)`G(A)`\alpha_1(A)]{560}1a
\putmorphism(-150,-320)(1,0)[G(A)`G(A).`=]{600}1a
\putmorphism(-180,50)(0,-1)[\phantom{Y_2}``\alpha_0(A)]{370}1l
\putmorphism(410,50)(0,-1)[\phantom{Y_2}``=]{370}1r
\put(30,-140){\fbox{$\Theta^\alpha_A$}}
\efig
$$
\end{defn}

It is directly proved that modifications between $\Theta$-double pseudonatural transformations are particular cases of modifications 
between double pseudonatural transformations. This gives a sub-tricategory $\DblPs^\Theta$ of the tricategory $\DblPs$. 

\subsection{Horizontal pseudo-associativity of the 2-cells} \sslabel{assoc t's}


In order to check the horizontal associativity for 2-cells $\alpha: F\Rightarrow G, \beta: F\s'\Rightarrow G', \gamma: F''\Rightarrow G''$, 
we should find a modification of double pseudonatural transformations $(\gamma\comp\beta)\comp\alpha\stackrel{\iso}{\to} 
\gamma\comp(\beta\comp\alpha): F''F\s'F\Rightarrow G''G'G$. Because of the symmetry between horizontal and vertical 
pseudonatural transformations and between 2-cells $t^\alpha_f$ and $r^\alpha_u$, for 1h-cells $f$ and 1v-cells $u$, 
it is sufficient to find a corresponding modification of horizontal pseudonatural transformations which satisfies an 
identity of the form of \equref{modif t} proving that, so to say, $t^\gamma\comp(t^\beta\comp t^\alpha)\iso(t^\gamma\comp t^\beta)\comp t^\alpha$, 
which by definition is: $t^{\gamma\comp(\beta\comp\alpha)}\iso t^{(\gamma\comp\beta)\comp\alpha}$. 

To write out the diagrams for these two 2-cells is one of the most complex parts of the check that 
$\DblPs$ is a tricategory. Observe that to do so one should write out the diagram for $G''(t^{\beta\comp\alpha})$, which itself requires 
attention: $G''$ of a composition of 1h- or 1v-cells does not give directly a composition of the respective images of $G''$, neither $G''$ 
of identity 1-cells gives identity 1-cells, in both cases coherence diagrams should be inserted. Also, one needs to use how the 2-cells 
$G''G'^\bullet$ for the composition of 1h-cells and identity 1h-cell are given in terms of $G''$ and $G'^\bullet$. In the comparision 
of the diagrams for $t^{\gamma\comp(\beta\comp\alpha)}$ and $t^{(\gamma\comp\beta)\comp\alpha}$ one uses for example axiom 1. for the 
vertical pseudonatural transformation $\gamma_0$ (in particular, the axiom is applied to the 2-cell $(\gamma_0)_{F\s'F(f)}$), and  
the vertical version of the identity \equref{3d's}. One gets to the following conclusion: \\
\begin{equation} \eqlabel{assoc t's}  
t^{(\gamma\comp\beta)\comp\alpha}=
\end{equation}
$$
\bfig
\putmorphism(-620,250)(1,0)[F''F\s'F(A)`\phantom{HF(A)}`=]{550}1a
\putmorphism(-20,250)(1,0)[F''F\s'F(A)`F''F\s'F(B)` F''F\s'F(f)]{800}1a
\putmorphism(2800,250)(0,-1)[` `=]{1420}1l
\put(1040,-360){\fbox{$t^{\gamma\comp(\beta\comp\alpha)}$}}

 \putmorphism(800,610)(1,0)[` `F''F\s'(\alpha_1(B))]{640}1a
 \putmorphism(1260,610)(1,0)[\phantom{HF(A)}` `F''\beta_1(G(B))]{720}1a  
 \putmorphism(800,660)(0,-1)[ ` `=]{390}1l 
 \putmorphism(1980,640)(0,-1)[` `=]{390}1r
\put(1300,450){\fbox{$(F''_\bullet)^{-1}$}}

 \putmorphism(530,300)(1,0)[` `F'']{750}0a
 \putmorphism(620,250)(1,0)[` `\big(]{750}0a
 \putmorphism(990,250)(1,0)[` `F\s'(\alpha_1(B))]{470}1a
 \putmorphism(1300,250)(1,0)[\phantom{HF(A)}` `]{530}1a  
 \putmorphism(1300,250)(1,0)[\phantom{HF(A)}` `\beta_1(G(B)) ]{550}0a  
 \putmorphism(1500,250)(1,0)[\phantom{HF(A)}` `\big) ]{550}0a  
 \putmorphism(2060,250)(1,0)[F''G'G(B)` `]{750}0a  
 \putmorphism(2250,250)(1,0)[` G''G'G(B)`]{650}1a  
 \putmorphism(2170,270)(1,0)[` `\gamma_1(G'G(B))]{650}0a  

\putmorphism(-580,320)(0,-1)[\phantom{Y_2}``F''\big(]{400}0r  
\putmorphism(-780,280)(0,-1)[\phantom{Y_2}``]{480}1r 
\putmorphism(-980,220)(0,-1)[\phantom{Y_2}``F''F\s'(\alpha_0(A))]{440}0r
 \putmorphism(-480,160)(0,-1)[\phantom{Y_2}``]{380}1r
 \putmorphism(-430,160)(0,-1)[\phantom{Y_2}``F\s'(\alpha_0(A))]{380}0r
\put(-740,-170){\fbox{$F''^\bullet$}}
\putmorphism(-780,-100)(0,-1)[\phantom{Y_2}``]{390}1r
\putmorphism(-980,-80)(0,-1)[\phantom{Y_2}``F''\beta_0(G(A)))]{400}0r
 \putmorphism(-480,-110)(0,-1)[\phantom{Y_2}``]{380}1r
 \putmorphism(-430,-110)(0,-1)[\phantom{Y_2}``\beta_0(G(A))]{380}0r
 \putmorphism(-630,-110)(0,-1)[\phantom{Y_2}`F''G'G(A)`]{380}0r
 \putmorphism(-40,-150)(0,-1)[\phantom{Y_2}``\big)]{380}0r
\putmorphism(-480,-460)(0,-1)[`F''G'G(A)`=]{340}1r
\putmorphism(-480,-770)(0,-1)[`G''G'G(A)`\gamma_0(G'G(A))]{400}1r

\putmorphism(-270,-1170)(1,0)[`G''G'G(B).`G''G'G(f)]{3050}1a

\efig
$$
Thus a candidate for a modification $\tilde a : (\gamma\comp\beta)\comp\alpha \Rrightarrow \gamma\comp(\beta\comp\alpha)$ is induced by the 2-cells 
$\tilde a_0(A)=
\putmorphism(240,250)(1,0)[F''F\s'F(A)`F''F\s'F(A)`=]{600}1a
\putmorphism(620,320)(0,-1)[\phantom{Y_2}``F''\big(]{400}0r  
\putmorphism(220,280)(0,-1)[\phantom{Y_2}``]{480}1r 
\putmorphism(20,220)(0,-1)[\phantom{Y_2}``F''F\s'(\alpha_0(A))]{440}0r
 \putmorphism(720,160)(0,-1)[\phantom{Y_2}``]{380}1r
 \putmorphism(770,160)(0,-1)[\phantom{Y_2}``F\s'(\alpha_0(A))]{380}0r
\put(360,-170){\fbox{$F''^\bullet$}}
\putmorphism(220,-100)(0,-1)[\phantom{Y_2}``]{390}1r
\putmorphism(20,-80)(0,-1)[\phantom{Y_2}``F''\beta_0(G(A)))]{400}0r
 \putmorphism(720,-110)(0,-1)[\phantom{Y_2}``]{380}1r
 \putmorphism(770,-110)(0,-1)[\phantom{Y_2}``\beta_0(G(A))]{380}0r
 \putmorphism(240,250)(1,0)[F''F\s'F(A)`F''F\s'F(A)`=]{600}1a
 \putmorphism(1160,-150)(0,-1)[\phantom{Y_2}``\big)]{380}0r
\putmorphism(240,-510)(1,0)[F''G'G(A)`F''G'G(A)`=]{600}1a
\putmorphism(220,-480)(0,-1)[\phantom{Y_2}``=]{380}1r 
\putmorphism(720,-480)(0,-1)[\phantom{Y_2}``=]{380}1r 
\putmorphism(240,-870)(1,0)[F''G'G(A)`F''G'G(A)`]{600}0a  
\put(460,-780){\fbox{$\Id$}}
\putmorphism(220,-850)(0,-1)[``]{400}1r
\putmorphism(20,-850)(0,-1)[``\gamma_0(G'G(A))]{400}0r
\putmorphism(720,-850)(0,-1)[``\gamma_0(G'G(A))]{400}1r
\putmorphism(240,-1270)(1,0)[G''G'G(A)` G''G'G(A)`=]{600}1a
\hspace{6cm}\text{and}\quad
$%
$$\tilde a_1(B)=
\bfig 
\putmorphism(-570,210)(1,0)[` `F''F\s'(\alpha_1(B))]{580}1a
 \putmorphism(-160,210)(1,0)[\phantom{HF(A)}` `F''\beta_1(G(B))]{720}1a  
  \putmorphism(-680,220)(0,-1)[``=]{350}1r 
\putmorphism(-750,220)(0,-1)[F''F\s'F(B) `F''F\s'F(B) `]{350}0l 
 \putmorphism(680,220)(0,-1)[``=]{380}1l
 \putmorphism(740,200)(0,-1)[F''G'G(B)`  `]{340}0r
\put(-200,50){\fbox{$(F''_\bullet)^{-1}$}}
 \putmorphism(-880,-100)(1,0)[` `F'']{750}0a
 \putmorphism(-800,-150)(1,0)[` `\big(]{750}0a
 \putmorphism(-440,-150)(1,0)[` `F\s'(\alpha_1(B))]{470}1a
 \putmorphism(-150,-150)(1,0)[\phantom{HF(A)}` `]{530}1a  
 \putmorphism(-150,-150)(1,0)[\phantom{HF(A)}` `\beta_1(G(B)) ]{550}0a  
  \putmorphism(60,-150)(1,0)[\phantom{HF(A)}` `\big) ]{550}0a  
 \putmorphism(660,-150)(1,0)[F''G'G(B)` `]{750}0a  
 \putmorphism(850,-150)(1,0)[` G''G'G(B).`]{800}1a  
 \putmorphism(750,-130)(1,0)[` `\gamma_1(G'G(B))]{750}0a  
 \putmorphism(940,210)(1,0)[` G''G'G(B)`\gamma_1(G'G(B))]{700}1a  
 \putmorphism(1580,220)(0,-1)[``=]{380}1l
\put(1030,40){\fbox{$\Id$}}
\efig
$$
It remains to check that the above 2-cells $\tilde a_0(A)$ and $\tilde a_1(A)$ determine modifications of vertical and horizontal 
pseudonatural transformations, respectively. From  \equref{assoc t's} it will then be clear that we have a modification of double 
pseudonatural transformations $\tilde a: (\gamma\comp\beta)\comp\alpha \Rrightarrow \gamma\comp(\beta\comp\alpha)$. Because of the 
symmetry we do the check only for horizontal pseudonatural transformations. For this computation we adopt the following notation 
(we omit the index 1 to simplify the notation). In view of  \leref{horiz comp hor.ps.tr.} we have for the horizontal composition: 
$(\alpha\vert\beta)_u=[F\s'(\alpha_u)\vert\beta_{G(u)}]$, where $u$ is a 1v-cell and the square bracket is understood as horizontal 
juxtaposition of 2-cells. From here we get: 
\begin{equation} \eqlabel{assoc modif}
((\alpha\vert\beta)\vert\gamma)_u=
\bfig 
  \putmorphism(-700,560)(0,-1)[``=]{330}1r 
\putmorphism(-770,550)(0,-1)[F''F\s'F(A) ` `]{350}0l 
 \putmorphism(680,560)(0,-1)[``=]{360}1l
 \putmorphism(740,550)(0,-1)[F''G'G(A)`  `]{340}0r
\put(-160,400){\fbox{$F''_\bullet$}}
 \putmorphism(-880,600)(1,0)[` `F'']{750}0a
 \putmorphism(-800,550)(1,0)[` `\big(]{750}0a
 \putmorphism(-440,550)(1,0)[` `F\s'(\alpha_1(A))]{470}1a
 \putmorphism(-150,550)(1,0)[\phantom{HF(A)}` `]{530}1a  
 \putmorphism(-150,550)(1,0)[\phantom{HF(A)}` `\beta_1(G(A)) ]{550}0a  
  \putmorphism(60,550)(1,0)[\phantom{HF(A)}` `\big) ]{550}0a  

\putmorphism(-570,210)(1,0)[` `F''F\s'(\alpha_1(A))]{580}1a
 \putmorphism(-160,210)(1,0)[\phantom{HF(A)}` `F''\beta_1(G(A))]{680}1a  
  \putmorphism(-700,220)(0,-1)[``]{400}1l
\putmorphism(-770,220)(0,-1)[F''F\s'F(A') ` `]{350}0l 
  \putmorphism(0,220)(0,-1)[``]{400}1l

 \putmorphism(680,220)(0,-1)[``]{420}1l
 \putmorphism(740,200)(0,-1)[F''G'G(A)`  `]{340}0r
\put(-500,40){\fbox{$F''F\s'(\alpha_u)$}}
\put(160,40){\fbox{$F''(\beta_{G(u)})$}}

 \putmorphism(940,-190)(1,0)[` G''G'G(A')`]{720}1a  
 \putmorphism(800,-170)(1,0)[` `\gamma_1(G'G(A'))]{750}0a  
 \putmorphism(950,210)(1,0)[` G''G'G(A)`\gamma_1(G'G(A))]{700}1a  
 \putmorphism(1580,220)(0,-1)[``]{400}1l
\put(1000,80){\fbox{$\gamma_{G'G(u)}$}}

\putmorphism(-570,-190)(1,0)[` `F''F\s'(\alpha_1(A'))]{580}1a
 \putmorphism(-130,-190)(1,0)[\phantom{HF(A)}` `F''\beta_1(G(A'))]{660}1a  
  \putmorphism(-700,-180)(0,-1)[``=]{350}1r 
\putmorphism(-770,-180)(0,-1)[F''F\s'F(A') `F''F\s'F(A') `]{350}0l 
 \putmorphism(680,-180)(0,-1)[``=]{360}1l
 \putmorphism(740,-200)(0,-1)[F''G'G(A')`  `]{340}0r
\put(-200,-350){\fbox{$(F''_\bullet)^{-1}$}}
 \putmorphism(-880,-500)(1,0)[` `F'']{750}0a
 \putmorphism(-800,-550)(1,0)[` `\big(]{750}0a
 \putmorphism(-440,-550)(1,0)[` `F\s'(\alpha_1(A'))]{470}1a
 \putmorphism(-110,-550)(1,0)[\phantom{HF(A)}` `]{530}1a  
 \putmorphism(-110,-550)(1,0)[\phantom{HF(A)}` `\beta_1(G(A')) ]{550}0a  
  \putmorphism(110,-550)(1,0)[\phantom{HF(A)}` `\big) ]{550}0a  
 \putmorphism(720,-550)(1,0)[F''G'G(A')` `]{750}0a  
\efig
=  \threefrac{[F''_\bullet\vert\Id]}{(\alpha\vert(\beta\vert\gamma))_u}{[(F''_\bullet)^{-1}\vert\Id]} 
\end{equation} 
(recall that vertical composition of 2-cells is strictly associative). Now it is easily seen that $\tilde a_1: 
(\gamma_1\comp\beta_1)\comp\alpha_1 \Rrightarrow \gamma_1\comp(\beta_1\comp\alpha_1)$ is indeed a modification, and by our above arguments, 
so is $\tilde a: (\gamma\comp\beta)\comp\alpha \Rrightarrow \gamma\comp(\beta\comp\alpha)$. It is clear that $\tilde a$ is invertible.

Given that the horizontal composition of 1-cells is strictly associative, we set for the 3-cell $\pi$ from (TD7) from \seref{tricat} to be identity.  
We only need to check the identity of 3-cells from (TD7), observe that the only non-trivial 3-cells are the associativity constraints 
evaluated at 2-cells, that is, the 3-cells of the type of the above $\tilde a$.  
For this 
it suffices to prove that the corresponding 
modification $\tilde a_1$ of the horizontal 
pseudonatural transformations 
satisfies such an identity. Let us consider four horizontally composable horizontal pseudonatural transformations: 
$
\xymatrix@C=8pt{ 
 \ar@/^1pc/[rr]^{F}\ar@{}[rr]| {\Downarrow \alpha_1} \ar@/_1pc/[rr]_{ G} &  &  
\ar@/^1pc/[rr]^{F\s'}\ar@{}[rr]|{\Downarrow \beta_1} \ar@/_1pc/[rr]_{ G'} &  &   
\ar@/^1pc/[rr]^{F''}\ar@{}[rr]|{\Downarrow \gamma_1} \ar@/_1pc/[rr]_{ G''} &  &   
\ar@/^1pc/[rr]^{K}\ar@{}[rr]|{\Downarrow \delta_1} \ar@/_1pc/[rr]_{ L} &  & .}  
$
In \equref{assoc modif} we see that 
$\tilde a_1^{-1}$ 
acts by conjugation of $F''_\bullet(\bullet)$ 
on the composition of the first two components and it leaves the third one $\gamma_1$ unchanged. We will write this by
$\tilde a_{\alpha,\beta,\gamma}^{-1}=c^{F''}_{F\s'(\alpha_u),\beta_{G(u)}}\times 1_\gamma$ (we omit indexes 1 to simplify notation). 
For the same reason, in the identity to check 
the last one of the four components in the composition will remain untouched: 
$$(((\alpha\vert\beta)\vert\gamma)\vert\delta) \stackrel{\tilde a_{\alpha,\beta,\gamma}\times 1_\delta}{\mapsto} 
((\alpha\vert(\beta\vert\gamma))\vert\delta) \stackrel{\tilde a_{\alpha,\beta\gamma,\delta}}{\mapsto}
(\alpha\vert((\beta\vert\gamma)\vert\delta)) \stackrel{1_\alpha\times\tilde a_{\beta,\gamma,\delta}}{\mapsto}
(\alpha\vert(\beta\vert(\gamma\vert\delta)))$$ 
$\hspace{4cm}  \stackrel{\tilde a_{\alpha\beta,\gamma,\delta}}{\searrow}  \hspace{2cm} ((\alpha\vert\beta)\vert(\gamma\vert\delta)) 
\hspace{2cm}  \stackrel{\tilde a_{\alpha,\beta,\gamma\delta}}{\nearrow} $ \\
Observe that the action of the 3-cells in each arrow in this diagram corresponds to the vertical composition of 2-cells. 
The above diagramatic equation is equivalent to (mind the passage from $\tilde a$ to $\tilde a^{-1})$: 
$$ (c^{F''}_{F\s'(\alpha_u),\beta_{G(u)}}\times 1_\delta)  \comp c^K_{F''F\s'(\alpha_u),(\gamma\comp\beta)_{G(u)}}\comp 
(1_\alpha\times c^K_{F''(\beta_{G(u)}),\gamma_{G'G(u)}})  = 
c^K_{F''(\beta\comp\alpha)_u,\gamma_{G'G(u)}}  \comp  c^{KF''}_{F\s'(\alpha_u),\beta_{G(u)}} .$$
Writing out the diagrams for this equation, applying axiom (6.2) of \cite[Definition 6.1]{Shul1} three times on the left and twice on the right hand-side above, 
one arrives to the same 2-cell, which proves that the identity holds indeed. This finishes the proof of associativity up 
to an isomorphism of horizontal pseudonatural transformations (by symmetry the same follows for vertical ones) and by our comments above 
the same follows for double pseudonatural transformations. 

\subsection{Horizontal associativity of the 3-cells} \sslabel{horiz ass modif}

For the horizontal associativity of the 3-cells due to (TD5) of \seref{tricat} one has to check that the following diagram of 
modifications among the double pseudonatural transformations commutes: 
$$
\bfig
 \putmorphism(-130,160)(1,0)[` `]{360}1a
\putmorphism(380,160)(1,0)[` `]{360}1a
 \putmorphism(740,160)(1,0)[` `]{360}1a

\putmorphism(-120,200)(0,-1)[\phantom{Y_2}``]{380}1l
\putmorphism(210,200)(0,-1)[` `]{380}1l
\putmorphism(370,200)(0,-1)[\phantom{Y_2}``]{380}1l
\putmorphism(720,200)(0,-1)[\phantom{Y_2}``]{380}1l
\putmorphism(1090,200)(0,-1)[\phantom{Y_2}`` ]{380}1l

 \putmorphism(-130,-140)(1,0)[` `]{360}1a
 \putmorphism(380,-140)(1,0)[` `]{360}1a
 \putmorphism(730,-140)(1,0)[` `]{360}1a

\put(-20,0){\fbox{$\alpha$}}
\put(480,0){\fbox{$\beta$}}
\put(830,0){\fbox{$\gamma$}}

\putmorphism(-10,-470)(1,0)[` `\Downarrow]{360}0a 
\putmorphism(440,-470)(1,0)[` `(c\circ b)\circ a]{360}0a

\efig
\quad\quad\stackrel{\tilde a_{\gamma,\beta,\alpha}}{\Rrightarrow}\quad
\bfig
 \putmorphism(-130,160)(1,0)[``]{360}1b
 \putmorphism(230,160)(1,0)[``]{360}1b
 \putmorphism(710,160)(1,0)[``]{360}1b

\putmorphism(-110,200)(0,-1)[\phantom{Y_2}``]{380}1l
\putmorphism(230,200)(0,-1)[` `]{380}1l
\putmorphism(560,200)(0,-1)[\phantom{Y_2}``]{380}1l
\putmorphism(720,200)(0,-1)[\phantom{Y_2}``]{380}1l
\putmorphism(1040,200)(0,-1)[\phantom{Y_2}``]{380}1l

\putmorphism(-130,-140)(1,0)[``]{360}1b
\putmorphism(230,-140)(1,0)[``]{360}1b
 \putmorphism(710,-140)(1,0)[``]{360}1b

\put(0,-10){\fbox{$\alpha$}}
\put(330,-10){\fbox{$\beta$}}
\put(820,-10){\fbox{$\gamma$}}

\putmorphism(100,-470)(1,0)[` `\Downarrow]{360}0a 
\putmorphism(550,-470)(1,0)[` `c\circ(b\circ a)]{360}0a

\efig
$$





\vspace{-0,1cm}
$$
\bfig
 \putmorphism(-130,160)(1,0)[` `]{360}1a
\putmorphism(380,160)(1,0)[` `]{360}1a
 \putmorphism(740,160)(1,0)[` `]{360}1a

\putmorphism(-120,200)(0,-1)[\phantom{Y_2}``]{380}1l
\putmorphism(210,200)(0,-1)[` `]{380}1l
\putmorphism(370,200)(0,-1)[\phantom{Y_2}``]{380}1l
\putmorphism(720,200)(0,-1)[\phantom{Y_2}``]{380}1l
\putmorphism(1090,200)(0,-1)[\phantom{Y_2}`` ]{380}1l

 \putmorphism(-130,-140)(1,0)[` `]{360}1a
 \putmorphism(380,-140)(1,0)[` `]{360}1a
 \putmorphism(730,-140)(1,0)[` `]{360}1a

\put(-20,0){\fbox{$\alpha'$}}
\put(480,0){\fbox{$\beta'$}}
\put(830,0){\fbox{$\gamma'$}}
\efig
\quad\quad\stackrel{\tilde a_{\gamma',\beta',\alpha'}}{\Rrightarrow}\quad
\bfig
 \putmorphism(-130,160)(1,0)[``]{360}1b
 \putmorphism(230,160)(1,0)[``]{360}1b
 \putmorphism(710,160)(1,0)[``]{360}1b

\putmorphism(-110,200)(0,-1)[\phantom{Y_2}``]{380}1l
\putmorphism(230,200)(0,-1)[` `]{380}1l
\putmorphism(560,200)(0,-1)[\phantom{Y_2}``]{380}1l
\putmorphism(720,200)(0,-1)[\phantom{Y_2}``]{380}1l
\putmorphism(1040,200)(0,-1)[\phantom{Y_2}``]{380}1l

\putmorphism(-130,-140)(1,0)[``]{360}1b
\putmorphism(230,-140)(1,0)[``]{360}1b
 \putmorphism(710,-140)(1,0)[``]{360}1b

\put(0,-10){\fbox{$\alpha'$}}
\put(330,-10){\fbox{$\beta'$}}
\put(820,-10){\fbox{$\gamma'$}}
\efig
$$
Again as in the previous subsection, it is sufficient to check the above identity on the 2-cell components of the modifications of the horizontal 
pseudonatural transformations, that is, on $(\tilde a_{\gamma,\beta,\alpha})_1(A)$ and $\big((c\circ b)\circ a\big)_1(A)$ and so on, 
where instead of the 2-cell $\gamma\circ(\beta\circ\alpha)$ we consider the 2-cell component $t^{\gamma\circ(\beta\circ\alpha)}$, 
and similarly for the rest of the three ``vertices'' of the above ``square diagram''. 

Concretely, we should check if $\big(t^{(\gamma\circ\beta)\circ\alpha} \stackrel{(\tilde a_{\gamma,\beta,\alpha})_1(A)}{\mapsto}
t^{\gamma\circ(\beta\circ\alpha)} \stackrel{(c\circ (b\circ a))_1(A) }{\mapsto} t^{\gamma'\circ(\beta'\circ\alpha')}\big)=:\Omega$
gives the same diagram as $\big(t^{(\gamma\circ\beta)\circ\alpha} \stackrel{((c\circ b)\circ a)_1(A)}{\mapsto}
t^{(\gamma'\comp\beta')\circ\alpha'} \stackrel{ (\tilde a_{\gamma',\beta',\alpha'})_1(A)}{\mapsto} t^{\gamma'\circ(\beta'\circ\alpha')}\big):=\Sigma$ (here the 
action of the horizontal 2-cell components $(-)_1(A)$ of the respective modifications $(-)$ is meant as the vertical composition from above of 
$\Id\comp (-)$ on the corresponding 2-cell $t^\bullet$). 

Writing out the diagrams for $((c\circ b)\circ a)_1(A)$ and $(c\circ (b\circ a))_1(A)$ one gets as in \equref{assoc modif}: 
$$(c\circ (b\circ a))_1(A)=  \threefrac{[F''_\bullet\vert\Id]}{((c\circ b)\circ a)_1(A)}{[(F''_\bullet)^{-1}\vert\Id]}. $$
As indicated in \deref{modif 3}, we start from the image in $\Omega$, write the diagram for $(c\circ (b\circ a))_1(A)$ on top of it, 
which by the above equation is vertical composition of three 2-cells, and then as indicated in \equref{assoc t's}  write $(F''_\bullet)^{-1}$ 
on top of it, cancel out and we end up with $\frac{((c\circ b)\circ a)_1(A)}{[(F''_\bullet)^{-1}\vert\Id]}$ on top of $t^{\gamma'\circ(\beta'\circ\alpha')}$. 
The other way around, starting with the image in $\Sigma$ and doing the analogous process similar factors will cancel out, and one gets the same 2-cell as 
the result. This proves the horizontal associativity of the 3-cells.

\subsection{Unity constraints for the horizontal composition }  \sslabel{unities}

Given that $\DblPs$ is 1-strict, in the axiom (TD6) of \seref{tricat} the unity constraints for the 1-cells in $\DblPs$ are identities. Let $\Id$ 
stand shortly for the identity double pseudonatural transformation, and $\Id_0$ and $\Id_1$ for its vertical and horizontal component, respectively. 
Take an arbitrary double pseudonatural transformation $\alpha:F\rightarrow G$, 1v-cell $u:A\to A'$ and a 1h-cell $f:A\to B$. 
For the composition of the horizontal components we find: 
$(\alpha_1\comp\Id_1)_u=[ F(\Id_u)\vert(\alpha_1)_u], (\Id_1\comp\alpha_1)_u=(\alpha_1)_u$,
$$\delta_{\alpha_1\comp\Id_1,f}=
\bfig

 \putmorphism(100,0)(1,0)[F(A)` F(A)` =]{500}1a
 \putmorphism(700,0)(1,0)[`F(B) ` F(f)]{500}1a
 \putmorphism(1270,0)(1,0)[` F(B) `F(id_B)]{500}1a

\putmorphism(50,0)(0,-1)[\phantom{Y_2}``=]{450}1l
\putmorphism(600,0)(0,-1)[\phantom{Y_2}``=]{450}1r
\putmorphism(1170,0)(0,-1)[``=]{450}1r
\putmorphism(1750,0)(0,-1)[\phantom{Y_2}``=]{450}1r

\put(160,-250){\fbox{$ (F_A)^{-1}$ }}
\put(1380,-250){\fbox{$  F_B $}}

 \putmorphism(50,-450)(1,0)[F(A) ` F(A) `F(id_A) ]{580}1b

 \putmorphism(600,-450)(1,0)[\phantom{F''(A)} ` `F(f) ]{500}1b
 \putmorphism(500,-450)(1,0)[\phantom{F'(A)} ` F(B) ` ]{680}0a
	
\putmorphism(1270,-450)(1,0)[ `F(B) `=]{500}1b
	
 \putmorphism(1790,-450)(1,0)[\phantom{ B'}` G(B) ` \alpha_1(B)]{580}1b

\put(840,-250){\fbox{$  \Id $}}
\put(1280,-690){\fbox{$ \delta_{\alpha_1,f}$ }}

\putmorphism(600,-440)(0,-1)[\phantom{Y_2}``=]{450}1l
\putmorphism(2330,-440)(0,-1)[\phantom{Y_2}``=]{450}1r

 \putmorphism(600,-880)(1,0)[F(A)` G(A) ` \alpha_1(A) ]{860}1b
 \putmorphism(1490,-880)(1,0)[\phantom{A''}`G(B) ` G(f)]{850}1b
\efig
$$
$\delta_{\Id_1\comp\alpha_1,f}=\delta_{\alpha_1,f}$, 
$t^{\Id\comp\alpha}_f=t^\alpha_f$, 
$$t^{\alpha\comp\Id}_f=
\bfig
\putmorphism(900,650)(1,0)[F(B)`F(B)  `F(id_B)]{520}1a
\putmorphism(890,650)(0,-1)[\phantom{Y_2}` `=]{380}1l
\putmorphism(1420,650)(0,-1)[\phantom{Y_2}` `=]{380}1r
\put(1060,470){\fbox{$F_B$}}

\putmorphism(-100,250)(1,0)[F(A)`F(A)`=]{550}1a
 \putmorphism(430,250)(1,0)[\phantom{F(A)}`F(B) `F(f)]{480}1a
 \putmorphism(900,250)(1,0)[\phantom{F(A)}`F(B) `=]{520}1a
\putmorphism(1940,250)(0,-1)[\phantom{Y_2}``=]{900}1r

 \putmorphism(-100,-200)(1,0)[F(A)`F(A)`=]{550}1a
 \putmorphism(450,-650)(1,0)[\phantom{F(A)}`G(B). `G(f)]{1500}1a

\putmorphism(-100,250)(0,-1)[\phantom{Y_2}``F(id_A)]{450}1l
\putmorphism(450,250)(0,-1)[\phantom{Y_2}``=]{450}1r

\put(70,10){\fbox{$(F^A)^{-1}$}}
\put(1070,-220){\fbox{$t^\alpha_f$}}

 \putmorphism(1400,250)(1,0)[\phantom{F(A)}`G(B) `\alpha_1(B)]{520}1a

\putmorphism(450,-210)(0,-1)[\phantom{Y_2}`G(A)`\alpha_0(A)]{450}1l

\efig\quad
$$
For the vertical components the compositions are analogous. The modification $l_\alpha: \Id\ot\alpha\Rrightarrow\alpha$ is just the identity, 
while the modification $r_\alpha: \alpha\ot\Id\Rrightarrow\alpha$ is given by 
$$
a_0(A)=
\bfig
 \putmorphism(-150,470)(1,0)[F(A)`F(A)  `=]{470}1a
\putmorphism(-160,470)(0,-1)[\phantom{Y_2}`F(A) `F(id_A)]{450}1l
\putmorphism(350,470)(0,-1)[\phantom{Y_2}`F(A) `=]{450}1r

\putmorphism(-160,50)(0,-1)[\phantom{Y_2}``\alpha_0(A)]{420}1l
\putmorphism(350,50)(0,-1)[\phantom{Y_2}`G(A)`\alpha_0(A)]{420}1r
\put(-60,220){\fbox{$(F^A)^{-1}$}}
\putmorphism(-170,20)(1,0)[\phantom{F(A)}``=]{430}1a
\putmorphism(-190,-350)(1,0)[G(A)` `=]{440}1b
\put(0,-190){\fbox{$\Id$ }}
\efig
\qquad\quad
a_1(A)=
\bfig
\putmorphism(-100,250)(1,0)[F(A)`F(A)`F(id_A)]{550}1a
 \putmorphism(430,250)(1,0)[\phantom{F(A)}`G(A) `\alpha_1(A)]{550}1a

 \putmorphism(-100,-200)(1,0)[F(A)`F(A)`=]{550}1b
 \putmorphism(450,-200)(1,0)[\phantom{F(A)}`G(A) `\alpha_1(A)]{550}1b

\putmorphism(-100,250)(0,-1)[\phantom{Y_2}``=]{450}1r
\putmorphism(450,250)(0,-1)[\phantom{Y_2}``]{450}1r
\putmorphism(300,250)(0,-1)[\phantom{Y_2}``=]{450}0r
\putmorphism(990,250)(0,-1)[\phantom{Y_2}``=]{450}1r
\put(80,20){\fbox{$ F_A$}}
\put(630,20){\fbox{$\Id$}}
\efig
$$
and it is clearly invertible.

\subsection{Interchange laws for 2-cells}  

Let us take four double pseudonatural transformations: 
$F\stackrel{\alpha}{\Rightarrow} G \stackrel{\gamma}{\Rightarrow} H$ horizontally composable with 
$F\s'\stackrel{\beta}{\Rightarrow} G' \stackrel{\delta}{\Rightarrow} H'$, so that 
the compositions in the following desired isomorphism make sense:
$$\xi: \bigg(\frac{(\alpha\vert\beta)}{(\gamma\vert\delta)}\bigg) \stackrel{\iso}{\to} 
\bigg(\big(\frac{\alpha}{\gamma}\big) \bigg\rvert \big(\frac{\beta}{\delta}\big)\bigg). $$
To verify that such an isomorphism exists, because of the symmetry it is sufficient to verify if it exists for the corresponding 
horizontal pseudonatural transformations 
$\alpha_1,\beta_1,\gamma_1,\delta_1$ and the corresponding 2-cells $t^\alpha, t^\beta, t^\gamma, t^\delta$. As in \ssref{assoc t's}, we will supress the 
index 1, and for a 1v-cell $u$ we may write the horizontal and vertical compositions as follows:
$$(\alpha\vert\beta)_u=[F\s'(\alpha_u)\vert\beta_{G(u)}] \hspace{1,4cm} \big(\frac{\alpha}{\gamma}\big)_u=[\alpha_u\vert\gamma_u].$$
Then one gets that $\bigg(\frac{(\alpha\vert\beta)}{(\gamma\vert\delta)}\bigg)_u=
[[F\s'(\alpha_u)\vert\beta_{G(u)}] \big\vert [G'(\gamma_u)\vert\delta_{H(u)}] ]$ while 
$\bigg(\big(\frac{\alpha}{\gamma}\big) \bigg\rvert \big(\frac{\beta}{\delta}\big)\bigg)_u$ equals the third row of compositions of 2-cells in:  
$$
\bfig 
 \putmorphism(160,950)(1,0)[` `\beta(G(A))]{620}1a  
 \putmorphism(760,950)(1,0)[` `G'(\gamma(A))]{550}1a  
\put(580,760){\fbox{$\delta_{\beta,\gamma(A)}^{-1}$ }}

 \putmorphism(890,560)(1,0)[` G'H(A)`\beta(H(A))]{600}1a  
\putmorphism(0,950)(0,-1)[F\s'G(A)``=]{400}1l  
 \putmorphism(1480,950)(0,-1)[G'H(A)``=]{380}1l

  \putmorphism(-700,560)(0,-1)[``=]{330}1r 
\putmorphism(-770,550)(0,-1)[F\s'F(A) ` `]{350}0l 
 \putmorphism(680,560)(0,-1)[``=]{360}1l
 \putmorphism(740,550)(0,-1)[F\s'H(A)`  `]{340}0r
\put(-160,400){\fbox{$(F''_\bullet)^{-1}$}}
 \putmorphism(-880,260)(1,0)[` `F\s']{750}0a
 \putmorphism(-800,210)(1,0)[` `\big(]{750}0a
 \putmorphism(-440,210)(1,0)[` `\alpha(A)]{470}1a
 \putmorphism(-150,210)(1,0)[\phantom{HF(A)}` `]{530}1a  
 \putmorphism(-150,210)(1,0)[\phantom{HF(A)}` `\gamma(A) ]{550}0a  
  \putmorphism(60,210)(1,0)[\phantom{HF(A)}` `\big) ]{550}0a  

\putmorphism(-570,550)(1,0)[` `F\s'(\alpha(A))]{580}1a
 \putmorphism(-160,550)(1,0)[\phantom{HF(A)}` `F\s'(\gamma(A))]{680}1a  
  \putmorphism(-700,220)(0,-1)[``F\s'F(u)]{400}1l
\putmorphism(-770,220)(0,-1)[F\s'F(A) ` `]{350}0l 

 \putmorphism(680,220)(0,-1)[``F\s'H(u)]{420}1l
 \putmorphism(740,220)(0,-1)[F\s'H(A)`  `]{340}0r
\put(-310,40){\fbox{$F\s'(\alpha_u\vert\gamma_u)$}}
 
 \putmorphism(890,-190)(1,0)[` G'H(A)`]{620}1a 
 \putmorphism(740,-170)(1,0)[` `\beta(H(A'))]{750}0a  
 \putmorphism(890,210)(1,0)[` G'H(A)`\beta(H(A))]{600}1a  
 \putmorphism(1480,220)(0,-1)[``]{400}1l
\put(980,50){\fbox{$\beta_{H(u)}$}}

  \putmorphism(1640,210)(1,0)[` H'H(A)`\delta(H(A))]{600}1a  
 \putmorphism(1660,-190)(1,0)[` H'H(A')`\delta(H(A'))]{620}1a 
\putmorphism(2180,220)(0,-1)[``]{400}1l
\put(1690,40){\fbox{$\delta_{H(u)}$}}


\putmorphism(-570,-550)(1,0)[` `F\s'(\alpha(A'))]{580}1a
 \putmorphism(-130,-550)(1,0)[\phantom{HF(A)}` `F\s'\gamma(A')]{660}1a  

  \putmorphism(-700,-180)(0,-1)[``=]{350}1r 
\putmorphism(-770,-180)(0,-1)[F\s'F(A') `F\s'F(A') `]{350}0l 
 \putmorphism(680,-180)(0,-1)[``=]{360}1l
 \putmorphism(740,-200)(0,-1)[F\s'H(A')`  `]{340}0r
\put(-200,-350){\fbox{$F\s'_\bullet$}}
 \putmorphism(-880,-140)(1,0)[` `F\s']{750}0a
 \putmorphism(-800,-190)(1,0)[` `\big(]{750}0a
 \putmorphism(-440,-190)(1,0)[` `\alpha(A')]{470}1a
 \putmorphism(-110,-190)(1,0)[\phantom{HF(A)}` `]{530}1a  
 \putmorphism(-110,-190)(1,0)[\phantom{HF(A)}` `\gamma(A') ]{550}0a  
  \putmorphism(110,-190)(1,0)[\phantom{HF(A)}` `\big) ]{550}0a  
 \putmorphism(740,-550)(1,0)[F\s'H(A')` `]{750}0a  

 \putmorphism(890,-550)(1,0)[` G'H(A)`\beta(H(A))]{600}1a  
  \putmorphism(0,-530)(0,-1)[`F\s'G(A')`=]{400}1l  
 \putmorphism(1480,-540)(0,-1)[`G'H(A).`=]{400}1l
 \putmorphism(160,-940)(1,0)[` `\beta(G(A'))]{620}1a  
 \putmorphism(760,-940)(1,0)[` `G'(\gamma(A'))]{550}1a  
\put(580,-730){\fbox{$\delta_{\beta,\gamma(A')}$}}

\efig
$$
(Observe that $\delta_{\beta,\gamma(A)}^{-1}$ at the top of the diagram is invertible by (T2) of \deref{double 2-cells}.) 
Applying first (6.2) of \cite[Definition 6.1]{Shul1} on the middle left and then pseudonaturality of $\beta$ with respect to $\gamma_u$ (axiom 1) of 
\deref{hor psnat tr}) one gets that the above diagram equals $\bigg(\frac{(\alpha\vert\beta)}{(\gamma\vert\delta)}\bigg)_u$. 
It follows that $\bigg(\big(\frac{\alpha}{\gamma}\big) \bigg\rvert \big(\frac{\beta}{\delta}\big)\bigg)_u$ 
is the image of the former by the modification obvious from the above diagram. This proves that the interchange law for horizontal (and vertical) pseudonatural transfromations holds up to isomorphism. 

\bigskip

Let us now comment on the other desired isomorphism:  
$$\zeta: \bigg(\frac{(t^\alpha\vert t^\beta)}{(t^\gamma\vert t^\delta)}\bigg)_f \stackrel{\iso}{\to}
\bigg(\big(\frac{t^\alpha}{t^\gamma}\big) \bigg\rvert \big(\frac{t^\beta}{t^\delta}\big)\bigg)_f$$
for a 1h-cell $f$. Writing out the diagrams for the two sides in the above identity is the other most complex part in the check, together with 
that of the horizontal associativity of 2-cells, as we mentioned earlier. 
One starts by writing out the diagram for the right hand-side above and then uses the following: axiom (v) of \cite[Definition 6.1]{Shul1} 
for $H'_\bullet$; axiom 2. of the vertical analogon of \deref{hor psnat tr} for $(\delta_0)_{\alpha_1(B)\comp F(f)}$; 
axiom 1. of the same analogon for $(\delta_0)_{\alpha_1(B)\comp F(f)}$ and $H'(t^\alpha_f)$; axiom 1. for $t^\delta_{\gamma_1(B)}$ and 
$H'(\Id_{\gamma_1(B)})$; axiom 2. of \deref{hor psnat tr} for $(\delta_1)_{id_{H(B)}}$; axiom (6.3) of \cite[Definition 6.1]{Shul1} 
for $G'(\Id_{\gamma_1(B)})$; and finally axiom 2. for $t^\beta_{\gamma_1(B)}$ and $(\beta_0)_{\alpha_1(B)}$. This way one gets an  
identity as \equref{modif t} in which $\beta$ becomes $\frac{\beta\comp\alpha}{\delta\comp\gamma}$, and $\alpha$ becomes 
$\big(\frac{\beta}{\delta}\big) \comp \big(\frac{\alpha}{\gamma}\big)$ in this context, and the modification between them is given via 
$$a_0(A)=
\bfig
 \putmorphism(-50,500)(1,0)[`F\s'F(A) `=]{460}1a
\putmorphism(-280,460)(0,-1)[F\s'\big(` `]{450}0l
\putmorphism(-130,500)(0,-1)[\phantom{Y_2}`G(A) `\alpha_0(A)]{450}1l
\put(0,250){\fbox{$(F\s'^\bullet)^{-1}$}}
\putmorphism(-50,-400)(1,0)[\big)`F\s'H(A) `=]{460}1a
\putmorphism(-130,50)(0,-1)[\phantom{Y_2}``\gamma_0(A)]{450}1l
\putmorphism(380,500)(0,-1)[\phantom{Y_2}` `F\s'(\alpha_0(A))]{450}1r
\putmorphism(380,50)(0,-1)[F\s'G(A)` `F\s'(\gamma_0(A))]{450}1l
\putmorphism(530,60)(1,0)[ `F\s'G(A)`=]{360}1a
\putmorphism(430,-850)(1,0)[\phantom{G(A)}`G'H(A)`=]{500}1a
\putmorphism(920,50)(0,-1)[\phantom{(B, \tilde A')}``\beta_0(G(A))]{450}1r
\putmorphism(920,-400)(0,-1)[G'G(A)`` G'(\gamma_0(A))]{450}1r
\putmorphism(380,-400)(0,-1)[\phantom{(B, \tilde A)}`G'H(A) `\beta_0(H(A))]{450}1l
\put(500,-190){\fbox{$\delta_{\beta_0,\gamma_0(A)}$}}
\efig
$$
and 
$$a_1(B)=
\bfig
 \putmorphism(-200,200)(1,0)[F\s'\big(`G(B)` \alpha_1(B)]{650}1a
 \putmorphism(460,200)(1,0)[\phantom{F(B)}` \big) `\gamma_1(B)]{600}1a

 \putmorphism(-250,-200)(1,0)[F\s'F(B)`F\s'G(B)`F\s'(\alpha_1(B))]{680}1a
 \putmorphism(440,-200)(1,0)[\phantom{A\ot B}`F\s'H(B) `F\s'(\gamma_1(B))]{700}1a
 \putmorphism(1130,-200)(1,0)[\phantom{A'\ot B'}` H(B) `\beta_1(H(B))]{700}1a

\putmorphism(-230,200)(0,-1)[` `=]{400}1r
\putmorphism(1070,200)(0,-1)[\phantom{Y_2}``=]{400}1r
\put(300,-10){\fbox{$ F\s'_\bullet$  }}
\put(980,-400){\fbox{$\delta_{\beta_1,\gamma_1(B)}$}}

 \putmorphism(450,-600)(1,0)[F\s'G(B)` G'G(B) `\beta_1(G(B))]{700}1a
 \putmorphism(1110,-600)(1,0)[\phantom{A''\ot B'}`G'H(B). ` G'(\gamma_1(B))]{760}1a

\putmorphism(450,-200)(0,-1)[\phantom{Y_2}``=]{400}1l
\putmorphism(1800,-200)(0,-1)[\phantom{Y_2}``=]{400}1r
\efig
$$
This means that the above modification maps 
$t^{(\frac{\alpha}{\gamma}) \rvert (\frac{\beta}{\delta})}
\Rrightarrow
t^{\frac{(\alpha\vert\beta)}{(\gamma\vert\delta)}} $, thus the desired isomorphism $\zeta$ from above is the inverse of the named modification. 
This is in accordance with the modification that we found upper above for the interchange law of horizontal (and vertical) pseudonatural transfromations. 
Thus we may conclude that indeed the interchange law for double pseudonatural transformations holds up to an isomorphism.

\begin{rem}
For the subclass of $\Theta$-double pseudonatural transformations, a part from the fact that the interchange law for them follows by what we have just proved, 
one also checks the interchange law for the 2-cells $\Theta^\alpha_A$ directly, 
then an interchange law for $t$'s comming from $\Theta$'s follows shortly by interchange laws of $\alpha_0$'s and those of $\Theta$'s: 
$$ t^{(\frac{\alpha}{\gamma}) \rvert (\frac{\beta}{\delta})}_f =
\big[\big(\big((\frac{\alpha}{\gamma}) \rvert (\frac{\beta}{\delta})\big)_0\big)_f \big\rvert 
\Theta^{(\frac{\alpha}{\gamma}) \rvert (\frac{\beta}{\delta})} \big] 
\iso 
\big[\big(\big(\frac{(\alpha\vert\beta)}{(\gamma\vert\delta)}\big)_0\big)_f \big\rvert
\Theta^{ \frac{(\alpha\vert\beta)}{(\gamma\vert\delta)}} \big]=
t^{\frac{(\alpha\vert\beta)}{(\gamma\vert\delta)}}_f$$
(here we said shortly `interchange law' for `interchange law up to isomorphism' for brevity). 
\end{rem}

\subsection{Axiom (TD8) for $\DblPs$}

Given two composable 1-cells $\Aa\stackrel{F}{\to}\Bb\stackrel{G}{\to}\Cc$ of $\DblPs$ we set for all the three 3-cells 
$\mu_{G,F},\lambda_{G,F}$ and $\rho_{G,F}$ from (TD8) from \seref{tricat} to be identities (observe that all the 2-cells 
involved in their (co)domains are identities). In this case the three identities from (TD8) come down to: 
$$\big( \beta\comp(\Id\comp\alpha) \stackrel{a_{\beta,\Id,\alpha}}{\Rrightarrow}
(\beta\comp\Id)\comp\alpha \stackrel{r_\beta\ot\Id_\alpha}{\Rrightarrow} \beta\comp\alpha
\big) \quad=\quad 
\big(\beta\comp(\Id\comp\alpha) \stackrel{\Id_\beta\ot l_\alpha}{\Rrightarrow} \beta\comp\alpha\big) $$

$$\big( \Id\comp(\beta\comp\alpha) \stackrel{a_{\Id,\beta,\alpha}}{\Rrightarrow}
(\Id\comp\beta)\comp\alpha \stackrel{l_\beta\ot\Id_\alpha}{\Rrightarrow} \beta\comp\alpha
\big) \quad=\quad 
\big(\Id\comp(\beta\comp\alpha) \stackrel{l_{\beta\comp\alpha}}{\Rrightarrow} \beta\comp\alpha\big) $$
and
$$\big( \beta\comp(\alpha\comp\Id) \stackrel{a_{\beta,\alpha,\Id}}{\Rrightarrow}
(\beta\comp\alpha)\comp\Id \stackrel{r_{\beta\comp\alpha}}{\Rrightarrow} \beta\comp\alpha
\big) \quad=\quad 
\big(\beta\comp(\alpha\comp\Id) \stackrel{\Id_\beta\ot r_\alpha}{\Rrightarrow} \beta\comp\alpha\big). $$
The middle identity is quickly checked: from \ssref{unities} we know that $l_\beta$ and $l_{\beta\comp\alpha}$ are 
identities. On the other hand, $a_{\Id,\beta,\alpha}$ is indeed identity: it is the modification in \equref{assoc modif}
where $\gamma=\Id$ and $F\s''=\id$. 
For the first identity above, 
as we did at the end of \ssref{horiz ass modif}, start with the image 2-cell $\beta\comp\alpha$, 
compute the action of $r_\beta^{-1}$ on $\beta$ (see \ssref{unities}), then compose the result horizontally 
with the 2-cell $\alpha$ (recall \leref{horiz comp hor.ps.tr.}), and then act by $a_{\beta,\Id,\alpha}^{-1}$ (use \equref{assoc modif}). 
Compute the action of $\Id_\beta\ot l_\alpha$ on $\beta\comp\alpha$ in the analogous way, and one has the identity. 
The third identity above is similarly checked.

\bigskip
We have constructed our tricategory $\DblPs$ so that it is weaker from a 3-category in that the horizontal associativity and 
right unitality on 2-cells and the interchange law work up to an isomorphism. In other words, it is weaker from a Gray-category in that the 
horizontal associativity of 2-cells is not strict.




\section{The 2-category $PsDbl$ embeds into our tricategory $\DblPs$}

As we announced at the end of \ssref{beyond}, we want to propose an alternative notion to intercategories so that monoids in $(Dbl, \ot)$, 
which is the same as monoids in the Cartesian monoidal category $(Dbl, \oast)$, fit in it. Our idea is to propose it as a {\em category 
internal to the tricategory $\DblPs$}, a notion that we will introduce in a subsequent paper. 

As we explained, we can not part from the 
2-category $LxDbl$ (whose 1-cells are lax double functors) and embed it into $\DblPs$, instead we will embed the 2-category $PsDbl$ 
of pseudo double categories, pseudo double functors and vertical transformations, used in \cite{Shul}. Observe that a part from 1-cells, 
it differs from $LxDbl$ in that the horizontal direction is weak and the vertical one is strict, as for us, while in the approach 
of Grandis and Par\'e and in $LxDbl$ it is the other way around. Moreover, 2-cells in $PsDbl$ are vertical rather than horizontal transformations, 
as in $LxDbl$. Thus the 2-category $PsDbl$ is the closest one to $LxDbl$ that is 
suitable to our purpose: to embed it into our tricategory $\DblPs$ in order to define alternative intercategories later on. 



The 0-cells of $PsDbl$ are pseudo double categories and not {\em strict} double categories as in $\DblPs$. Though, by Strictification 
Theorem of \cite[Section 7.5]{GP:Limits} every pseudo double category is equivalent by a pseudodouble functor 
to a strict double category. Let $PsDbl^*_3$ be the 3-category defined by adding only the identity 3-cells to the 2-category equivalent 
to $PsDbl$ having strict double categories for 0-cells. Thus $PsDbl^*_3$ consists of strict double categories, pseudo double functors, 
vertical transformations and identity modifications among the latter. Pseudo double functors are in particular double pseudo functors, 
so the only thing it remains to check is how to make a vertical transformation a double pseudonatural transformation, that is, embed 2-cells 
of $PsDbl$ into those of $\DblPs$. 

Before doing this, we prove some more general results. 



\subsection{Bijectivity between strong vertical and strong horizontal transformations}

Recall that a {\em companion} for a 1v-cell $u:A\to A'$ is a 1h-cell  $u_*:A\to A'$ together with certain 2-cells $\Epsilon$ and $\eta$ satisfying 
$[\eta\vert\Epsilon]=\Id_{u_*}$ and $\frac{\eta}{\Epsilon}=\Id_{u}$, 
\cite[Section 1.2]{GP:Adj}, \cite[Section 3]{Shul}. 
We will say that $u_*$ is a {\em 1h-companion} of $u$. 
Companions are unique up to a unique globular isomorphism \cite[Lemma 3.8]{Shul} and a {\em connection} on a double category is 
a functorial choice of a companion for each 1v-cell, \cite{BS}. We will need a functorial choice of companions only for 1v-cell 
components of vertical pseudonatural transformations, accordingly we will speak about a {\em connection on those 1v-cells}. 

\begin{prop} \prlabel{vertic->horiz}
Let $\alpha_0:F\Rightarrow G$ be a strong vertical transformation between pseudo double functors acting between 
strict double categories $\Aa\to\Bb$ (\cite[Section 7.4]{GP:Limits}). 
The following data define a horizontal pseudonatural transformation $\alpha_1:F\Rightarrow G$:
\begin{itemize}
\item a fixed choice of a 1h-companion of $\alpha_0(A)$, for every 0-cell $A$ of $\Aa$ (with corresponding 2-cells 
$\Epsilon^\alpha_A$ and $\eta^\alpha_A$), 
we denote it by $\alpha_1(A)$; 
\item the 2-cell 
$$(\alpha_1)_u=\quad
\bfig
 \putmorphism(-40,50)(1,0)[``=]{320}1b

\put(550,30){\fbox{$\delta_{\alpha_0,u}$}}
\putmorphism(-150,-400)(1,0)[F(A')`G(A') `\alpha_1(A')]{520}1a
\putmorphism(-150,50)(0,-1)[F(A')``=]{450}1l
\putmorphism(380,500)(0,-1)[\phantom{Y_2}` `F(u)]{450}1l
\putmorphism(950,500)(0,-1)[\phantom{Y_2}`G(A) `\alpha_0(A)]{450}1r
\putmorphism(380,50)(0,-1)[F(A')` `\alpha_0(A')]{450}1r
\putmorphism(350,500)(1,0)[F(A)`F(A)`=]{600}1a
 \putmorphism(950,500)(1,0)[\phantom{F(A)}`G(A) `\alpha_1(A)]{650}1a
 \putmorphism(1060,50)(1,0)[`G(A)`=]{500}1b

\putmorphism(1570,500)(0,-1)[\phantom{Y_2}``=]{450}1r
\put(1250,260){\fbox{$\Epsilon^\alpha_A$}}
\putmorphism(480,-400)(1,0)[`G(A') `=]{500}1a
\putmorphism(950,50)(0,-1)[\phantom{Y_2}``G(u)]{450}1r
\put(0,-160){\fbox{$\eta^\alpha_{A'}$}}
\efig
$$
for every 1v-cell $u:A\to A'$; 
\item the 2-cell 
$$\delta_{\alpha_1,f}=\quad
\bfig
 \putmorphism(-150,250)(1,0)[F(A)`F(A) `=]{500}1a
\put(-100,30){\fbox{$\eta^\alpha_A$}}
\putmorphism(380,250)(0,-1)[\phantom{Y_2}` `\alpha_0(A)]{450}1l
\putmorphism(950,250)(0,-1)[\phantom{Y_2}` `\alpha_0(B)]{450}1r
\putmorphism(350,250)(1,0)[F(A)`F(B)`F(f)]{600}1a
 \putmorphism(950,250)(1,0)[\phantom{F(A)}`G(B) `\alpha_1(B)]{600}1a
 \putmorphism(470,-200)(1,0)[`G(B)`G(f)]{500}1b
 \putmorphism(1060,-200)(1,0)[`G(B)`=]{500}1b
\putmorphism(1570,250)(0,-1)[\phantom{Y_2}``=]{450}1r
\put(550,10){\fbox{$(\alpha_0)_f$}}
\put(1250,30){\fbox{$\Epsilon^\alpha_B$}}
\putmorphism(-150,-200)(1,0)[F(A)`G(A) `\alpha_1(A)]{520}1a
\putmorphism(-150,250)(0,-1)[``=]{450}1l
\efig
$$
for every 1h-cell $f:A\to B$. 
\end{itemize}
\end{prop}

\begin{proof}
To prove axiom 1), the axiom 1) of $\alpha_0$ is used; 
the first part of the axiom 2) works directly, and in the second one use second part of the axiom 3) for $\alpha_0$; 
the first part of the axiom 3) works directly, and in the second one use second part of the axiom 2) for $\alpha_0$; 
in checking of all the three axioms also the rules $\Epsilon\x\eta$ are used. 
\end{proof}

\medskip

Observe that there is a way in the other direction:

\begin{prop} \prlabel{horiz->vertic}
Let $\alpha_1:F\Rightarrow G$ be a strong horizontal transformation between pseudo double functors acting between 
strict double categories $\Aa\to\Bb$. Suppose that for every 0h-cell $A$ the 1h-cell $\alpha_1(A)$ is a 1h-companion 
of some 1v-cell (with corresponding 2-cells $\Epsilon^\alpha_A$ and $\eta^\alpha_A$). 
Fix a choice of such 1v-cells for each $A$ and denote them by $\alpha_0(A)$. 
The following data define a vertical pseudonatural transformation $\alpha_0:F\Rightarrow G$:
\begin{itemize}
\item the 1v-cell $\alpha_0(A)$, for every 0-cell $A$ of $\Aa$;
\item the 2-cell 
$$(\alpha_0)_f=\quad
\bfig
\putmorphism(-150,0)(1,0)[F(A)`F(B)`F(f)]{600}1a
 \putmorphism(450,0)(1,0)[\phantom{F(A)}`G(B) `\alpha_1(B)]{680}1a
 \putmorphism(-150,-450)(1,0)[F(A)`G(A)`\alpha_1(A)]{600}1a
 \putmorphism(450,-450)(1,0)[\phantom{F(A)}`G(B) `G(f)]{680}1a

\putmorphism(-180,0)(0,-1)[\phantom{Y_2}``=]{450}1r
\putmorphism(1100,0)(0,-1)[\phantom{Y_2}``=]{450}1r
\put(350,-240){\fbox{$\delta_{\alpha_1,f}$}}
\putmorphism(450,450)(1,0)[F(B)`F(B) `=]{640}1a

\putmorphism(450,450)(0,-1)[\phantom{Y_2}``=]{450}1l
\putmorphism(1100,450)(0,-1)[\phantom{Y_2}``\alpha_0(B)]{450}1r
\put(680,250){\fbox{$\eta^\alpha_B$}} 

\putmorphism(-150,-450)(0,-1)[\phantom{Y_2}``\alpha_0(A)]{450}1l
\putmorphism(450,-450)(0,-1)[\phantom{Y_2}``=]{450}1l 
 \putmorphism(-150,-900)(1,0)[G(A)`G(A) `=]{580}1a
\put(20,-700){\fbox{$\Epsilon^\alpha_A$}} 

\efig
$$
for every 1h-cell $f:A\to B$; 
\item the 2-cell 
$$\delta_{\alpha_0,u}=
\bfig
 \putmorphism(-150,410)(1,0)[F(A)`F(A)  `=]{530}1a
\putmorphism(-160,400)(0,-1)[\phantom{Y_2}`F(A) `=]{380}1l
\putmorphism(370,400)(0,-1)[\phantom{Y_2}`G(A) `\alpha_0(A)]{380}1r
\putmorphism(-160,50)(0,-1)[\phantom{Y_2}`F(A')`F(u)]{430}1l
\putmorphism(370,50)(0,-1)[\phantom{Y_2}`G(A')`G(u)]{430}1r
\put(-20,240){\fbox{$\eta^\alpha_A$}}
\putmorphism(-70,20)(1,0)[``\alpha_1(A)]{330}1a
\put(-60,-160){\fbox{$(\alpha_1)_u$ }}
\putmorphism(-160,-350)(0,-1)[\phantom{Y_2}``\alpha_0(A')]{400}1l
\putmorphism(-70,-350)(1,0)[``\alpha_1(A')]{330}1a
\putmorphism(370,-350)(0,-1)[\phantom{Y_2}`G(A')`=]{400}1r
\putmorphism(-180,-750)(1,0)[G(A')` `=]{440}1b
\put(-20,-560){\fbox{$\Epsilon^\alpha_{A'}$ }}
\efig
$$
for every 1v-cell $u:A\to A'$. 
\end{itemize}
\end{prop}

By $\Epsilon\x\eta$-relations, there is a 1-1 correspondence between those strong vertical transformations whose 1v-cell components 
have 1h-companions 
and those strong horizontal transformations 
whose 1h-cell components are 1h-companions of some 1v-cells. 

\begin{cor} \colabel{bij}
Suppose that there is a connection on 1v-components of strong vertical transformations. Then there is a bijection between 
strong vertical transformations and those strong horizontal transformations whose 1h-cell components are 1h-companions of some 1v-cells.
\end{cor}

As a direct corollary of \prref{vertic->horiz} we get:

\begin{cor} \colabel{4 identities e-eta}
Suppose that the 1v-components of a strong vertical transformation $\alpha_0:F\Rightarrow G$ have 1h-companions $\alpha_1(A)$, 
for every 0-cell $A$ (with corresponding 2-cells $\Epsilon^\alpha_A$ and $\eta^\alpha_A$), and define the 2-cells 
$(\alpha_1)_u$ and $\delta_{\alpha_1,f}$ as in \prref{vertic->horiz}. The following identities then follow: 
$$
\bfig
\putmorphism(-150,500)(1,0)[F(A)`F(B)`F(f)]{600}1a
 \putmorphism(450,500)(1,0)[\phantom{F(A)}`G(B) `\alpha_1(B)]{640}1a

 \putmorphism(-150,50)(1,0)[G(A)`G(B)`G(f)]{600}1a
 \putmorphism(450,50)(1,0)[\phantom{F(A)}`G(B) `=]{640}1a

\putmorphism(-180,500)(0,-1)[\phantom{Y_2}``\alpha_0(A)]{450}1l
\putmorphism(450,500)(0,-1)[\phantom{Y_2}``]{450}1r
\putmorphism(300,500)(0,-1)[\phantom{Y_2}``\alpha_0(B)]{450}0r
\putmorphism(1100,500)(0,-1)[\phantom{Y_2}``=]{450}1r
\put(0,260){\fbox{$(\alpha_0)_f$}}
\put(700,270){\fbox{$\Epsilon^\alpha_B$}}
\efig
\quad=\quad
\bfig
\putmorphism(-150,500)(1,0)[F(A)`F(B)`F(f)]{600}1a
 \putmorphism(450,500)(1,0)[\phantom{F(A)}`G(B) `\alpha_1(B)]{680}1a
 \putmorphism(-150,50)(1,0)[F(A)`G(A)`\alpha_1(A)]{600}1a
 \putmorphism(450,50)(1,0)[\phantom{F(A)}`G(B) `G(f)]{680}1a

\putmorphism(-180,500)(0,-1)[\phantom{Y_2}``=]{450}1r
\putmorphism(1100,500)(0,-1)[\phantom{Y_2}``=]{450}1r
\put(350,260){\fbox{$\delta_{\alpha_1,f}$}}
\putmorphism(-150,-400)(1,0)[F(A')`G(A'). `=]{640}1a

\putmorphism(-180,50)(0,-1)[\phantom{Y_2}``\alpha_0(A)]{450}1l
\putmorphism(450,50)(0,-1)[\phantom{Y_2}``]{450}1l
\putmorphism(610,50)(0,-1)[\phantom{Y_2}``=]{450}0l 
\put(20,-180){\fbox{$\Epsilon^\alpha_A$}} 
\efig
$$
and
$$
\bfig
\putmorphism(-150,0)(1,0)[F(A)`F(A)`=]{600}1a
 \putmorphism(450,0)(1,0)[\phantom{F(A)}`G(B) `F(f)]{640}1a

 \putmorphism(-150,-450)(1,0)[F(A)`G(A)`\alpha_1(A)]{600}1a
 \putmorphism(450,-450)(1,0)[\phantom{F(A)}`G(B) `G(f)]{640}1a

\putmorphism(-180,0)(0,-1)[\phantom{Y_2}``=]{450}1l
\putmorphism(450,0)(0,-1)[\phantom{Y_2}``]{450}1r
\putmorphism(300,0)(0,-1)[\phantom{Y_2}``\alpha_0(A)]{450}0r
\putmorphism(1100,0)(0,-1)[\phantom{Y_2}``\alpha_0(B)]{450}1r
\put(50,-200){\fbox{$\eta^\alpha_A$}}
\put(650,-200){\fbox{$(\alpha_0)_f$}}
\efig
\quad=\quad
\bfig
\putmorphism(-150,0)(1,0)[F(A)`F(B)`F(f)]{600}1a
 \putmorphism(450,0)(1,0)[\phantom{F(A)}`G(B) `\alpha_1(B)]{680}1a
 \putmorphism(-150,-450)(1,0)[F(A)`G(A)`\alpha_1(A)]{600}1a
 \putmorphism(450,-450)(1,0)[\phantom{F(A)}`G(B) `G(f)]{680}1a

\putmorphism(-180,0)(0,-1)[\phantom{Y_2}``=]{450}1r
\putmorphism(1100,0)(0,-1)[\phantom{Y_2}``=]{450}1r
\put(350,-240){\fbox{$\delta_{\alpha_1,f}$}}
\putmorphism(450,450)(1,0)[F(B)`F(B) `=]{640}1a

\putmorphism(450,450)(0,-1)[\phantom{Y_2}``=]{450}1l
\putmorphism(1100,450)(0,-1)[\phantom{Y_2}``\alpha_0(B)]{450}1l 
\put(600,250){\fbox{$\eta^\alpha_B$}} 
\efig
$$
for every 1h-cell $f:A\to B$; 
$$
\bfig
 \putmorphism(-150,500)(1,0)[F(A)`G(A)  `\alpha_1(A)]{600}1a
\putmorphism(-180,500)(0,-1)[\phantom{Y_2}`F(A') `F(u)]{450}1l
\putmorphism(-150,-400)(1,0)[G(A')` `=]{480}1a
\putmorphism(-180,50)(0,-1)[\phantom{Y_2}``\alpha_0(A')]{450}1l
\putmorphism(450,50)(0,-1)[\phantom{Y_2}`G(A')`=]{450}1r
\putmorphism(450,500)(0,-1)[\phantom{Y_2}`G(A') `G(u)]{450}1r
\put(0,260){\fbox{$(\alpha_1)_u$}}
\putmorphism(-180,50)(1,0)[\phantom{F(A)}``\alpha_1(A')]{500}1a
\put(40,-170){\fbox{$\Epsilon^\alpha_{A'}$}}
\efig=
\bfig
 \putmorphism(-150,500)(1,0)[F(A)`F(A) `=]{500}1a
 \putmorphism(450,500)(1,0)[` `\alpha_1(A)]{380}1a
\putmorphism(-180,500)(0,-1)[\phantom{Y_2}`F(A') `F(u)]{450}1l
\put(0,-160){\fbox{$\delta_{\alpha_0,u}$}}
\putmorphism(-150,-400)(1,0)[F(A'')`F(A'') `=]{500}1a
\putmorphism(-180,50)(0,-1)[\phantom{Y_2}``\alpha_0(A')]{450}1l
\putmorphism(380,500)(0,-1)[\phantom{Y_2}` `\alpha_0(A)]{450}1l
\putmorphism(380,50)(0,-1)[G(A)` `G(u)]{450}1r
\put(550,290){\fbox{$\Epsilon^\alpha_A$} }
\putmorphism(940,500)(0,-1)[G(A)`G(A)`=]{450}1r
\putmorphism(470,60)(1,0)[``=]{360}1a
\efig
$$
and 
$$
\bfig
 \putmorphism(-150,500)(1,0)[F(A)`F(A)  `=]{600}1a
\putmorphism(-180,500)(0,-1)[\phantom{Y_2}`F(A) `=]{450}1l
\putmorphism(-150,-400)(1,0)[F(A')` `\alpha_1(A')]{480}1a
\putmorphism(-180,50)(0,-1)[\phantom{Y_2}``F(u)]{450}1l
\putmorphism(450,50)(0,-1)[\phantom{Y_2}`G(A')`G(u)]{450}1r
\putmorphism(450,500)(0,-1)[\phantom{Y_2}`G(A) `\alpha_0(A)]{450}1r
\put(50,300){\fbox{$\eta^\alpha_A$}}
\putmorphism(-180,50)(1,0)[\phantom{F(A)}``\alpha_1(A)]{520}1a
\put(20,-170){\fbox{$(\alpha_1)_u$}}
\efig=
\bfig
 \putmorphism(-400,-400)(1,0)[` `\alpha_1(A')]{380}1a
 \putmorphism(80,500)(1,0)[F(A)`F(A) `=]{500}1a
\putmorphism(-480,50)(0,-1)[F(A')`F(A')`=]{450}1r
\putmorphism(-380,60)(1,0)[``=]{360}1a
\putmorphism(70,500)(0,-1)[\phantom{Y_2}`F(A') `F(u)]{450}1l
\put(220,270){\fbox{$\delta_{\alpha_0,u}$}}
\putmorphism(80,-400)(1,0)[G(A')`G(A') `=]{500}1a
\putmorphism(70,50)(0,-1)[\phantom{Y_2}``\alpha_0(A')]{450}1r
\putmorphism(580,500)(0,-1)[\phantom{Y_2}` `\alpha_0(A)]{450}1r
\putmorphism(580,50)(0,-1)[G(A)` `G(u)]{450}1r
\put(-280,-160){\fbox{$\eta^\alpha_{A'}$} }
\efig
$$
for every 1v-cell $u:A\to A'$. 
\end{cor}

By axioms (6.2) and (6.3) of \cite[Definition 6.1]{Shul1} (which are also valid for pseudo double functors) one has:

\begin{lma}
Let $u_*$ be a 1h-companion of $u$ with the corresponding 2-cells $\eta_u, \Epsilon_u$, and let $F$ be a pseudo double functor. 
Then $F(u)_*=F(u_*)$ where 
$$
\Epsilon_{F(u)}=
\bfig
\putmorphism(-100,250)(1,0)[F(A)`F(A')`F(u_*)]{550}1a

 \putmorphism(-100,-200)(1,0)[F(A')`F(A')`F(id)]{550}1a

\putmorphism(-100,250)(0,-1)[\phantom{Y_2}``]{450}1l
\putmorphism(-70,250)(0,-1)[\phantom{Y_2}``F(u)]{450}0l
\putmorphism(450,250)(0,-1)[\phantom{Y_2}``=]{450}1r
\put(0,40){\fbox{$ F(\Epsilon_u)$}}
\put(70,-430){\fbox{$F_{A'}$}}

\putmorphism(-100,-640)(1,0)[G(A')` `=]{440}1b
\putmorphism(-100,-200)(0,-1)[\phantom{Y_2}``=]{420}1l
\putmorphism(450,-200)(0,-1)[\phantom{Y_2}`G(A')`=]{420}1r

\efig
\qquad\text{and}\qquad
\eta_{F(u)}=
\bfig
 \putmorphism(30,250)(1,0)[F(A)`F(A) `=]{550}1a

 \putmorphism(50,-200)(1,0)[\phantom{F(A)}`F(A), `F(id)]{580}1a

\putmorphism(50,250)(0,-1)[`F(A)`=]{450}1l
\putmorphism(600,250)(0,-1)[\phantom{Y_2}``=]{450}1r
\put(180,-430){\fbox{$F(\eta_u)$}}
\put(180,20){\fbox{$(F^{A})^{-1}$}}

\putmorphism(50,-640)(1,0)[F(A)` `F(u_*)]{440}1b
\putmorphism(50,-200)(0,-1)[\phantom{Y_2}``=]{420}1l
\putmorphism(600,-200)(0,-1)[\phantom{Y_2}`F(A').`F(u)]{420}1r

\efig
$$
\end{lma}

\begin{prop} \prlabel{invertible delta}
Given a strong vertical transformation $\alpha_0$ under conditions of \prref{vertic->horiz}. 
Let $u:A\to A'$ be a 1v-cell with a 1h-companion $f=u_*$. Then the inverse of the 2-cell $\delta_{\alpha_1,u_*}$ is given by: 
$$\delta_{\alpha_1,u_*}^{-1}=
\bfig
\putmorphism(-700,500)(1,0)[F(A)`F(A)`=]{530}1a
\putmorphism(-150,500)(0,-1)[\phantom{Y_2}` `F(u)]{450}1l
\putmorphism(-700,500)(0,-1)[\phantom{Y_2}`F(A) `=]{450}1l
\putmorphism(-610,50)(1,0)[``F(u)_*]{360}1b
\put(-630,260){\fbox{$\eta_{F(u)}$}}

 \putmorphism(-40,50)(1,0)[``=]{320}1b

\put(550,30){\fbox{$\delta_{\alpha_0,u}$}}
\putmorphism(-150,-400)(1,0)[F(A')`G(A') `\alpha_1(A')]{520}1a
\putmorphism(-150,50)(0,-1)[F(A')``=]{450}1l
\putmorphism(380,500)(0,-1)[\phantom{Y_2}` `F(u)]{450}1l
\putmorphism(950,500)(0,-1)[\phantom{Y_2}`G(A) `\alpha_0(A)]{450}1r
\putmorphism(380,50)(0,-1)[F(A')` `\alpha_0(A')]{450}1r
\putmorphism(350,500)(1,0)[F(A)`F(A)`=]{600}1a
 \putmorphism(950,500)(1,0)[\phantom{F(A)}`G(A) `\alpha_1(A)]{580}1a
 \putmorphism(1060,50)(1,0)[`G(A)`=]{500}1b

\putmorphism(1520,500)(0,-1)[\phantom{Y_2}``=]{450}1r
\put(1250,260){\fbox{$\Epsilon^\alpha_A$}}
\putmorphism(480,-400)(1,0)[`G(A') `=]{500}1a
\putmorphism(950,50)(0,-1)[\phantom{Y_2}``G(u)]{450}1r
\put(0,-160){\fbox{$\eta^\alpha_{A'}$}}

\putmorphism(1640,50)(1,0)[`G(A')`G(u)_*]{440}1a
\putmorphism(2070,50)(0,-1)[\phantom{Y_2}``=]{450}1r
\putmorphism(1520,50)(0,-1)[\phantom{Y_2}`G(A')`G(u)]{450}1l
\putmorphism(1630,-400)(1,0)[`G(A'). `=]{440}1a
\put(1660,-160){\fbox{$\Epsilon_{G(u)}$}}

\efig
$$
\end{prop}

\begin{proof}
Use axiom 1) for $\alpha_0$ and (6.3) of \cite[Definition 6.1]{Shul1}, together with $\Epsilon\x\eta$-relations.  
\end{proof}

\medskip

Note that the above inverse of $\delta_{\alpha_1,u_*}$ is in fact the image of the known pseudofunctor $V\Dd\to H\Dd$ from the 
vertical 2-category to the horizontal one of a given double category $\Dd$ in which all 1v-cells have 1h-companions. For a horizontally 
globular 2-cell $a$ with a left 1v-cell $u$ and 
a right 1v-cell $v$, the image by this functor of $a$ is given by $[\eta_u \vert a\vert\Epsilon_v]$ (horizontal composition of 2-cells).

\begin{rem}
One could start with a vertical transformation $\alpha_0$ (for which $\delta_{\alpha_0,u}=\Id$ for all 1v-cells $u:A\to A'$) 
and define a horizontal transformation $\alpha_1$ setting $\delta_{\alpha_1,f}=\Id$ for all 1h-cells $f:A\to B$ and 
defining $(\alpha_1)_u$  as in \prref{vertic->horiz}. Though, in order for $\alpha_1$ to satisfy the corresponding axiom 1), 
one needs to assume the first two identities of \coref{4 identities e-eta}. 
\end{rem}

\subsection{Embedding $PsDbl^*_3$ into $DblPs$}

It remains to show how to turn vertical transformations into double pseudonatural transformations. We will assume that 
1v-cell components of vertical transformations have 1h-companions. 

Observe that for vertical transformations the 2-cells $\delta_{\alpha_0,u}$ are identities. Moreover,  
we know that vertical transformations are particular cases of vertical pseudonatural transformations (\rmref{hor tr as subcase}), 
and that strong horizontal transformations are particular cases of horizontal pseudonatural transformations. 
By \prref{vertic->horiz} we have that a vertical transformation $\alpha_0$ determines a strong horizontal transformation. 
So far we have axiom (T1) of \deref{double 2-cells}. 
Furthermore, by \prref{invertible delta} we have in particular that for all 1v-cell components $\alpha_0(A)$ of vertical transformations 
the 2-cells $\delta_{\alpha_1,\alpha_1(A)}$ are invertible. Then we have that the axiom (T2) is fulfilled.  
Observe that setting $\Theta^\alpha_A=\Epsilon^\alpha_A$, by the first and third identities in \coref{4 identities e-eta} 
we have a $\Theta$-double pseudonatural transformation between pseudo double functors. Due to \prref{teta->double} we have 
indeed a double pseudonatural transformation, as we wanted. (Actually, thanks to the $\Epsilon\x\eta$-relations, by the first identity in 
\coref{4 identities e-eta}, axiom 1) for the horizontal pseudonatural transformation $\alpha_1$ holds 
if and only if axiom (T3-1) for $t^\alpha_f$ in \prref{teta->double} holds.)

Moreover, we may deduce the following bijective correspondence $t^\alpha_f\leftrightarrow\delta_{\alpha_1,f}$:
$$
t^\alpha_f=
\bfig
\putmorphism(-150,500)(1,0)[F(A)`F(B)`F(f)]{600}1a
 \putmorphism(450,500)(1,0)[\phantom{F(A)}`G(B) `\alpha_1(B)]{680}1a
 \putmorphism(-150,50)(1,0)[F(A)`G(A)`\alpha_1(A)]{600}1a
 \putmorphism(450,50)(1,0)[\phantom{F(A)}`G(B) `G(f)]{680}1a

\putmorphism(-180,500)(0,-1)[\phantom{Y_2}``=]{450}1r
\putmorphism(1100,500)(0,-1)[\phantom{Y_2}``=]{450}1r
\put(350,260){\fbox{$\delta_{\alpha_1,f}$}}
\putmorphism(-150,-400)(1,0)[F(A')`G(A') `=]{640}1a

\putmorphism(-180,50)(0,-1)[\phantom{Y_2}``\alpha_0(A)]{450}1l
\putmorphism(450,50)(0,-1)[\phantom{Y_2}``]{450}1l
\putmorphism(610,50)(0,-1)[\phantom{Y_2}``=]{450}0l 
\put(20,-180){\fbox{$\Epsilon^\alpha_A$}} 
\efig
$$

$$\delta_{\alpha_1,f}=\quad
\bfig
 \putmorphism(-150,250)(1,0)[F(A)`F(A) `=]{500}1a
\put(0,30){\fbox{$\eta^\alpha_A$}}
\putmorphism(380,250)(0,-1)[\phantom{Y_2}` `\alpha_0(A)]{450}1r
\putmorphism(350,250)(1,0)[F(A)`F(B)`F(f)]{600}1a
 \putmorphism(950,250)(1,0)[\phantom{F(A)}`G(B) `\alpha_1(B)]{600}1a
 \putmorphism(470,-200)(1,0)[`G(B)`G(f)]{500}1b
 \putmorphism(1060,-200)(1,0)[`G(B),`=]{500}1b
\putmorphism(1570,250)(0,-1)[\phantom{Y_2}``=]{450}1r
\put(900,30){\fbox{$t^\alpha_f$}}
\putmorphism(-150,-200)(1,0)[F(A)`G(A) `\alpha_1(A)]{520}1a
\putmorphism(-150,250)(0,-1)[``=]{450}1l
\efig
$$
and complete the bijection $t^\alpha_f\leftrightarrow(\alpha_0)_f$:
$$(\alpha_0)_f=
\quad
\bfig
\putmorphism(-150,0)(1,0)[F(A)`F(B)`F(f)]{600}1a
 \putmorphism(450,0)(1,0)[\phantom{F(A)}`G(B) `\alpha_1(B)]{680}1a
 \putmorphism(-150,-450)(1,0)[F(A)`G(A).`G(f)]{1280}1a

\putmorphism(-180,0)(0,-1)[\phantom{Y_2}``=]{450}1r
\putmorphism(1100,0)(0,-1)[\phantom{Y_2}``=]{450}1r
\put(350,-200){\fbox{$t^\alpha_f$}}
\putmorphism(450,450)(1,0)[F(B)`F(B) `=]{640}1a

\putmorphism(450,450)(0,-1)[\phantom{Y_2}``=]{450}1l
\putmorphism(1100,450)(0,-1)[\phantom{Y_2}``\alpha_0(B)]{450}1r
\put(680,250){\fbox{$\eta^\alpha_B$}} 

\efig
$$

\begin{rem}
Given the $\Epsilon\x\eta$-relations, by the properties developed in this and the previous Subsection, 
axiom 1) for the horizontal pseudonatural transformation $\alpha_1$ holds if and only if 
axiom (T3-1) for $t^\alpha_f$ in \prref{teta->double} holds, if and only if 
axiom 1) for the vertical pseudonatural transformation $\alpha_0$ holds. 
\end{rem}

\bigskip

\bigskip

{\bf Acknowledgement.}
I am profoundly thankful to Gabi B\"ohm for helping me 
understand the problem of 
non-fitting of monoids in her monoidal category $Dbl$ into intercategories, for suggesting me to try her $Dbl$ 
as the codomain for the embedding in Section 3, and for many other richly nurturing discussions. 
This research was partly developed during my sabbatical year from the Instituto de Matem\'atica Rafael Laguardia of the 
Facultad de Ingeneir\'ia of the Universidad de la Rep\'ublica in Montevideo (Uruguay). My thanks to ANII and PEDECIBA Uruguay 
for financial support. The work was also supported by the Serbian Ministry of Education, Science and
Technological Development through Mathematical Institute of the Serbian Academy of Sciences and Arts. 

\end{document}